\documentclass[reqno]{amsart}
\usepackage{amssymb,eucal,latexsym, mathrsfs,enumerate,graphicx,url,tikz,tikz-cd}

\tikzset{every picture/.style={line width=0.7pt}}

\newenvironment{enumeratei}{\begin{enumerate}[\upshape (i)]}%
{\end{enumerate}}
{\end{enumerate}}
\newenvironment{enumerater}{\begin{enumerate}[\upshape (1)]}%
{\end{enumerate}}

\newcommand{\cm}{commutative monoid}
\newcommand{\Vhom}{V-ho\-mo\-mor\-phism}
\newcommand{\lex}{\times_{\mathrm{lex}}}
\newcommand{\bt}{\boxtimes}

\newcommand{\todot}{\overset{\boldsymbol{.}}{\rightarrow}}
\newcommand{\Todot}{\overset{\boldsymbol{.}\ }{\Rightarrow}}
\newcommand{\utr}{\unlhd}
\newcommand{\utrr}{\unlhd_{\mathrm{r}}}

\newcommand{\aspr}{\asymp_{\mathrm{r}}}

\newcommand{\otmF}{\otimes^{\mu}_{\Phi}}
\newcommand{\otlF}{\otimes^{\gl}_{\Phi}}

\newcommand{\Pow}{\mathfrak{P}}
\newcommand{\Powi}{\mathfrak{P}_{\mathrm{inj}}}
\newcommand{\pup}[1]{\textup{(}{#1}\textup{)}}

\newcommand{\lgrp}{$\ell$-group}

\newcommand{\lhom}{$\ell$-ho\-mo\-mor\-phism}

\newcommand{\eqdef}{\overset{\mathrm{def}}{=}}

\newcommand{\rF}{\mathrm{F}}
\newcommand{\xF}{\mathbf{F}}
\newcommand{\Metr}{\mathbf{Metr}}
\newcommand{\Pres}{\mathbf{Pres}}

\newcommand{\Mat}[2]{\operatorname{M}_{{#1}}({#2})}
\newcommand{\IM}[2]{\operatorname{I}_{{#1}}({#2})}

\DeclareMathOperator{\cf}{cf}
\DeclareMathOperator{\rV}{V}
\DeclareMathOperator{\Idp}{Idp}
\DeclareMathOperator{\Subc}{Sub_c}

\DeclareMathOperator{\At}{At}
\DeclareMathOperator{\Ult}{Ult}
\newcommand{\Ultb}{\mathrm{Ult}^{\flat}}

\newcommand{\LL}{\mathbb{L}}

\DeclareMathOperator{\card}{card}

\DeclareMathOperator{\Ids}{Id_{s}}

\DeclareMathOperator{\Max}{Max}
\DeclareMathOperator{\Min}{Min}

\newcommand{\ga}{\alpha}
\newcommand{\gb}{\beta}
\newcommand{\gc}{\gamma}
\newcommand{\gd}{\delta}
\newcommand{\gf}{\varphi}
\newcommand{\gk}{\kappa}
\newcommand{\gl}{\lambda}
\newcommand{\gq}{\theta}
\newcommand{\gs}{\sigma}
\newcommand{\go}{\omega}
\newcommand{\eps}{\varepsilon}

\newcommand{\bga}{\boldsymbol{\alpha}}

\newcommand{\bxi}{\boldsymbol{\xi}}

\newcommand{\gC}{\Gamma}
\newcommand{\gD}{\Delta}
\newcommand{\gL}{\Lambda}
\newcommand{\gO}{\Omega}
\newcommand{\gS}{\Sigma}

\newcommand{\sd}{\mathbin{\smallsetminus}}

\newcommand{\jz}{$(\vee,0)$}
\newcommand{\js}{join-semi\-lat\-tice}

\newcommand{\jzs}{\jz-semi\-lat\-tice}
\newcommand{\jzh}{\jz-ho\-mo\-mor\-phism}
\newcommand{\ajs}{al\-most join-sem\-i\-lat\-tice}
\newcommand{\pjs}{pseu\-do join-sem\-i\-lat\-tice}

\newcommand{\two}{\mathbf{2}}

\newcommand{\ol}[1]{\overline{#1}}

\newcommand{\pI}[1]{\bigl({#1}\bigr)}
\newcommand{\pII}[1]{\Bigl({#1}\Bigr)}

\newcommand{\rI}[1]{\bigl[{#1}\bigr]}

\newcommand{\set}[1]{\{#1\}}
\newcommand{\setm}[2]{\set{{#1}\mid{#2}}}
\newcommand{\vecm}[2]{({#1}\mid{#2})}
\newcommand{\Vecm}[2]{\left({#1}\mid{#2}\right)}
\newcommand{\seq}[1]{\langle{#1}\rangle}

\newcommand{\seql}[1]{{\langle{#1}\rangle}^{\ell}}

\newcommand{\id}{\mathrm{id}}
\newcommand{\dnw}{\mathbin{\downarrow}}

\newcommand{\upw}{\mathbin{\uparrow}}

\newcommand{\sor}{\mathbin{\triangledown}}
\newcommand{\Sor}{\mathbin{\bigtriangledown}}
\newcommand{\es}{\varnothing}
\newcommand{\res}{\mathbin{\restriction}}

\newcommand{\CC}{\mathbb{C}}

\newcommand{\ZZ}{\mathbb{Z}}

\newcommand{\lGrp}{\boldsymbol{\ell}\mathbf{Grp}}
\newcommand{\vNR}{\mathbf{vNRing}}
\newcommand{\Ring}{\mathbf{Ring}}
\newcommand{\DLatz}{\mathbf{DLat_0}}
\newcommand{\SLatz}{\mathbf{SLat_0}}
\newcommand{\Cev}{\mathbf{Cev}}

\newcommand{\Str}{\operatorname{\mathbf{Str}}}
\newcommand{\Bool}{\mathbf{Bool}}
\newcommand{\Boolc}{\mathbf{Bool}^{\mathrm{cstr}}}

\DeclareMathOperator{\Conc}{Con_c}
\DeclareMathOperator{\Id}{Id}
\DeclareMathOperator{\Idc}{Id_c}
\DeclareMathOperator{\Cs}{Cs}
\DeclareMathOperator{\Csc}{Cs_c}
\DeclareMathOperator{\dom}{dom}
\DeclareMathOperator{\rng}{rng}

\newcommand{\fa}{{\mathfrak{a}}}
\newcommand{\fb}{{\mathfrak{b}}}
\newcommand{\fc}{{\mathfrak{c}}}

\newcommand{\cA}{{\mathcal{A}}}
\newcommand{\cB}{{\mathcal{B}}}
\newcommand{\cC}{{\mathcal{C}}}

\newcommand{\cG}{{\mathcal{G}}}

\newcommand{\cI}{{\mathcal{I}}}

\newcommand{\cP}{{\mathcal{P}}}
\newcommand{\cR}{{\mathcal{R}}}
\newcommand{\cS}{{\mathcal{S}}}
\newcommand{\cT}{{\mathcal{T}}}
\newcommand{\cU}{{\mathcal{U}}}
\newcommand{\cV}{{\mathcal{V}}}
\newcommand{\cX}{{\mathcal{X}}}

\numberwithin{equation}{section}

\newtheorem*{stat}{\name}
\newcommand{\name}{testing}

\newenvironment{all}[1]{\renewcommand{\name}{#1}\begin{stat}}
                        {\end{stat}}

\theoremstyle{plain}

\newtheorem{theorem}{Theorem}[section]
\newtheorem{proposition}[theorem]{Proposition}
\newtheorem{corollary}[theorem]{Corollary}
\newtheorem{lemma}[theorem]{Lemma}

\newtheorem{claim}{Claim}
\newtheorem*{sclaim}{Claim}

\theoremstyle{definition}

\newtheorem{definition}[theorem]{Definition}
\newtheorem{notation}[theorem]{Notation}

\theoremstyle{remark}
\newtheorem{remark}[theorem]{Remark}
\newtheorem*{note}{Note}

\newcommand{\qedc}{{\qed}~{\rm Claim~{\theclaim}.}}
\newcommand{\qedsc}{{\qed}~{\rm Claim.}}

\newenvironment{cproof}
{\begin{proof}[Proof of Claim.]}
{\qedc\renewcommand{\qed}{}\end{proof}}

\newenvironment{scproof}
{\begin{proof}[Proof of Claim.]}
{\qedsc\renewcommand{\qed}{}\end{proof}}

\numberwithin{figure}{section}
\numberwithin{table}{section}

\newcommand{\ba}{\boldsymbol{a}}
\newcommand{\bb}{\boldsymbol{b}}

\newcommand{\be}{\boldsymbol{e}}

\newcommand{\bp}{\boldsymbol{p}}

\newcommand{\bs}{\boldsymbol{s}}

\newcommand{\bx}{\boldsymbol{x}}
\newcommand{\by}{\boldsymbol{y}}

\newcommand{\bA}{\boldsymbol{A}}
\newcommand{\bB}{\boldsymbol{B}}
\newcommand{\bC}{\boldsymbol{C}}

\newcommand{\bU}{\boldsymbol{U}}

\newcommand{\bX}{\boldsymbol{X}}

\newcommand{\vx}{\mathsf{x}}
\newcommand{\vy}{\mathsf{y}}

\newcommand{\sF}{\mathsf{F}}
\newcommand{\sG}{\mathsf{G}}

\newcommand{\scL}{\mathbin{\mathscr{L}}}

\newcommand{\us}[1]{\underset{{}^{\smile}}{#1}}

\title[Anti-elementary classes]%
{{}From non-commutative diagrams to\\
anti-elementary classes}

\author[F. Wehrung]{Friedrich Wehrung}
\address{LMNO, CNRS UMR 6139\\
D\'epartement de Math\'ematiques\\
Universit\'e de Caen Normandie\\
14032 Caen cedex\\
France}
\email{friedrich.wehrung01@unicaen.fr}
\urladdr{https://wehrungf.users.lmno.cnrs.fr}

\date{\today}



\begin{document}

\begin{abstract}
\emph{Anti-el\-e\-men\-tar\-ity} is a strong way of ensuring that a class of structures, in a given first-order language, is not closed under elementary equivalence with respect to any infinitary language of the form~$\scL_{\infty\gl}$\,.
We prove that many naturally defined classes are anti-el\-e\-men\-tary, including the following:
\begin{itemize}
\item
the class of all lattices of finitely generated convex $\ell$-sub\-groups of members of any class of \lgrp{s} containing all Archimedean \lgrp{s};

\item
the class of all semilattices of finitely generated $\ell$-ideals of members of any nontrivial quasivariety of \lgrp{s};

\item
the class of all Stone duals of spectra of MV-algebras --- this yields a negative solution to the \emph{MV-spectrum Problem};

\item
the class of all semilattices of finitely generated two-sided ideals of rings;

\item
the class of all semilattices of finitely generated submodules of modules;

\item
the class of all monoids encoding the nonstable K${}_0$-theory of von Neumann regular rings, respectively C*-algebras of real rank zero;

\item
\pup{assuming arbitrarily large Erd\H{o}s cardinals} the class of all coordinatizable sectionally complemented modular lattices with a large $4$-frame.

\end{itemize}
The main underlying principle is that under quite general conditions, for a functor $\Phi\colon\cA\to\cB$, if there exists a non-com\-mu\-ta\-tive diagram~$\vec{D}$ of~$\cA$, indexed by a common sort of poset called an \emph{\ajs}, such that
\begin{itemize}
\item
$\Phi\vec{D}^I$ is a com\-mu\-ta\-tive diagram for every set~$I$,

\item
$\Phi\vec{D}\not\cong\Phi\vec{X}$ for any com\-mu\-ta\-tive diagram~$\vec{X}$ in~$\cA$,
\end{itemize}
then the range of~$\Phi$ is anti-el\-e\-men\-tary.
\end{abstract}

\maketitle

\tableofcontents

\section{Introduction}\label{S:Intro}
The present paper is an extension of a negative solution of Problem~2 in Mundici's monograph~\cite{Mund2011}, sometimes called the \emph{MV-spectrum Problem}, stated as ``which topological spaces are homeomorphic to the spectrum of some MV-algebra?''.
Due to the generality of the methods involved, much more came out than expected.

Formalizing the MV-spectrum problem in terms of the Stone duals of the topological spaces in question, and involving the categorical equivalence between MV-algebras and Abelian \lgrp{s} with order-unit established in Mundici~\cite{Mund86}, the MV-spectrum Problem can be recast in terms of the classical problem, dating back to the seventies, of characterizing the lattices of principal $\ell$-ideals of Abelian \lgrp{s}.
Although the author characterized the \emph{countable} such lattices by a first-order statement called \emph{complete normality} (cf. Wehrung~\cite{MV1}), it was not even known, in the uncountable case, whether those lattices could be characterized by any class of~$\scL_{\infty\gl}$ sentences of lattice theory, for some large enough cardinal number~$\gl$.

The present paper's results imply that it cannot be so (cf. Corollary~\ref{C:NotMV}).

For any set~$\gO$, denote by~$\Powi(\gO)$ the category whose objects are all subsets of~$\gO$ and whose morphisms are all one-to-one maps $f\colon X\rightarrowtail Y$ where $X,Y\subseteq\gO$, with $g\circ f$ defined if{f} the domain of~$g$ is equal to the codomain of~$f$.

\begin{definition}\label{D:AntiElt}
Let~$\cC_0$ and~$\cC_1$ be classes of objects in a category~$\cS$. 
The pair $(\cC_0,\cC_1)$ is \emph{anti-el\-e\-men\-tary} if for any cardinal~$\gq$ there are infinite cardinals~$\gl$ and~$\gk$, with~$\gl$ regular and $\gq\leq\gl<\gk$, together with a functor $\gC\colon\Powi(\gk)\to\cS$ preserving all $\gl$-directed colimits, such that $\gC(\gl)\in\cC_0$ and $\gC(\gk)\notin\nobreak\cC_1$\,.
If $\cC_0=\cC_1$ we then say that~$\cC_0$ is anti-el\-e\-men\-tary.
\end{definition}

If $(\cC_0,\cC_1)$ is anti-el\-e\-men\-tary, then every class~$\cC$ such that $\cC_0\subseteq\cC\subseteq\cC_1$ is anti-el\-e\-men\-tary (obviously this is meaningful only in case $\cC_0\subseteq\cC_1$).
Throughout the paper we will often meet situations where the~$\cC_i$ are subcategories  (as opposed to mere classes of objects) of~$\cS$, in which case anti-el\-e\-men\-tar\-ity will of course be stated on their respective classes of objects.

Let us now relate that concept to elementary equivalence with respect to infinitary languages.
For any first-order language~$\gS$ (with possibly infinite arity), denote by~$\Str{\gS}$ the category of all models for~$\gS$ with $\gS$-homomorphisms.
By an extension of Feferman's \cite[Theorem~6]{Fefe1972}, stated in Proposition~\ref{P:EltEq} (established \emph{via} a categorical version of $\gl$-el\-e\-men\-tar\-ity that we call \emph{$\gl$-freshness}), an anti-el\-e\-men\-tary class in~$\Str\gS$ cannot be closed under $\scL_{\infty\gl}$-el\-e\-men\-tary equivalence, for any large enough infinite regular cardinal~$\gl$.
In particular, it is not the class of models of any class of~$\scL_{\infty\gl}$ sentences.
However, anti-el\-e\-men\-tar\-ity also implies further forms of failure of el\-e\-men\-tar\-ity.
For example, if~$\gk=\gl^{+2}$ witnesses the anti-el\-e\-men\-tar\-ity of~$\cC$ (this will be the case throughout Sections \ref{S:Ceva}--\ref{S:V(R)}), then the least cardinal~$\mu$ such that $\gC(\mu)\notin\cC$ is either~$\gl^+$ or~$\gl^{+2}$, thus (using Proposition~\ref{P:EltEq}) $\vecm{\gC(\xi)}{\gl\leq\xi<\mu}$ is an $\scL_{\infty\gl}$-el\-e\-men\-tary chain in~$\cC$, of length the regular cardinal $\mu>\gl$, whose union~$\gC(\mu)$ does not belong to~$\cC$.

The bulk of the present paper is devoted to developing techniques enabling us to prove anti-el\-e\-men\-tar\-ity for many classes:
\begin{itemize}
\item
The part ``$\gC(\gl)\in\cC_0$'' is taken care of by two new results, which we call the Uniformization Lemma (Lemma~\ref{L:UnifLem}) and the Boosting Lemma (Lemma~\ref{L:Boosting}).

\item
The part ``$\gC(\gk)\notin\cC_1$'' is taken care of by Lemmas~\ref{L:ExtArm} and~\ref{L:ExtCLL}, which are extensions, to the case of non-com\-mu\-ta\-tive diagrams, of the original Armature Lemma and Condensate Lifting Lemma (CLL) established in Gillibert and Wehrung~\cite{Larder}.

\item
Lemmas~\ref{L:ExtArm} and~\ref{L:ExtCLL} entail a collection of results of which the underlying principle can, loosely speaking, be paraphrased as follows.
\begin{quote}\em
Under quite general conditions, for a functor $\Phi\colon\cA\to\cB$, if there exists a poset-indexed \pup{necessarily non-com\-mu\-ta\-tive} diagram~$\vec{D}$ of~$\cA$ such that
\begin{itemize}
\item
$\Phi\vec{D}^I$ is a com\-mu\-ta\-tive diagram for every set~$I$,

\item
$\Phi\vec{D}\not\cong\Phi\vec{X}$ for any com\-mu\-ta\-tive diagram~$\vec{X}$ in~$\cA$,
\end{itemize}
then the range of~$\Phi$ is anti-el\-e\-men\-tary.
\end{quote}
\end{itemize}
The class of posets~$P$ indexing the diagrams~$\vec{D}$ above needs to be restricted to the \emph{\ajs{s} with zero} introduced in Gillibert and Wehrung~\cite{Larder} (cf. Section~\ref{S:NormCov}).
Every \js\ is an \ajs.
In Sections~\ref{S:Ceva}--\ref{S:V(R)}, $P$ is the powerset lattice of a three-element chain, while in Section~\ref{S:4SCML} it is a lattice with zero, of cardinality~$\aleph_1$\,, denoted there by~$Q$.
In Gillibert and Wehrung \cite[Ch.~5]{Larder}, condensate-type constructions, used there to establish non-rep\-re\-sentabil\-ity results, are applied with~$P$ the powerset lattice of a two-element chain.
In all currently known applications of condensates, $P$ is a well-founded lattice.

Corollary~\ref{C:NotMV} states that the class of all Stone duals of MV-spectra is anti-el\-e\-men\-tary; in particular, it cannot be the class of all models of any class of~$\scL_{\infty\gl}$ sentences.

The main problem consists of determining the functor~$\gC$ witnessing anti-el\-e\-men\-tar\-ity.
This is quite a difficult task, whose framework is mostly categorical and which will take up most of Sections~\ref{S:CatBckgd}--\ref{S:ExtCLL}.
The functor~$\gC$ will arise from constructions called \emph{$\Phi$-condensates}, denoted in the form $\bA\otlF\vec{S}$.
In that notation,
\begin{itemize}
\item
$\bA$ is a \emph{$P$-scaled Boolean algebra} (cf. Section~\ref{S:BoolP}), that is, a Boolean algebra, together with a collection of ideals (subjected to certain conditions) indexed by a poset~$P$;

\item
$\gl$ is an infinite regular cardinal;

\item
$\Phi$ is a functor from a category~$\cS$ to a category~$\cT$;

\item
$\vec{S}$ is a $P$-indexed diagram in~$\cS$.
In contrast to the general situation in Gillibert and Wehrung~\cite{Larder}, the diagram~$\vec{S}$ will only satisfy a weak form of commutativity called \emph{$\Phi$-com\-mu\-ta\-tiv\-ity} (cf. Definition~\ref{D:PhiComm}).

\end{itemize}

The object $\bA\otlF\vec{S}$ is first defined over $\gl$-pre\-sentable $P$-scaled Boolean algebras~$\bA$ (under additional assumptions on~$P$ if~$\gl>\go$);
in that case, $\bA\otlF\vec{S}=\Phi(\bA\bt\vec{S})$, where $\bA\bt\vec{S}$ is a new construct (cf. Definition~\ref{D:gfotlvecS}), agreeing with the $\bA\otimes\vec{S}$ from Gillibert and Wehrung~\cite{Larder} if $\gl=\go$ and~$\vec{S}$ is a commutative diagram.

An important difference between ${}_{-}\bt\vec{S}$ and ${}_{-}\otimes\vec{S}$ is that, due to the non-com\-mu\-ta\-tiv\-ity of the diagram~$\vec{S}$, the former no longer sends morphisms to morphisms: \emph{for a morphism~$\gf$, $\gf\bt\vec{S}$ is now a nonempty set of morphisms} (as opposed to a single morphism).
This is where the abovementioned $\Phi$-com\-mu\-ta\-tiv\-ity of~$\vec{S}$ comes into play: if that assumption holds, then the value taken by the functor~$\Phi$ on all morphisms in $\gf\bt\vec{S}$ is constant, and then naturally defined as $\gf\otlF\vec{S}$.

In the general case, $\bA\otlF\vec{S}$ is then defined as the $\gl$-directed colimit of the $\bU\otlF\vec{S}$, formed over $\gl$-pre\-sentable substructures~$\bU$ of~$\bA$ (cf. Definition~\ref{D:AotmFS}).

The combinatorial background for our results mostly rests on the concept of \emph{lifter} introduced in Gillibert and Wehrung~\cite{Larder}, which itself rests on infinitary Ramsey-type statements arising from the classical relation $(\gk,r,\gl)\rightarrow\rho$ (cf. Erd\H{o}s \emph{et al.}~\cite{EHMR}).
The lifters~$X$ in question are involved in the construction of certain $P$-scaled Boolean algebras, introduced in~\cite{Larder} and denoted as~$\xF(X)$ (cf. Section~\ref{S:NormCov}).
As in~\cite{Larder}, those lifters will be crucial in both formulations of the extended Armature Lemma and CLL (cf. Lemmas~\ref{L:ExtArm} and~\ref{L:ExtCLL}).

The last layer of our construction, of the functor~$\gC$, will consist of making the $\gl$-lifter~$X$ sufficiently functorial in a cardinal~$\gk>\gl$ defined by a certain infinite combinatorial property.
This is where \emph{standard lifters} will come into play (cf. Definition~\ref{D:StandLift}), enabling us to write the desired lifters in the form~$P\seq{\gk}$.

The outline above can then be condensed into the single formula
 \begin{equation*}
 \gC(U)=\xF(P\seq{U})\otlF\vec{S}\,,\quad\text{for every }
 U\subseteq\gk\,.
 \end{equation*}
 
Due to the generality of the underlying framework, our results will enable us to establish the anti-el\-e\-men\-tar\-ity of many further classes of algebraic structures.
A sample of such classes runs as follows:

\begin{enumerater}\em
\item\label{AntiCsc}
\pup{Corollary~\textup{\ref{C:NotCev.5}}}
The class of all lattices of finitely generated convex subgroups of the members of any class of \lgrp{s} containing all Ar\-chi\-me\-dean \lgrp{s}.

\item\pup{Corollary~\textup{\ref{C:NonIdc}}}
The class of all semilattices of finitely generated $\ell$-ideals of any class of \lgrp{s} which is closed under products and colimits indexed by all large enough regular cardinals and contains all Ar\-chi\-me\-dean \lgrp{s}.

\item\pup{Corollary~\textup{\ref{C:NonIdcvNR}}}
The class of all semilattices of finitely generated two-sided ideals of the members of any class of unital rings, containing all matrix algebras over a given field, and closed under products and all directed colimits.

\item\pup{Corollary~\textup{\ref{C:NonIdcvNR}}}
The class of all semilattices of finitely generated submodules of right modules.

\item\pup{Corollaries~\textup{\ref{C:NonVvNR}} and~\textup{\ref{C:NonVRR0}}}
The class of all monoids encoding the nonstable K${}_0$-theory of  unital von~Neumann regular rings \pup{resp., unital C*-algebras of real rank zero}.
\emph{(The problem for von~Neumann regular rings was initiated in Goodearl~\cite{Good95}.)}

\item\label{AntiCoord}
\pup{Corollary~\textup{\ref{C:LLAntiElt}}, assuming the existence of arbitrarily large Erd\H{o}s cardinals}
The class of all coordinatizable sectionally complemented modular lattices with a large $4$-frame.
\emph{(The problem was initiated in J\'onsson~\cite{Jons1962}; see Wehrung~\cite{NonCoord}.)}
\end{enumerater}

Define a \emph{projective class} (or~\emph{PC}) within~$\scL_{\gk\gl}$\,, on a first-order language~$\gS$, as the class of $\gS$-reducts of the class of all models of some~$\scL_{\gk\gl}$ sentence in a language containing~$\gS$.
It is worth observing that the classes of algebras considered in \eqref{AntiCsc}--\eqref{AntiCoord} above, which we prove to be anti-el\-e\-men\-tary, are usually PCs within~$\scL_{\infty\go}$\,, thus lending some optimality to our anti-el\-e\-men\-tar\-ity results.
For example, the class~$\cC$, of all lattices of the form $\Idc{G}\eqdef$ lattice of all principal $\ell$-ideals of~$G$, for some Abelian \lgrp~$G$, is a PC within~$\scL_{\go_1\go}$\,.
To see this, observe that if an infinite lattice~$D$ belongs to~$\cC$, then it is isomorphic to~$\Idc{G}$ for some Abelian \lgrp~$G$ with the same cardinality as~$D$;
in other words, there are an Abelian \lgrp\ structure~$G$ on~$D$ and a surjective map $f\colon G^+\twoheadrightarrow D$ such that
 \[
 f(x)\leq_D f(y)\quad\Longleftrightarrow\quad\pII{x\leq_Gy
 \underbrace{+_G\cdots+_G}_{n\text{ times}}y\quad
 \text{for some positive integer }n}
 \]
whenever $x,y\in G^+$.
This is PC within~$\scL_{\go_1\go}$\,.
Now the remaining collection of finite lattices is easily taken care of by an~$\scL_{\go_1\go}$ sentence.

All our anti-el\-e\-men\-tar\-ity results, established in Sections~\ref{S:Ceva}--\ref{S:4SCML}, are obtained by combining our new techniques ($\Phi$-condensates) with earlier non-rep\-re\-sentabil\-ity results on the functors in question.

\section{Notation and terminology}\label{S:Basic}

\subsection{Category theory}
The categorical context of our paper is much related to the one of Gabriel and Ulmer~\cite{GabUlm}, Ad\'amek and Rosick\'y~\cite{AdRo94}.

For any class~$\cX$ of objects in a category~$\cA$ and any functor $\Phi\colon\cA\to\cB$, we denote by~$\Phi\cX$ (or~$\Phi(\cX)$) the class of all objects~$B$ of~$\cB$ such that $B\cong\Phi(X)$ for some $X\in\cX$.
If~$\cX$ is the class of all objects of~$\cA$, $\Phi\cX$ will be called the \emph{range of~$\Phi$} and denoted by~$\rng\Phi$.

We denote by~$\id_A$ the identity morphism on any object~$A$ in a given category.

\begin{itemize}
\item
Monomorphisms (in any category) and one-to-one maps (between sets) will be denoted in the form $f\colon A\rightarrowtail B$.

\item
A category may come along with a special class of monomorphisms called \emph{embeddings} (e.g., order-embeddings between posets), which will then be denoted in the form $f\colon A\hookrightarrow B$.

\item
Epimorphisms (in any category) and surjective maps (between sets) will be denoted in the form $f\colon A\twoheadrightarrow B$.

\end{itemize}

To any category~$\cC$, any object~$C$ of~$\cC$, and any full subcategory~$\cS$ of~$\cC$, we assign
\begin{itemize}
\item
the \emph{slice category} $\cS/C$, whose objects are all the arrows with domain in~$\cS$ and codomain~$C$, and where a morphism from an arrow $x\colon X\to C$ to an arrow $y\colon Y\to C$ is defined as an arrow $f\colon X\to Y$ such that $x=y\circ f$;

\item
the \emph{subobject category} $\cS\dnw C$, whose objects are all the monomorphisms with domain in~$\cS$ and codomain~$C$, and where a morphism from an arrow $x\colon X\rightarrowtail C$ to an arrow $y\colon Y\rightarrowtail C$ is defined as the unique arrow (necessarily monic) $f\colon X\rightarrowtail Y$, if it exists, such that $x=y\circ f$; in that case we write $x\utr y$.

\end{itemize}

Natural transformations from a functor~$\Phi$ to a functor~$\Psi$ will be denoted in the form $\eta\colon\Phi\todot\Psi$.

\subsection{Set theory}
We denote by~$\card{X}$ the cardinality of a set~$X$, and we say that~$X$ is \emph{$\gl$-small} (for a cardinal~$\gl$) if $\card{X}<\gl$.
We denote by~$\Pow(X)$ the powerset of~$X$, by~$[X]^{{<}\gl}$ the set of all $\gl$-small subsets of~$X$, and by~$[X]^{\gl}$ the set of all subsets of~$X$ with cardinality~$\gl$.

For sets~$X$ and~$Y$ with $X\subseteq Y$, we denote by~$\id_X^Y$ the inclusion map from~$X$ into~$Y$, and we set $\id_X\eqdef\id_X^X$\,.

We denote by~$\dom(f)$ (resp., $\rng(f)$) the domain (resp., range) of a function~$f$.
For a subset~$X$ of the domain (resp., range) of~$f$, we denote by~$f[X]$ (resp., $f^{-1}[X]$, or~$f^{-1}X$) the image (resp., inverse image) of~$X$ under~$f$.

We denote by~$\sup{X}$ the least upper bound of any set~$X$ of cardinal numbers.

For any cardinal number~$\gl$, we denote by~$\gl^+$ the successor cardinal of~$\gl$.
Further, we set $\gl^{+0}\eqdef\gl$, $\gl^{+(n+1)}\eqdef(\gl^{+n})^+$ for every nonnegative integer~$n$.
We also set
$\exp_0(\gl)\eqdef\gl$, $\exp_{n+1}(\gl)\eqdef2^{\exp_n(\gl)}$,
$\exp_n(\us{\gl})\eqdef\sup\setm{\exp_n(\ga)}{\ga<\gl}$.

\subsection{Partially ordered sets (posets)}
Let~$P$ be a poset.
We set $P^{\infty}\eqdef P\cup\set{\infty}$ for a new top element~$\infty$.

Let~$\gl$ be an infinite regular cardinal.
We say that~$P$ is
\begin{itemize}
\item
\emph{$\gl$-directed} if every nonempty $\gl$-small subset of~$P$ is bounded above in~$P$;

\item
\emph{directed} if it is $\go$-directed;

\item
\emph{$\gl$-join-com\-plete} if every nonempty $\gl$-small subset of~$P$ has a least upper bound in~$P$.

\end{itemize}

For a subset~$X$ and an element~$a$ in~$P$, we set
 \[
 X\dnw a\eqdef\setm{x\in X}{x\leq a}\quad\text{and}\quad
 X\upw a\eqdef\setm{x\in X}{x\geq a}\,.
 \]
We say that~$P$ is a \emph{forest} if~$P\dnw p$ is a chain whenever $p\in P$.

We say that a subset~$X$ of~$P$ is
\begin{itemize}
\item
\emph{a lower subset} (resp., \emph{an upper subset}) of~$P$ if for every $x\in X$, $P\dnw x\subseteq X$ (resp., $P\upw x\subseteq X$);

\item
\emph{an ideal of~$P$} if it is a nonempty directed lower subset of~$P$; we  denote by~$\Id{P}$ the set of all ideals of~$P$, partially ordered under set inclusion;

\item
\emph{cofinal in~$P$} if every element of~$P$ is bounded above by some element of~$X$;

\item
\emph{$\gl$-closed in~$P$}, where~$\gl$ is an infinite regular cardinal and~$P$ is $\gl$-join-com\-plete, if the least upper bound of any $\gl$-small subset of~$X$ belongs to~$X$.

\end{itemize}

It is well known that whenever $\gl>\go$ and~$P$ is $\gl$-join-com\-plete, \emph{any intersection of a $\gl$-small set of $\gl$-closed cofinal subsets of~$P$ is $\gl$-closed cofinal}.
This observation will be mostly applied to finite products of sets of the form~$[\gO]^{<\gl}$.

Many of our ``counterexample diagrams'' will be indexed by the powerset of a three-element set, denoted as
$\Pow[3]\eqdef\set{\es,1,2,3,12,13,23,123}$ (with $12=21$, $123=312$, and so on), partially ordered by inclusion.

A poset~$P$ is \emph{lower $\gl$-small} if $\card(P\dnw x)<\gl$ whenever $x\in P$.
We will just say \emph{lower finite} instead of lower $\go$-small.

We denote by~$\Min{P}$ (resp., $\Max{P}$) the set of all minimal (resp. maximal) elements of~$P$, and we set
 \[
 P^-\eqdef P\setminus\Min{P}\quad\text{and}\quad
 P^=\eqdef P\setminus\Max{P}\,.
 \]
For posets~$P$ and~$Q$, a map $f\colon P\to Q$ is \emph{isotone} (resp., \emph{antitone}) if for all $x,y\in P$, $x\leq y$ implies that $f(x)\leq f(y)$ (resp., $f(y)\leq f(x)$).
If~$x\leq y$ is equivalent to $f(x)\leq f(y)$, we say that~$f$ is an \emph{order-embedding}.

We refer the reader to Gr\"atzer~\cite{LTF} for all undefined lattice-theoretical concepts.

\subsection{Lattice-ordered groups}
An \emph{\lgrp} is a group endowed with a translation-invariant lattice order (cf. Bigard \emph{et al.}~\cite{BKW}, Anderson and Feil~\cite{AnFe}).
Although our \lgrp{s} will not be assumed to be Abelian, we will denote them additively.
The \emph{positive cone} of an \lgrp~$G$ is $G^+\eqdef\setm{x\in G}{x\geq0}$.
An \emph{$\ell$-sub\-group} of~$G$ is a subgroup closed under the lattice operations, and an \emph{$\ell$-ideal} is an order-convex normal $\ell$-sub\-group.
For any element~$x$ in an \lgrp~$G$, we denote by~$\seq{x}$ (resp., $\seql{x}$) the convex $\ell$-sub\-group (resp., $\ell$-ideal) of~$G$ generated by~$x$.

The Stone dual of the spectrum of any Abelian $\ell$-group~$G$ is isomorphic to the (distributive) lattice~$\Idc{G}$ of all principal (equivalently, finitely generated) $\ell$-ideals of~$G$ (cf. Proposition~1.19, together with Theorem~1.10 and Lemma~1.20, in Keimel~\cite{Keim1971}).

Whenever~$G$ is a totally ordered group and~$H$ is an \lgrp, the \emph{lexicographical product} $G\lex H$ is the \lgrp\ structure on the cartesian product~$G\times H$ with the positive cone
 \[
 (G\lex H)^+\eqdef\setm{(x,y)\in G\times H}{x\geq0\text{ and }
 (x=0\Rightarrow y\geq0)}\,.
 \]
 
\subsection{Rings}
Our rings will not necessarily be unital.
A ring~$R$ is
\begin{itemize}
\item
\emph{von Neumann regular} (cf. Goodearl~\cite{Good91} for the unital case) if for all $x\in R$ there exists $y\in R$ such that $x=xyx$;

\item\emph{semiprimitive} if its Jacobson radical (cf. Jacobson \cite[\S~2]{Jaco1945}, Herstein \cite[page~16]{Hers1968}) is trivial;

\item an \emph{exchange ring} if for all $x\in R$ there are $r,s\in R$ together with an idempotent~$e\in R$ such that $e=rx=x+s-sx$ (cf. Warfield~\cite{Warf1972}, and Ara~\cite{Ara97} for the non-unital case).
\end{itemize}

Also recall that a C*-algebra is an exchange ring if{f} it has real rank zero (cf. Ara \cite[Theorem~3.8]{Ara97}).
For more information and references, see Wehrung~\cite{VLiftDefect}.

\subsection{Model theory for infinitary languages}
Let~$\gS$ be a first-order language and let~$\gl$ be an infinite regular cardinal.
The formulas of~$\scL_{\infty\gl}$ are constructed inductively from atomic formulas of~$\gS$, by allowing quantifications over less than~$\gl$ variables, negations, and arbitrary conjunctions and disjunctions over sets of formulas with less than~$\gl$ free variables altogether.
The satisfaction of an $\scL_{\infty\gl}$ formula, in a $\gS$-structure~$M$, is defined the standard way.
Assignments of free variables will be written in the form $\vec{a}\subseteq M$; we should keep in mind that~$\vec{a}$ is in fact a function with $\gl$-small domain and with values in~$M$.

We refer to Dickmann~\cite{Dickm1985a} for more information on model theory for infinitary languages.

\section{Setting up the categorical background}\label{S:CatBckgd}

\subsection{Non-commutative diagrams}\label{Su:NonCommDiag}
The whole paper will be articulated around the concept of a (non-com\-mu\-ta\-tive) diagram introduced in Wehrung~\cite{Ceva}.
Let us recall the underlying definitions.

A (\emph{com\-mu\-ta\-tive}) \emph{diagram}, in a category~$\cS$, is often defined as a functor~$D$ from a category~$\cP$ (the ``indexing category'' of the diagram) to~$\cS$.
Allowing any morphism in~$\cP$ to be sent to more than one morphism in~$\cS$, we get a more general definition of diagram that makes~$D$ a kind of ``non-de\-ter\-min\-is\-tic
 functor''.

\begin{definition}\label{D:Diagram}
Let~$\cP$ and~$\cS$ be categories.
A \emph{$\cP$-indexed diagram in~$\cS$} is an assignment~$D$, sending every object~$p$ of~$\cP$ to an object~$D(p)$ (or~$D_p$) of~$\cS$ and each morphism $x\colon p\to q$ in~$\cP$ to a \emph{nonempty set}~$D(x)$ of morphisms from~$D(p)$ to~$D(q)$, in such a way that the following statements hold:
\begin{enumerater}
\item $\id_{D(p)}\in D(\id_p)$ for every objet~$p$ of~$\cP$;

\item whenever $x\colon p\to q$ and $y\colon q\to r$ are morphisms in~$\cP$, $u\in D(x)$, and $v\in D(y)$, the composite~$v\circ u$ belongs to~$D(y\circ x)$.
\end{enumerater}
We say that~$D$ is a \emph{com\-mu\-ta\-tive diagram} if~$D(x)$ is a singleton whenever~$x$ is a morphism in~$\cP$.
A \emph{uniformization} of a diagram~$D$ is a $\cP$-indexed \emph{com\-mu\-ta\-tive} diagram~$D^*$ such that $D^*(x)\subseteq D(x)$ for every morphism~$x$ in~$\cP$.

We denote by~$\cS^{\cP}$ the category of all functors from~$\cP$ to~$\cS$ with natural transformations as morphisms.
\end{definition}

Specializing to the case where~$\cP$ is the category naturally assigned to a poset~$P$ (i.e., there is an arrow from~$p$ to~$q$ if{f} $p\leq q$, and then the arrow is unique), we get poset-indexed diagrams.
We will often write poset-indexed \emph{com\-mu\-ta\-tive} diagrams in the form
 \[
 \vec{D}=\vecm{D_p,\gd_p^q}{p\leq q\text{ in }P}\,,
 \]
where all~$D_p$ are objects and all $\gd_p^q\colon D_p\to D_q$ are morphisms subjected to the usual commutation relations (i.e., $\gd_p^p=\id_{D_p}$, $\gd_p^r=\gd_q^r\circ\gd_p^q$ whenever $p\leq q\leq r$); hence $\vec{D}(p,q)=\set{\gd_p^q}$.
If~$P$ is a directed (resp., $\gl$-directed, with~$\gl$ an infinite regular cardinal) poset we will say that~$\vec{D}$ is a \emph{direct system} (resp., \emph{$\gl$-direct system}).

Our next definition introduces powers of a diagram.

\begin{definition}\label{D:D^I}
Let~$I$ be a set, let~$\cS$ be a category, let~$P$ be a poset, and let~$D$ be a $P$-indexed diagram in~$\cS$.
We suppose that the product $D^I(p)\eqdef\prod_{i\in I}D(p_i)$ exists in~$\cS$ for every $p=\vecm{p_i}{i\in I}\in P^I$.
Whenever $p=\vecm{p_i}{i\in I}$ and $q=\vecm{q_i}{i\in I}$ in~$P^I$ with $p\leq q$, let~$D^I(p,q)$ consist of all morphisms in~$\cS$ of the form $\prod_{i\in I}f_i\colon D^I(p)\to D^I(q)$ where each $f_i\in D(p_i,q_i)$.
\end{definition}

It is straightforward to verify that the structure~$D^I$ introduced in Definition~\ref{D:D^I} is a $P^I$-indexed diagram (in the sense of Definition~\ref{D:Diagram}).

\begin{remark}\label{Rk:Diagram}
It could be tempting to set Definition~\ref{D:Diagram} as foundation of a kind of ``non-commutative category theory''.
The author's enquiries on the matter, aimed both at a few specialists and the literature, were unsuccessful.
Not all usual categorical concepts would carry over to that more general framework, most notably the one of directed colimit.
That caveat, and the needs of the present paper, gave rise to the Uniformization Lemma (Lemma~\ref{L:UnifLem}).

On the other hand, observe that every $\cP$-indexed (non-commutative) diagram~$D$ in~$\cS$ can, in principle, be viewed as a commutative diagram~$\widehat{D}$, indexed by a category~$\widehat{\cP}$ ``covering''~$\cP$.
In more detail: let~$\widehat{\cP}$ have the same objects as~$\cP$, and for objects~$p$ and~$q$ in~$\cP$ define~$\widehat{\cP}(p,q)$ (i.e., the class of all morphisms from~$p$ to~$q$ in~$\widehat{\cP}$) as the class of all pairs $(x,u)$ where $x\in\cP(p,q)$ and $u\in D(x)$.
Such pairs are composed componentwise (i.e., $(x,u)\circ(y,v)=(x\circ y,u\circ v)$).
The forgetful functor $F\colon\widehat{\cP}\to\cP$ sends every object to itself and every pair~$(x,u)$ to~$x$.
This functor is \emph{full} (because all~$D(x)$ are nonempty) and \emph{small-to-one} in the sense that the preimage under~$F$ of any morphism in~$\cP$ is a set (because all~$D(x)$ are sets).
Now define a functor $\widehat{D}\colon\widehat{\cP}\to\cS$ by $\widehat{D}(p)\eqdef D(p)$ and $\widehat{D}(x,u)\eqdef u$ whenever $x\colon p\to q$ in~$\widehat{\cP}$ and $u\in D(x)$.

Conversely, let~$\cP'$ be a category with the same objects as~$\cP$ and let $F\colon\cP'\to\cP$ be a full, small-to-one functor sending every object of~$\cP$ to itself.
Then every commutative diagram $D'\colon\cP'\to\cS$ gives rise to a $\cP$-indexed diagram~$D$ in~$\cS$, by setting $D(p)\eqdef D'(p)$ and $D(x)\eqdef\setm{D'(X)}{X\in\cP'(p,q)\,,\ x=F(X)}$ whenever $x\colon p\to q$ in~$\cP$.
This transformation, evaluated at $\cP'\eqdef\widehat{\cP}$ and the forgetful functor $F\colon\widehat{\cP}\to\cP$ both defined in the paragraph above, returns the original diagram~$D$.
In that sense, the commutative diagram~$\widehat{D}$ encodes the diagram~$D$.
\end{remark}

\subsection{{}From finitely presentable to $\gl$-pre\-sentable}
\label{S:Fin2gk}

Let~$\gl$ be an infinite regular cardinal.
We will say that an object~$A$ in a category~$\cA$ is \emph{weakly $\gl$-pre\-sentable}%
\footnote{
Although formally different, the definition of ``weakly $\gl$-pre\-sentable'' introduced in Gillibert and Wehrung~\cite{Larder} will be equivalent to ours in all contexts occurring in the present paper.
}
if for every $\gl$-directed colimit co-cone
 \[
 \vecm{B,\gb_i}{i\in I}=\varinjlim\vecm{B_i,\gb_i^j}{i\leq j\text{ in }I}
 \quad\text{within }\cA\,,
 \]
every morphism $\gf\colon A\to B$ factors through some~$B_i$\,, that is, $\gf=\gb_i\circ\psi$ for some $\psi\colon A\to B_i$\,.
If, in addition, for all $i\in I$ and all $\xi,\eta\colon A\to B_i$ such that $\gb_i\circ\xi=\gb_i\circ\eta$ there exists $j\in I$ such that $i\leq j$ and $\gb_i^j\circ\xi=\gb_i^j\circ\eta$, we get the usual definition of \emph{$\gl$-pre\-sentability} of~$A$ (cf. Gabriel and Ulmer \cite[Definition~6.1]{GabUlm} or Definitions~1.1 and~1.13 in Ad\'amek and Rosick\'y~\cite{AdRo94}).

We shall denote by~$\Pres_{\gl}\cA$ the full subcategory of~$\cA$ consisting of all $\gl$-pre\-sentable objects.

Recall (cf. Gabriel and Ulmer~\cite{GabUlm},  Ad{\'a}mek and Ro\-si\-ck{\'y}~\cite{AdRo94}) that a category~$\cC$ is \emph{$\gl$-filtered} if every subcategory of~$\cC$ with less than~$\gl$ morphisms has a compatible co-cone.
If $\gl=\go$ we will just say \emph{filtered} instead of $\go$-filtered.

Ad\'amek and Rosick\'y \cite[Theorem~1.5]{AdRo94} prove that every small filtered category~$\cC$ admits a cofinal functor from a directed poset.
They also observe in \cite[Remark~1.21]{AdRo94} that this result extends, with a similar proof, to $\gl$-filtered categories: \emph{for every small $\gl$-filtered category~$\cC$, there are a $\gl$-directed poset~$P$ and a cofinal functor from~$P$ to~$\cC$}.
Hence, \emph{a functor preserves $\gl$-filtered colimits \pup{indexed by small $\gl$-filtered categories} if{f} it preserves $\gl$-directed colimits \pup{indexed by $\gl$-directed posets}.}

We will say that a functor is \emph{$\gl$-continuous} if it preserves all $\gl$-directed colimits.

Next, we need to show how to extend the results of Gillibert and Wehrung \cite[\S~1.4]{Larder}, enabling us to extend a functor from $\gl$-pre\-sentable objects to all objects, from the finitely presentable case to the $\gl$-pre\-sentable case.
The technical result \cite[Lemma~1.4.1]{Larder} extends modulo the following changes:
\begin{itemize}
\item
The category~$\cS$, instead of having all directed colimits, is required to have all $\gl$-directed colimits.

\item
The~$A_i$ are all $\gl$-pre\-sentable.
\end{itemize}

We say that a full subcategory~$\cA^{\dagger}$ of~$\cC$ is \emph{$\gl$-dense} in~$\cA$ if every object of~$\cA$ is a colimit of a $\gl$-direct system from~$\cA^{\dagger}$.
The relevant analogue of Gillibert and Wehrung \cite[Proposition~1.4.2]{Larder} is then the following.

\begin{lemma}\label{L:ExtFunctgkPres}
Let~$\cA$ be a category, let~$\cA^{\dagger}$ be a full subcategory of~$\Pres_{\gl}\cA$ which is $\gl$-dense in~$\cA$, and let~$\cS$ be a category with all $\gl$-directed colimits.
Then every functor~$\Psi\colon\cA^{\dagger}\to\cS$ extends to a unique \pup{up to natural isomorphism} functor $\ol{\Psi}\colon\cA\to\nobreak\cS$ which preserves all $\gl$-directed colimits from~$\cA^{\dagger}$.
Furthermore, if~$\cA^{\dagger}$ has small hom-sets, then~$\ol{\Psi}$ is $\gl$-continuous.
\end{lemma}

Necessarily, if an object~$A$ of~$\cA$ is expressed as a $\gl$-directed colimit $A=\varinjlim_{i\in I}A_i$\,, with all $A_i\in\cA^{\dagger}$, then $\ol{\Psi}(A)=\varinjlim_{i\in I}\Psi(A_i)$.
One of the main difficulties of the argument is to prove that~$\varinjlim_{i\in I}\Psi(A_i)$ is, up to isomorphism, independent of the chosen $\gl$-direct system based on the~$A_i$\,.
The proof of Lemma~\ref{L:ExtFunctgkPres} is essentially the same as the one of Gillibert and Wehrung \cite[Proposition~1.4.2]{Larder}, with the following changes:
\begin{itemize}
\item
Change ``directed'' to ``$\gl$-directed'' and ``directed colimits'' to ``$\gl$-directed colimits''.

\item
Restate the Claim on page~30 as ``$\cP$ is $\gl$-filtered''.
The proof remains almost identical.

\item
At the bottom of page~31, change the use of \cite[Theorem~1.5]{AdRo94} to \cite[Remark~1.21]{AdRo94}.

\end{itemize}

All along the present paper, Lemma~\ref{L:ExtFunctgkPres} will be applied to the case where~$\cA$ is the category~$\Bool_P$ of all $P$-scaled Boolean algebras (cf. Section~\ref{S:BoolP}), for a poset~$P$, and~$\cA^{\dagger}$ is the full subcategory of all $\gl$-pre\-sentable objects of~$\cA$, for an infinite regular cardinal~$\gl$ (under additional conditions on~$P$ if $\gl>\go$).

\section{Purity and freshness}\label{S:fresh}

In this section we shall introduce a concept, making sense in any category, strengthening the one of \emph{$\gl$-purity of a morphism} introduced in Ad\'a\-mek and Rosick\'y \cite[Definition~2.27]{AdRo94}.
It will turn out that this concept is equivalent to $\gl$-purity in categories of sets (cf. Lemma~\ref{L:SetPureFresh}), implies $\scL_{\infty\gl}$-el\-e\-men\-tary embeddability for categories of models of first-order languages (cf. Proposition~\ref{P:Fresh2Eltary}), and is preserved under $\gl$-continuous functors (cf. Proposition~\ref{P:PresHomFresh}).
As a consequence, $\gl$-continuous functors on categories of models generate $\scL_{\infty\gl}$-el\-e\-men\-tary embeddings (cf. Proposition~\ref{P:EltEq}); this will be the main support for our definition of anti-el\-e\-men\-tar\-ity.

\begin{definition}\label{D:Refresh}
Let~$\cC$ be a category and let~$\gl$ be an infinite regular cardinal.
A morphism $f\colon A\to B$ in~$\cC$ is \emph{$\gl$-fresh} if for all $\gl$-pre\-sentable objects~$A'$ and~$B'$ and all morphisms $\xi\colon A'\to B'$, $\ga\colon A'\to A$, and $\gb\colon B'\to B$, if $\gb\xi=f\ga$, then there are a morphism $\eta\colon B'\to A$ and an automorphism~$\gs$ of~$B$ such that $\ga=\eta\xi$ and $f\eta=\gs\gb$ (cf. Figure~\ref{Fig:fresh}).
\end{definition}

\begin{figure}[htb]
\begin{tikzcd}
\centering
A' \arrow[r,"\xi"]\arrow[d,"\ga"] & B'\arrow[d,"\gb"] &&
A' \arrow[r,"\xi"]\arrow[d,"\ga"] &
B'\arrow[d,"\gs\gb"]\arrow[dl,"\eta"]\arrow[r,"\gb"] &
B\arrow[dl,"\gs",bend left=10]\\
A\arrow[r,"f"'] & B &&
A\arrow[r,"f"'] & B &
\end{tikzcd}
\caption{Stating $\gl$-freshness of the morphism $f$}
\label{Fig:fresh}
\end{figure}

Observe that $\gl$-freshness strengthens $\gl$-purity, which is just the part of Definition~\ref{D:Refresh} not involving~$\gs$ (i.e., only require the existence of~$\eta$ such that $\ga=\eta\xi$).
There are other purely categorical approaches of $\gl$-el\-e\-men\-tar\-ity embeddings, such as Beke and Rosick\'{y}'s \emph{$\gl$-embeddings} \cite[Definition~3.1]{BekRos2018}.
Our Definition~\ref{D:Refresh} covers an \emph{a priori} different purpose: in order to maximize the strength and scope of our results of anti-el\-e\-men\-tar\-ity, we need $\gl$-freshness to be as strong as possible while consistent with both Lemma~\ref{L:SetPureFresh} and Proposition~\ref{P:PresHomFresh}.
In that light, the simplicity of Definition~\ref{D:Refresh} came to the author as a surprise.

As the following lemma, whose proof we leave to the reader as an exercise, shows, purity and freshness coincide in every category of the form~$\Powi(\gO)$.
This observation extends easily to further categories, such as the category of all sets with one-to-one maps.

\begin{lemma}[$\gl$-freshness for sets]\label{L:SetPureFresh}
Let~$\gl$ be an infinite regular cardinal, let~$\gO$ be a set, and let $f\colon A\rightarrowtail B$ be a morphism in~$\Powi(\gO)$.
The following are equivalent:
\begin{enumeratei}
\item $f$ is $\gl$-fresh.

\item $f$ is $\gl$-pure.

\item Either~$f$ is a bijection or $\gl\leq\card A$.

\end{enumeratei}
\end{lemma}

In \cite[Proposition~5.15]{AdRo94}, Ad\'amek and Rosick\'y characterize $\gl$-purity of a homomorphisms of~$\gS$-structures by elementarity with respect to all positive-primitive formulas.
Our next observation shows that in that context, $\gl$-freshness is a strong form of $\scL_{\infty\gl}$-el\-e\-men\-tar\-ity.

\begin{proposition}\label{P:Fresh2Eltary}
Let~$\gl$ be an infinite regular cardinal and let~$\gS$ be a first-order language of arity less%
\footnote{
More precisely, all function and relation symbols of~$\gS$ have arity less than~$\gl$.
}
than~$\gl$.
Let $f\colon A\to\nobreak B$ be a $\gS$-homomorphism.
If~$f$ is $\gl$-fresh within~$\Str\gS$, then it is an $\scL_{\infty\gl}$-el\-e\-men\-tary embedding.
\end{proposition}

\begin{proof}
We need to prove that $A\models\sF(\vec{a})$ if{f} $B\models\sF(f\vec{a})$, for every~$\scL_{\infty\gl}$ formula~$\sF$ of~$\gS$ and every assignment~$\vec{a}$ of the free variables of~$\sF$ in~$A$.
We argue by induction on the complexity of~$\sF$.

The only two nontrivial induction steps consist of proving that $B\models\sF(f\vec{a})$ implies that $A\models\sF(\vec{a})$, for any $\scL_{\infty\gl}$ formula~$\sF$ which is either atomic or of the form~$(\exists\vec{\vy})\sG$ for a formula~$\sG$ of smaller complexity.

Let us begin with the case where~$\sF$ is atomic and denote by~$\vec{\vx}$ the set of all free variables of~$\sF$.
Define~$A'$ as the free~$\gS$-structure on~$\vec{\vx}$ (term algebra, with all relation symbols interpreted by the empty set) and~$\bxi$ as the $\gS$-congruence of~$A'$ generated by the single relation~$\sF(\vec{\vx})$; then set $B'\eqdef A'/{\bxi}$ and denote by $\xi\colon A'\twoheadrightarrow B'$ the canonical projection.
Using the relation $B\models\sF(f\vec{a})$, we get (unique) $\gS$-homomorphisms $\ga\colon A'\to A$ and $\gb\colon B'\to B$ such that $\ga\vec{\vx}=\vec{a}$ and $\gb\xi\vec{\vx}=f\vec{a}$.
Since~$f$ is $\gl$-pure, there exists $\eta\colon B'\to A$ such that $\ga=\eta\xi$.
Since $B'\models\sF(\xi\vec{\vx})$ and~$\eta$ is a $\gS$-homomorphism, we get $A\models\sF(\eta\xi\vec{\vx})$, that is, $A\models\sF(\vec{a})$, as desired.

Now let~$\sF(\vec{\vx})$ be~$(\exists\vec{\vy})\sG(\vec{\vx},\vec{\vy})$.
Our assumption $B\models\sF(f\vec{a})$ means that there exists $\vec{c}\subseteq B$ such that $B\models\sG(f\vec{a},\vec{c})$.
Define~$A'$ as the free~$\gS$-structure on~$\vec{\vx}$, $B'$ as the free~$\gS$-structure on~$\vec{\vx}\cup\vec{\vy}$, and~$\xi$ as the natural inclusion from~$A'$ into~$B'$.
There are (unique) $\gS$-homomorphisms $\ga\colon A'\to A$ and $\gb\colon B'\to B$ such that $\ga\vec{\vx}=\vec{a}$, $\gb\vec{\vx}=f\vec{a}$, and $\gb\vec{\vy}=\vec{c}$.
Since~$f$ is $\gl$-fresh, there are a $\gS$-homomorphism $\eta\colon B'\to A$ and a $\gS$-automorphism~$\gs$ of~$B$ such that $\ga=\eta\xi$ and $f\eta=\gs\gb$ (cf. Figure~\ref{Fig:fresh}).
Observe that $f\ga=f\eta\xi=\gs\gb\xi=\gs f\ga$, thus (as $\vec{a}=\ga\vec{\vx}$) we get $\gs f\vec{a}=f\vec{a}$.
Furthermore, setting $\vec{b}\eqdef\eta\vec{\vy}$, we get $\gs\vec{c}=\gs\gb\vec{\vy}=f\vec{b}$.
Now since $B\models\sG(f\vec{a},\vec{c})$ and~$\gs$ is an automorphism, we get $B\models\sG(\gs f\vec{a},\gs\vec{c})$, that is, by the above, $B\models\sG(f\vec{a},f\vec{b})$.
{F}rom our induction hypothesis it follows that $A\models\sG(\vec{a},\vec{b})$, thus $A\models\sF(\vec{a})$, as desired.
\end{proof}

Ad\'amek and Rosick\'y establish in \cite[Proposition~2.30]{AdRo94} that in any locally $\gl$-pre\-sentable category, the $\gl$-pure morphisms are exactly the $\gl$-directed colimits of split monomorphisms.
It follows that any image of a $\gl$-pure morphism, under a $\gl$-continuous functor (between $\gl$-accessible categories) is $\gl$-pure.
Our next result focuses on the latter preservation result, also extending it to $\gl$-fresh morphisms.

\begin{proposition}\label{P:PresHomFresh}
Let~$\cS$ and~$\cT$ be categories, let~$\gl$ be an infinite regular cardinal, and let $\gC\colon\cS\to\cT$ be a $\gl$-continuous functor.
We assume that~$\Pres_{\gl}\cS$ is $\gl$-dense in~$\cS$.
Then for every morphism $f\colon A\to B$ in~$\cS$, if~$f$ is $\gl$-pure \pup{resp., $\gl$-fresh} in~$\cS$, then~$\gC(f)$ is $\gl$-pure \pup{resp., $\gl$-fresh} in~$\cT$.
\end{proposition}

\begin{proof}
We provide the proof for $\gl$-freshness; the proof for $\gl$-purity is contained in that argument.
Let~$P$ and~$Q$ be $\gl$-pre\-sentable objects of~$\cT$ and let $\xi\colon P\to Q$, $\ga\colon P\to\gC(A)$, and $\gb\colon Q\to\gC(B)$ be morphisms in~$\cT$ such that $\gb\xi=\gC(f)\ga$.
{F}rom our denseness assumption it follows that there are $\gl$-directed colimit representations
 \begin{align*}
 \vecm{A,\mu_i}{i\in I}&=
 \varinjlim\Vecm{U_i,\mu_i^{i'}}{i\leq i'\text{ in }I}\,,\\
 \vecm{B,\nu_j}{j\in J}&=
 \varinjlim\Vecm{V_j,\nu_j^{j'}}{j\leq j'\text{ in }J}
 \end{align*}
in~$\cS$, with all~$U_i$ and~$V_j$ being $\gl$-pre\-sentable.
Since $\ga\colon P\to\gC(A)=\varinjlim_{i\in I}\gC(U_i)$ and~$P$ is $\gl$-pre\-sentable, there are~$i\in I$ and $\ga'\colon P\to\gC(U_i)$ such that $\ga=\gC(\mu_i)\ga'$.
Since $f\mu_i\colon U_i\to B=\varinjlim_{j\in J}V_j$ and~$U_i$ is $\gl$-pre\-sentable, there exists $j_0\in J$ such that~$f\mu_i$ factors through~$V_{j_0}$\,.
Since $\gb\colon Q\to\gC(B)=\varinjlim_{j\in J}\gC(V_j)$ and~$Q$ is $\gl$-pre\-sentable, there are~$j\in J$ and $\gb'\colon Q\to\gC(V_j)$ such that $\gb=\gC(\nu_j)\gb'$.
Furthermore, we may replace~$j$ by any upper bound of~$\set{j_0,j}$ in~$J$ and thus assume that $f\mu_i$ factors through~$V_j$\,.
Pick $\xi'\colon U_i\to V_j$ such that $f\mu_i=\nu_j\xi'$.
{F}rom $\gC(f)\ga=\gb\xi$ it follows that $\gC(f\mu_i)\ga'=\gC(\nu_j)\gb'\xi$, that is, $\gC(\nu_j)\gC(\xi')\ga'=\gC(\nu_j)\gb'\xi$.
Since the common domain of the morphisms~$\gC(\xi')\ga'$ and~$\gb'\xi$ (viz.~$P$) is $\gl$-pre\-sentable, there exists $k\geq j$ such that $\gC(\nu_j^k)\gC(\xi')\ga'=\gC(\nu_j^k)\gb'\xi$.
Replacing~$j$ by~$k$, we may thus suppose that $\gC(\xi')\ga'=\gb'\xi$, that is, the diagram represented in the left hand side of Figure~\ref{Fig:PresHomFresh} is commutative.

\begin{figure}[htb]
\begin{tikzcd}
\centering
P \arrow[r,"\xi"]\arrow[d,"\ga'"]\arrow[dd,"\ga"',bend right=40]
& Q\arrow[d,"\gb'"']\arrow[dd,"\gb",bend left=40] &&&&\\
\gC(U_i)\arrow[r,"\gC(\xi')"]\arrow[d,"\gC(\mu_i)"] &
\gC(V_j)\arrow[d,"\gC(\nu_j)"'] &&
U_i\arrow[r,"\xi'"]\arrow[d,"\mu_i"'] &
V_j\arrow[d,"\gs'\nu_j"]\arrow[dl,"\eta'"]\arrow[r,"\nu_j"] &
B\arrow[dl,"\gs'",bend left=30]\\
\gC(A)\arrow[r,"\gC(f)"'] & \gC(B) && A\arrow[r,"f"'] & B&
\end{tikzcd}
\caption{Illustrating the proof of Proposition~\ref{P:PresHomFresh}}
\label{Fig:PresHomFresh}
\end{figure}

Since~$f$ is $\gl$-fresh and~$U_i$ and~$V_j$ are both $\gl$-pre\-sentable, there are $\eta'\colon V_j\to A$ and an automorphism~$\gs'$ of~$B$ such that $\eta'\xi'=\mu_i$ and $f\eta'=\gs'\nu_j$ (see the right hand side of Figure~\ref{Fig:PresHomFresh}).
Now set $\eta\eqdef\gC(\eta')\gb'$ and $\gs\eqdef\gC(\gs')$.
Then
 \[
 \eta\xi=\gC(\eta')\gb'\xi=\gC(\eta')\gC(\xi')\ga'=\gC(\mu_i)\ga'=\ga
 \]
whereas~$\gs$ is an automorphism of~$\gC(B)$ and
 \begin{equation*}
 \gC(f)\eta=\gC(f\eta')\gb'=\gC(\gs'\nu_j)\gb'=\gs\gb\,.\tag*{\qed}
 \end{equation*}
 \renewcommand{\qed}{}
\end{proof}

The following result is essentially a reformulation, in our context, of Feferman's \cite[Theorem~6]{Fefe1972} or Beke and Rosick\'{y}'s \cite[Proposition~2.14]{BekRos2018}.
However, while the preservation of monomorphisms belongs to the \emph{assumptions} of the two abovementioned results, it appears in our result as a \emph{conclusion} (for monomorphisms with large enough domain).

\begin{proposition}\label{P:EltEq}
Let~$\gl$ be an infinite regular cardinal, let~$\cC$ be a category, let~$\gO$ be a set, and let $\gC\colon\Powi(\gO)\to\cC$ be a $\gl$-continuous functor.
Then for every $f\colon X\rightarrowtail Y$ in~$\Powi(\gO)$ with $\card{X}\geq\gl$, $\gC(f)$ is a $\gl$-fresh morphism from~$\gC(X)$ into~$\gC(Y)$.
In particular, if~$\cC=\Str\gS$ for a first-order language~$\gS$ with arity less than~$\gl$, then~$\gC(f)$ is an $\scL_{\infty\gl}$-el\-e\-men\-tary embedding from~$\gC(X)$ into~$\gC(Y)$.
\end{proposition}

\begin{proof}
By Lemma~\ref{L:SetPureFresh}, $f$ is a $\gl$-fresh morphism within~$\Powi(\gO)$.
The $\gl$-pre\-sentable members of~$\Powi(\gO)$ are exactly the $\gl$-small subsets of~$\gO$, and those form a $\gl$-dense full subcategory of~$\Powi(\gO)$.
By Proposition~\ref{P:PresHomFresh}, $\gC(f)$ is a $\gl$-fresh morphism within~$\cC$.
The last statement of Proposition~\ref{P:EltEq} now follows from Proposition~\ref{P:Fresh2Eltary}.
\end{proof}

\section{$P$-scaled Boolean algebras; normal morphisms}
\label{S:BoolP}

The whole monograph Gillibert and Wehrung~\cite{Larder} is articulated around the concept of a \emph{$P$-scaled Boolean algebra}.
This section will consist of a gentle recollection of some known material on those structures, followed by a few basic results on their $\gl$-pre\-sentability.

For an arbitrary poset~$P$, a \emph{$P$-scaled Boolean algebra} is a structure
 \[
 \bA=\pI{A,\vecm{A^{(p)}}{p\in P}}\,,
 \]
where~$A$ is a Boolean algebra, every~$A^{(p)}$ is an ideal of~$A$, $A=\bigvee\vecm{A^{(p)}}{p\in P}$ within the ideal lattice of~$A$, and $A^{(p)}\cap A^{(q)}=\bigvee\vecm{A^{(r)}}{r\geq p,q}$ whenever $p,q\in P$.
For $P$-scaled Boolean algebras~$\bA$ and~$\bB$, a \emph{morphism} from~$\bA$ to~$\bB$ is a homomorphism $f\colon A\to B$ of Boolean algebras such that $f[A^{(p)}]\subseteq B^{(p)}$ for every $p\in P$.
If~$f$ is surjective and $f[A^{(p)}]=B^{(p)}$ for every~$p$, we say that~$f$ is \emph{normal}%
\footnote{
In Gillibert and Wehrung \cite[\S~2.5]{Larder} we observed, without proof, that the normal morphisms are exactly the regular epimorphisms in~$\Bool_P$\,.
We will not need that fact here either.
}
.
For a $P$-scaled Boolean algebra~$\bA$ and an ideal~$I$ of~$A$, with canonical projection~$\pi_I\colon A\twoheadrightarrow A/I$, the \emph{quotient algebra}~$\bA/I$ has underlying Boolean algebra~$B/I$ and $(A/I)^{(p)}=\pi_I[A^{(p)}]$.
Furthermore, $\pi_I$ is a normal morphism and every normal morphism arises this way (cf. \cite[\S~2.5]{Larder}).

The category of all $P$-scaled Boolean algebras is denoted by~$\Bool_P$\,.

The following observation is stated without proof in Gillibert and Wehrung \cite[Remark~2.4.8]{Larder}.
We include a proof here for completeness.

\begin{lemma}\label{L:BoolPmono}
A morphism $\gf\colon\bA\to\bB$ in $\Bool_P$ is monic if{f} it is one-to-one.
\end{lemma}

\begin{proof}
It is trivial that if~$\gf$ is one-to-one then it is monic.
Suppose, conversely, that~$\gf$ is monic and let $a_0,a_1\in A$ such that $\gf(a_0)=\gf(a_1)$.
Since~$1$ belongs to $\bigvee\vecm{A^{(p)}}{p\in P}$, there are a nonempty finite subset~$Q$ of~$P$ and elements $u_q\in A^{(q)}$, for $q\in Q$, such that $1=\bigoplus\vecm{u_q}{q\in Q}$.
Let $C\eqdef\Pow(Q\times\set{0,1})$.
For each $p\in P$, denote by~$C^{(p)}$ the ideal of~$C$ generated by $c_p\eqdef(Q\upw p)\times\set{0,1}$.
Then $\bC\eqdef(C,\vecm{C^{(p)}}{p\in P})$ is a $P$-scaled Boolean algebra.
For each $i\in\set{0,1}$ there is a unique morphism $\ga_i\colon\bC\to\bA$ sending each $\set{(q,0)}$ to~$u_q\wedge a_i$ and each $\set{(q,1)}$ to~$u_q\wedge\neg a_i$\,.
Since $\gf\circ\ga_0=\gf\circ\ga_1$ and~$\gf$ is monic, we get $\ga_0=\ga_1$\,, thus $a_0=a_1$\,.
\end{proof}

For every $\bA\in\Bool_P$\,, the ultrafilter space of~$A$ is denoted in Gillibert and Wehrung~\cite{Larder} by~$\Ult{A}$, and further,
for every $\fa\in\Ult{A}$, the subset
 \begin{equation}\label{Eq:||fa||}
 \|\fa\|_{\bA}\eqdef\setm{p\in P}{\fa\cap A^{(p)}\neq\es}
 \end{equation}
is an ideal of~$P$ (cf. Gillibert and Wehrung \cite[Lemma~2.2.4]{Larder}).
The pair $(\Ult{A},\|{}_{-}\|_{\bA})$ is a so-called \emph{$P$-normed Boolean space} (cf. \cite[\S~2.2]{Larder}, in particular for the duality between $P$-scaled Boolean algebras and $P$-normed Boolean spaces).

\begin{notation}\label{Not:BoolP<gl}
For any infinite regular cardinal~$\gl$, we denote by $\Bool_P^{<\gl}$ the full subcategory of~$\Bool_P$ consisting of all its $\gl$-pre\-sentable members.
\end{notation}

We established in Gillibert and Wehrung \cite[Proposition~2.3.1]{Larder} that the category~$\Bool_P$ has all directed colimits.
Furthermore, the description of those directed colimits given there amounts to saying that $\bA=\varinjlim_{i\in I}\bA_i$\,, with limiting morphisms $\ga_i\colon\bA_i\to\bA$, if{f} $A=\varinjlim_{i\in I}A_i$ in the category of all Boolean algebras and each $A^{(p)}$ is the (directed) union, over $i\in I$, of all $\ga_i[A_i^{(p)}]$.
We will express this by saying that \emph{directed colimits in~$\Bool_P$ are standard}.
Moreover, in \cite[\S~2.4]{Larder} we characterized the finitely pre\-sentable $P$-scaled Boolean algebras  as those~$\bA$ with finite underlying Boolean algebra~$A$ such that for every atom~$a$ of~$A$ there is a largest $p\in P$ such that $a\in A^{(p)}$\,.

For a $P$-scaled Boolean algebra~$\bA$, we introduced in \cite[\S~2.4]{Larder} the set~$\gS_{\bA}$ of all maps~$f$ from the set~$\At{U}$ of all atoms of~$U$ to~$P$, for a finite subalgebra~$U=A_f$ of~$A$, such that $u\in A^{(f(u))}$ for all $u\in\At U$.
Denoting, for each $p\in P$, by~$A_f^{(p)}$ the ideal of~$A_f$ generated by all atoms~$u$ of~$A_f$ such that $p\leq f(u)$, we obtained there a finitely presentable $P$-scaled Boolean algebra~$\bA_f$\,, for which the inclusion map $\bA_f\rightarrowtail\bA$ is a monomorphism in~$\Bool_P$\,.

For $f,g\in\gS_{\bA}$\,, let $f\sqsubseteq g$ hold if $A_f\subseteq A_g$ and for all $(u,v)\in(\At{A_f})\times(\At{A_g})$, $v\leq u$ implies that $f(u)\leq g(v)$.
We proved in \cite[\S~2.4]{Larder} that~$\gS_{\bA}$ is directed under~$\sqsubseteq$ and that $\bA=\varinjlim_{f\in\gS_{\bA}}\bA_f$ with all transition morphisms and limiting morphisms defined as inclusion mappings (they are thus all monomorphisms).
This enabled us to prove that $\Bool_P^{<\go}$ is $\go$-dense in~$\Bool_P$\,; in particular, \emph{$\Bool_P$ is an $\go$-accessible}%
\footnote{
As established in \cite[Proposition~2.3.2]{Larder}, the category~$\Bool_P$ has also all binary products.
On the other hand, it may not have binary coproducts, in which case it is not locally presentable.
}
\emph{category}.
By Ad\'amek and Rosick\'y \cite[Theorem~2.11]{AdRo94}, it thus follows that \emph{$\Bool_P$ is a $\gl$-accessible category, for every infinite regular cardinal~$\gl$.
In particular, $\Bool_P^{<\gl}$ is $\gl$-dense in~$\Bool_P$}\,.

Since~$\Bool_P$ is $\go$-accessible, the equivalence between the strong form of~\eqref{glpresBoolP} and~\eqref{glsmallcolim} in the following Lemma~\ref{L:glPresBoolP} is a consequence of Makkai and Par\'e \cite[Proposition~2.3.11]{MakPar1989}.
However, the following argument is direct, and it yields the additional combinatorial criterion~\eqref{gSbAcof}.

\begin{lemma}\label{L:glPresBoolP}
The following are equivalent, for any infinite regular cardinal~$\gl$ and any $P$-scaled Boolean algebra~$\bA$:
\begin{enumeratei}
\item\label{glpresBoolP}
$\bA$ is $\gl$-pre\-sentable \pup{resp., weakly $\gl$-pre\-sentable}.

\item\label{glsmallcolim}
$\bA$ is a directed colimit of a $\gl$-small direct system in $\Bool_P^{<\go}$\,.

\item\label{gSbAcof}
The poset $(\gS_{\bA},\sqsubseteq)$ has a $\gl$-small cofinal subset.

\end{enumeratei}
\end{lemma}

\begin{proof}
\eqref{glpresBoolP}$\Rightarrow$\eqref{glsmallcolim}.
Let~$\bA$ be weakly $\gl$-pre\-sentable.
The set~$\cI$ of all nonempty $\gl$-small, $\sqsubseteq$-directed subsets of~$\gS_{\bA}$ is $\gl$-directed under set inclusion.
We set $\bA(I)\eqdef\varinjlim_{f\in I}\bA_f$ whenever $I\in\cI$.
Since all directed colimits in~$\Bool_P$ are standard, every limiting morphism $\ga_I\colon\bA(I)\to\bA$ is monic (cf. Lemma~\ref{L:BoolPmono}).
{F}rom $\bA=\varinjlim_{f\in\gS_{\bA}}\bA_f$ it follows that $\bA=\varinjlim_{I\in\cI}\bA(I)$, thus, since~$\bA$ is weakly $\gl$-pre\-sentable and~$\cI$ is $\gl$-directed, there are $I\in\cI$ and $\xi\colon\bA\to\bA(I)$ such that $\id_{\bA}=\ga_I\circ\xi$.
Since~$\ga_I$ is monic, it follows that~$\ga_I$ and~$\xi$ are mutually inverse isomorphisms; whence $\bA\cong\bA(I)=\varinjlim_{f\in I}\bA_f$\,.

\eqref{glsmallcolim}$\Rightarrow$\eqref{gSbAcof}.
We are given a $\gl$-small directed colimit cocone $\vecm{\bA,\gb_i}{i\in I}=\varinjlim\vecm{\bB_i,\gb_i^j}{i\leq j\text{ in }I}$ with all~$\bB_i$ finitely pre\-sentable.
For each $i\in I$, we consider the finite Boolean subalgebra $V_i\eqdef\gb_i[B_i]$ of~$A$.
Any atom~$v$ of~$V_i$ can be written as $v=\gb_i(x)$ for a unique atom~$x$ of~$B_i$\,, necessarily outside $\gb_i^{-1}\set{0}$\,.
Since~$B_i$ is finitely pre\-sentable, there is a largest $p\in P$ such that $x\in B_i^{(p)}$; denote it by~$g_i(v)$.
Now let $f\in\gS_{\bA}$ and set $U\eqdef A_f$\,.
Since directed colimits in~$\Bool_P$ are standard, there is $i\in I$ such that $U\subseteq V_i$ and $u\in\gb_i\rI{B_i^{(f(u))}}$ whenever $u\in\At{U}$.
Let $(u,v)\in(\At{U})\times(\At{V_i})$ such that $v\leq u$.
Write $v=\gb_i(x)$ where $x\in(\At{B_i})$\,.
Since $\gb_i(x)=v\leq u\in\gb_i\rI{B_i^{(f(u))}}$ and~$x$ is an atom of~$B_i$ outside~$\gb_i^{-1}\set{0}$\,, we get $x\in B_i^{(f(u))}$, that is, $f(u)\leq g_i(v)$; whence $f\sqsubseteq g_i$\,.
We have thus proved that $\setm{g_i}{i\in I}$ is cofinal in~$\gS_{\bA}$\,.

\eqref{gSbAcof}$\Rightarrow$\eqref{glpresBoolP}.
Since $\bA=\varinjlim_{f\in\gS_{\bA}}\bA_f$\,, it follows from our assumption that~$\bA$ is a $\gl$-small directed colimit of finitely pre\-sentable $P$-scaled Boolean algebras.
By Ad\'amek and Rosick\'y \cite[Proposition~1.16]{AdRo94}, it follows that~$\bA$ is $\gl$-pre\-sentable.
\end{proof}

\begin{corollary}\label{C:glPresBoolP}
For any $P$-scaled Boolean algebra~$\bA$,
 \[
 \card{A}+\card{P}<\gl\Rightarrow
 (\bA\text{ is }\gl\text{-presentable}\,)
 \Rightarrow\card{A}<\gl\,.
 \]
\end{corollary}

\begin{proof}
The first implication follows from the relation $\bA=\varinjlim_{f\in\gS_{\bA}}\bA_f$\,, together with the inequality $\card\gS_{\bA}\leq\card{A}+\card{P}$ in case~$\gS_{\bA}$ is infinite.
The second implication follows from the criterion stated in Lemma~\ref{L:glPresBoolP}\eqref{glsmallcolim}.
\end{proof}

\begin{note}
Easy examples show that none of the implications of Corollary~\ref{C:glPresBoolP} can be reversed.
\end{note}

\begin{lemma}\label{L:Approxmusmall}
Let~$\gl$ be an infinite regular cardinal.
Then every normal morphism $\gf\colon\bA\to\bB$ of $P$-scaled Boolean algebras is a $\gl$-directed colimit \pup{within the category of all arrows of $\Bool_P$} of normal morphisms in~$\Bool_P^{<\gl}$\,.
\end{lemma}

\begin{proof}
Our argument follows the lines of the one of Gillibert and Wehrung \cite[Proposition~2.5.5]{Larder}.
We may identify~$\bB$ with~$\bA/I$ and~$\gf$ with the canonical projection from~$\bA$ onto~$\bA/I$, where $I\eqdef\gf^{-1}\set{0}$.
If~$\bA$ is $\gl$-pre\-sentable, then, by Lemma~\ref{L:glPresBoolP}, there is a $\gl$-small directed colimit representation
 \begin{equation}\label{Eq:bAlimbAj}
 \vecm{\bA,\ga_j}{j\in J}=
 \varinjlim\vecm{\bA_j,\ga_j^k}{j\leq k\text{ in }J}
 \end{equation}
where each~$\bA_j$ is finitely pre\-sentable.
Each $I_j\eqdef\ga_j^{-1}[I]$ is an ideal of~$A_j$ (necessarily principal since~$A_j$ is finite), each~$\bA_j/{I_j}$ is finitely pre\-sentable, and $\bA/I=\varinjlim_{j\in J}(\bA_j/{I_j})$ is thus $\gl$-pre\-sentable.

In the general case, there is a representation of the form~\eqref{Eq:bAlimbAj}, now with~$J$ $\gl$-directed and each~$\bA_j$ $\gl$-pre\-sentable.
This yields a directed colimit representation $\gf=\varinjlim_{j\in J}\gf_j$ where~$\gf_j$ denotes the canonical projection from~$\bA_j$ onto~$\bA_j/{\ga_j^{-1}[I]}$, which is (using the result of the paragraph above) a normal morphism between $\gl$-pre\-sentable objects.
\end{proof}

\section{Condensates}
\label{S:CondPhiCond}

In this section we shall introduce the crucial constructs~$\bA\bt\vec{S}$ (box condensates) and $\bA\otlF\vec{S}$ ($\Phi$-condensates).
Throughout Section~\ref{S:CondPhiCond} we shall fix a poset~$P$, a category~$\cS$, and a \pup{not necessarily com\-mu\-ta\-tive} $P$-indexed diagram $\vec{S}=\vecm{S_p,\vec{S}(p,q)}{p\leq q\text{ in }P}$ in~$\cS$.

\subsection{Box condensates and constricted morphisms}\label{Su:Cond}

\begin{notation}\label{Not:UltbA}
For any $P$-scaled Boolean algebra~$\bA$, we set
 \[
 \Ultb\bA\eqdef\setm{\fa\in\Ult{A}}
 {\|\fa\|_{\bA}\text{ has a least upper bound in }P}\,.
 \]
For any ultrafilter~$\fa$ of~$A$, we denote by~$|\fa|_{\bA}$ the least upper bound of~$\|\fa\|_{\bA}$ in~$P$ if it exists.
\end{notation}

\begin{definition}\label{D:Constricted}
A morphism $\gf\colon\bA\to\bB$ in~$\Bool_P$ is \emph{constricted} if $\gf^{-1}[\fb]$ belongs to~$\Ultb\bA$ whenever $\fb\in\Ultb\bB$.
That is, if~$\|\fb\|_{\bB}$ has a least upper bound in~$P$, then so does $\|\gf^{-1}[\fb]\|_{\bA}$\,, whenever $\fb\in\Ult B$.

We denote by~$\Boolc_P$ the category of all $P$-scaled Boolean algebras with constricted morphisms.
\end{definition}

Observe that in the context of Definition~\ref{D:Constricted}, the containment $\|\gf^{-1}[\fb]\|_{\bA}\subseteq\|\fb\|_{\bB}$ holds (because~$\gf$ is a morphism in~$\Bool_P$); whence $|\gf^{-1}[\fb]|_{\bA}\leq|\fb|_{\bB}$ if both sides of that inequality are defined.

\begin{definition}\label{D:otoS}
For any $P$-scaled Boolean algebra~$\bA$, we set
 \[
 \bA\bt\vec{S}\eqdef\prod
 \Vecm{S_{|\fa|_{\bA}}}{\fa\in\Ultb\bA}\quad
 \text{if the product exists}.
 \]
We will say that~$\bA\bt\vec{S}$ is a \emph{box condensate} of~$\vec{S}$.
\end{definition}

In particular, if~$\bA$ is finitely presentable, then $\bA\bt\vec{S}$ is identical to the construct (called there a condensate) $\bA\otimes\vec{S}$ introduced in Gillibert and Wehrung~\cite{Larder}.
In that case (i.e., $\bA$ is finitely presentable), $\Ultb\bA=\Ult A$ is the set of all principal ultrafilters associated to the atoms of~$A$.

\begin{definition}\label{D:gfotlvecS}
For any constricted morphism $\gf\colon\bA\to\bB$ in~$\Bool_P$\,,
if $\bA\bt\vec{S}$ and $\bB\bt\vec{S}$ both exist, then we define $\gf\bt\vec{S}$ as the set of all morphisms $f\colon\bA\bt\vec{S}\to\bB\bt\vec{S}$ in~$\cS$ of the form
 \[
 f=\prod\vecm{f_{\fb}}{\fb\in\Ultb\bB}\,,
 \qquad\text{where each }
 f_{\fb}\in\vec{S}\pI{|\gf^{-1}[\fb]|_{\bA},|\fb|_{\bB}}\,.
 \]
\end{definition}

This means that~$f$ is the unique morphism from~$\bA\bt\vec{S}$ to~$\bB\bt\vec{S}$ making the diagram represented in Figure~\ref{Fig:gfbtS} commute whenever $\fb\in\Ultb\bB$.
In that diagram, $\gd^{\bB}_{\fb}$ denotes the canonical projection of $\bB\bt\vec{S}$ onto~$S_{|\fb|_{\bB}}$\,.

\begin{figure}[htb]
\begin{tikzcd}
\centering
\bA\bt\vec{S}\arrow[r,"f"]
\arrow[d,"\gd^{\bA}_{\gf^{-1}[\fb]}"'] &
\bB\bt\vec{S}\arrow[d,"\gd^{\bB}_{\fb}"]\\
S_{|\gf^{-1}[\fb]|_{\bA}}
\arrow[r,"f_{\fb}"] &
S_{|\fb|_{\bB}}
\end{tikzcd}
\caption{A morphism~$f$ in $\gf\bt\vec{S}$}
\label{Fig:gfbtS}
\end{figure}

This is to be put in contrast with Gillibert and Wehrung \cite[\S~3.1]{Larder}, where $\gf\otimes\vec{S}$ is defined as a single morphism from~$\bA\otimes\vec{S}$ to $\bB\otimes\vec{S}$.

\begin{lemma}\label{L:otlDiagr}
The assignment ${}_{-}\bt\vec{S}$ defines a diagram in~$\cS$ \pup{in the sense of Definition~\textup{\ref{D:Diagram}}}, indexed by the full subcategory~$\Boolc_P(\vec{S})$ of~$\Boolc_P$ consisting of all~$\bA$ such that $\bA\bt\vec{S}$ exists.
\end{lemma}

\begin{proof}
It is trivial that $\id_{\bA}\bt\vec{S}$ contains, as an element, the identity on $\bA\bt\nobreak\vec{S}$.
Now let $\gf\colon\bA\to\bB$ and $\psi\colon\bB\to\bC$ be morphisms in $\Boolc_P(\vec{S})$, let $f\in\gf\bt\vec{S}$, and let $g\in\psi\bt\vec{S}$.
We need to verify that $g\circ f$ belongs to $(\psi\circ\gf)\bt\nobreak\vec{S}$.
Let
 \[
 f=\prod\vecm{f_{\fb}}{\fb\in\Ultb\bB}\quad\text{and}\quad
 g=\prod\vecm{g_{\fc}}{\fc\in\Ultb\bC}\,,
 \]
where each $f_{\fb}\in\vec{S}\pI{|\gf^{-1}[\fb]|_{\bA},|\fb|_{\bB}}$ and each $g_{\fc}\in\vec{S}\pI{|\psi^{-1}[\fc]|_{\bB},|\fc|_{\bC}}$.
Now $g\circ f=\prod\vecm{g_{\fc}\circ f_{\psi^{-1}[\fc]}}{\fc\in\Ultb\bC}$ where each $g_{\fc}\circ f_{\psi^{-1}[\fc]}$ belongs to $\vec{S}(|\gf^{-1}\psi^{-1}[\fc]|_{\bA},|\fc|_{\bC})=\vec{S}(|(\psi\circ\gf)^{-1}[\fc]|_{\bA},|\fc|_{\bC})$, as required.
\end{proof}

The usually unwieldy requirement that a morphism be constricted will often be ensured by the following convenient condition.

\begin{definition}\label{D:CondJoinCplt}
Let~$\gl$ be an infinite cardinal.
We say that the poset~$P$ is a \emph{conditional $\gl$-DCPO}%
\footnote{
Recall that \emph{DCPO} usually stands for ``directed-com\-plete partial order''.
}
if every nonempty, bounded above, directed, $\gl$-small subset of~$P$ has a least upper bound in~$P$.
If this holds for every~$\gl$ then we say that~$P$ is a \emph{conditional DCPO}.
\end{definition}

Observe that every poset is, trivially, a conditional $\go$-DCPO.

Throughout the paper we will be constantly using the following observation.

\begin{lemma}\label{L:Constr2easy}
The following statements hold, for any infinite regular cardinal~$\gl$, any conditional $\gl$-DCPO~$P$, and any $P$-scaled Boolean algebras~$\bA$ and~$\bB$:
\begin{enumerater}
\item\label{altUltb}
If~$\bA$ is $\gl$-pre\-sentable, then
$\Ultb\bA=\setm{\fa\in\Ult A}{\|\fa\|_{\bA}\text{ is bounded above in }P}$.

\item\label{AllConstr}
Every morphism $\gf\colon\bA\to\bB$ in~$\Bool_P$\,, with $\bA\in\Bool_P^{<\gl}$\,, is constricted.

\end{enumerater}
\end{lemma}

\begin{proof}
\emph{Ad}~\eqref{altUltb}.
By Lemma~\ref{L:glPresBoolP}, there is a $\gl$-small directed colimit representation
 \[
 \vecm{\bA,\ga_i}{i\in I}=
 \varinjlim\vecm{\bA_i,\ga_i^j}{i\leq j\text{ in }I}
 \]
with each~$\bA_i$ finitely pre\-sentable.
By standardness of the colimit, it follows that for every ultrafilter~$\fa$ of~$A$, the equality
$\|\fa\|_{\bA}=\bigcup\vecm{\|\ga_i^{-1}[\fa]\|_{\bA_i}}{i\in I}$ holds.
Since each~$\bA_i$ is finitely pre\-sentable, each $\|\ga_i^{-1}[\fa]\|_{\bA_i}$ has a largest element~$p_i$\,.
Hence $\|\fa\|_{\bA}$ has the $\gl$-small cofinal subset $\setm{p_i}{i\in I}$.
In particular, if $\|\fa\|_{\bA}$ is bounded above, then, since it is directed and since~$P$ is a conditional $\gl$-DCPO, it has a least upper bound.

\emph{Ad}~\eqref{AllConstr}.
Since~$\gf$ is a morphism in~$\Bool_P$\,, $\|\gf^{-1}[\fb]\|_{\bA}$ is contained in~$\|\fb\|_{\bB}$ for every ultrafilter~$\fb$ of~$B$.
In particular, if $\|\fb\|_{\bB}$ has a least upper bound, then it is bounded above, thus so is $\|\gf^{-1}[\fb]\|_{\bA}$\,.
Apply~\eqref{altUltb}.
\end{proof}

\begin{remark}\label{Rk:GenBoxCond}
Box condensates can be defined in more general settings than $P$-scaled Boolean algebras.
For any set~$I$, any map $a\colon I\to P$, and any $P$-indexed diagram~$\vec{S}$, one can define the box condensate $a\bt\vec{S}\eqdef\prod\vecm{S_{a(i)}}{i\in I}$ if the product exists.
The construction $\bA\bt\vec{S}$ of Definition~\ref{D:otoS} is then a particular case of that construction (let $I=\Ultb\bA$ and $a(i)=|i|_{\bA}$ whenever $i\in I$).
The construction $\gf\bt\vec{S}$ of Definition~\ref{D:gfotlvecS} can then be extended likewise.
A morphism $\gf\colon a\to b$, where $a\colon I\to P$ and $b\colon J\to P$, is now a map $\gf\colon J\to I$ such that $a\circ\gf\leq b$ (i.e., $a\gf(j)\leq b(j)$ whenever $j\in J$);
so it lives in what could be called an ``ordered slice category'' over~$P$.
One can then define the diagram ${}_{-}\bt\vec{S}$ as in Lemma~\ref{L:otlDiagr}, with $\gf\bt\vec{S}$ defined as the set of all morphisms of the form $\prod\vecm{f_j}{j\in J}$ where each $f_j\in\vec{S}(a\gf(j),b(j))$.
Following the method outlined in Remark~\ref{Rk:Diagram}, this diagram can then be encoded by a commutative diagram indexed by a category ``covering'' the ordered slice category outlined above.

Despite the greater level of generality brought by that approach, all our applications developed in Sections~\ref{S:Ceva}--\ref{S:4SCML}, and in fact all the applications we are currently aware of, require the passage through the original box condensates from Definition~\ref{D:otoS}.
We thus chose to keep the apparently less general approach through the paper.
\end{remark}

\subsection{$\Phi$-condensates}\label{Su:PhiCond}
Let~$\cT$ be a category and let~$\Phi\colon\cS\to\cT$ be a functor.
We want to argue that under certain conditions, the composition of the diagram ${}_{-}\bt\vec{S}$ with~$\Phi$ yields a \emph{com\-mu\-ta\-tive} diagram.
This condition arises from the construction~$\vec{S}^I$ introduced in Definition~\ref{D:D^I}.

\begin{definition}\label{D:PhiComm}
We say that a $P$-indexed diagram~$\vec{S}$ in~$\cS$ is \emph{$\Phi$-com\-mu\-ta\-tive} if the composition $\Phi\vec{S}^I$ is a com\-mu\-ta\-tive diagram for any set~$I$ such that~$\vec{S}^I$ is defined.

This means that for any $p=\vecm{p_i}{i\in I}$ and $q=\vecm{q_i}{i\in I}$ in~$P^I$ such that $p\leq q$, and any $f,g\in\prod\Vecm{\vec{S}(p_i,q_i)}{i\in I}$, the equality $\Phi(f)=\Phi(g)$ holds.
\end{definition}

Taking~$I$ a singleton, it follows that if the diagram~$\vec{S}$ is $\Phi$-com\-mu\-ta\-tive, then the composite~$\Phi\vec{S}$ is a com\-mu\-ta\-tive diagram.

The diagram of Abelian \lgrp{s} denoted by~$\vec{A}$ in Wehrung~\cite{Ceva} (cf. Section~\ref{S:Ceva}) is proven in that paper to be $\Idc$-com\-mu\-ta\-tive, where~$\Idc$ is the functor sending every Abelian \lgrp\ to its lattice of principal $\ell$-ideals.
However, $\vec{A}$ is not com\-mu\-ta\-tive.

\begin{lemma}\label{L:PhiComm}
Let~$\vec{S}$ be a $\Phi$-com\-mu\-ta\-tive $P$-indexed diagram in~$\cS$.
Then the composite $\Phi\pI{{}_{-}\bt\vec{S}}$ is a functor from~$\Boolc_P(\vec{S})$ to~$\cT$.
\end{lemma}

\begin{proof}
We need to prove that for any $P$-scaled Boolean algebras~$\bA$ and~$\bB$, any constricted morphism $\gf\colon\bA\to\bB$, and any $f,g\in\prod\Vecm{\vec{S}\pI{|\gf^{-1}[\fb]|_{\bA},|\fb|_{\bB}}}{\fb\in\Ultb\bB}$, the equality
$\Phi(f)=\Phi(g)$ holds.
This follows trivially from the commutativity of the diagram $\Phi\vec{S}^{\Ultb\bB}$.
\end{proof}

Let us introduce a context that facilitates dealing with the category $\Boolc_P(\vec{S})$.

\begin{definition}\label{D:PRODgl}
Let~$\gl$ be an infinite cardinal.
We say that the category~$\cS$ \emph{has all $2^{<\gl}$-products} if it has all binary products and all products of at most~$2^{\ga}$ objects whenever $\go\leq\ga<\gl$.
(\emph{In particular, if $\gl>\go$, then~$\cS$ has a terminal object}.)
\end{definition}

\begin{lemma}\label{L:PRODgl}
Let~$\gl$ be an infinite regular cardinal.
Suppose that~$P$ is a conditional $\gl$-DCPO and that~$\cS$ has all $2^{<\gl}$-products.
Then $\bA\bt\vec{S}$ exists for every $\gl$-pre\-sentable~$\bA$, and every morphism in $\Bool_P^{<\gl}$ is constricted.
\end{lemma}

\begin{proof}
Let~$\bA\in\Bool_P^{<\gl}$ and set $\ga\eqdef\card A$.
By Corollary~\ref{C:glPresBoolP}, $\ga<\gl$.
Since the set $\Ultb\bA$ has at most~$2^{\ga}$ elements and since it is nonempty if $\gl=\go$ (for in that case~$\bA$ is finitely pre\-sented, so $\Ultb\bA=\Ult A$), $\bA\bt\vec{S}$ exists in~$\cS$.
The last part of the statement of Lemma~\ref{L:PRODgl} follows immediately from Lemma~\ref{L:Constr2easy}.
\end{proof}

\begin{lemma}\label{L:otmFexists}
Let~$\gl$ be an infinite regular cardinal.
Suppose that~$P$ is a conditional $\gl$-DCPO, $\cS$ has all $2^{<\gl}$-products, $\cT$ has all $\gl$-directed colimits, and~$\vec{S}$ is $\Phi$-com\-mu\-ta\-tive.
Then the restriction of the functor $\Phi\pI{{}_{-}\bt\vec{S}}$ to~$\Bool_P^{<\gl}$ extends uniquely, up to natural isomorphism, to a $\gl$-continuous functor $\ol{\Psi}\colon\Bool_P^{<\gl}\to\nobreak\cT$.
\end{lemma}

\begin{proof}
We apply Lemma~\ref{L:ExtFunctgkPres} to $\cA:=\Bool_P$ and $\cA^{\dagger}:=\Bool_P^{<\gl}$\,, with~$\Psi$ defined as the restriction of the functor $\Phi\pI{{}_{-}\bt\vec{S}}$ to~$\Bool_P^{<\gl}$ (here and at many other places we apply Lemma~\ref{L:PRODgl}).
Since $\Bool_P^{<\gl}$ is $\gl$-dense in~$\Bool_P$\,, the assumptions of Lemma~\ref{L:ExtFunctgkPres} are indeed satisfied.
\end{proof}

\begin{definition}\label{D:AotmFS}
We shall denote by ${}_{-}\otlF\vec{S}$ the functor~$\ol{\Psi}$ whose existence is ensured by Lemma~\ref{L:otmFexists}.
Objects of~$\cS$ of the form $\bA\otlF\vec{S}$, for $\bA\in\Bool_P$\,, will be called \emph{$\Phi$-condensates} of~$\vec{S}$.
\end{definition}

We end this section by recalling a useful example of a $P$-scaled Boolean algebra, introduced in Gillibert and Wehrung \cite[Definition~2.6.1]{Larder}; namely,
 \[
 \two[p]\eqdef\pI{\two,\vecm{\two[p]^{(q)}}{q\in P}}
 \]
where we set~$\two\eqdef\set{0,1}$ and define $\two[p]^{(q)}$ as $\set{0,1}$ if~$q\leq p$, $\set{0}$ otherwise.
Whenever $p\leq q$ in~$P$, the identity map on~$\two$ induces a monomorphism $\eps_p^q\colon\two[p]\rightarrowtail\two[q]$ in~$\Bool_P^{<\go}$\,.
The following straightforward analogue of \cite[Lemma~3.1.3]{Larder} remains valid.

\begin{lemma}\label{L:2potmvecS}
Let~$\vec{S}$ be a \pup{not necessarily com\-mu\-ta\-tive} $P$-indexed diagram in~$\cS$.
Then the following statements hold:
\begin{enumerater}
\item
$\two[p]\bt\vec{S}=S_p$\,, for all $p\in P$.

\item
$\eps_p^q\bt\vec{S}=\vec{S}(p,q)$, for all $p\leq q$ in~$P$.
\end{enumerater}
\end{lemma}

In particular, in the context of Lemma~\ref{L:2potmvecS}, we get
\begin{itemize}
\item
For all $p\in P$, $\two[p]\otlF\vec{S}=\Phi(S_p)$.

\item
For all $p\leq q$ in~$P$, $\eps_p^q\otlF\vec{S}$ is the common value of all~$\Phi(x)$ for $x\in\vec{S}(p,q)$.
\end{itemize}

\section{The Uniformization Lemma and the Boosting Lemma}\label{S:Forests}

While the construction~${}_{-}\bt\vec{S}$, originally defined on constricted morphisms, cannot be functorially extended to morphisms in full generality, we find, in this section, conditions under which such an extension is possible on diagrams of $P$-scaled Boolean algebras indexed by \emph{forests}.
We begin with an easy lemma.

\begin{lemma}\label{L:ChainInitSeg}
Let~$I$ be a chain, let~$P$ be a poset, and let $f\colon I\to P$ be an isotone map.
Then $f[I]\dnw f(i)=f[I\dnw i]$ for every $i\in I$.
\end{lemma}

\begin{proof}
The containment $f[I\dnw i]\subseteq f[I]\dnw f(i)$ follows from the isotonicity of~$f$.
Conversely, every $p\in f[I]\dnw f(i)$ can be written as~$f(x)$ for some $x\in I$.
Since~$I$ is a chain, $x\leq i$ or $i\leq x$.
In the former case, $x\in I\dnw i$ thus $p=f(x)\in f[I\dnw i]$.
In the latter case, $p\leq f(i)\leq f(x)=p$, so $p=f(i)\in f[I\dnw i]$.
\end{proof}

The main result of Section~\ref{S:Forests} states that forest-indexed diagrams of box condensates of any diagram of constricted morphisms can always be uniformized (cf. Definition~\ref{D:Diagram} for the latter).

\begin{lemma}[Uniformization Lemma]\label{L:UnifLem}
Let~$P$ be a poset in which every bounded chain is finite, let~$\gL$ be a forest, let~$\vecm{\bA_i,\ga_i^j}{i\leq j\text{ in }\gL}$ be a $\gL$-indexed commutative diagram in~$\Bool_P$\,, let~$\cS$ be a category, and let~$\vec{S}$ be a $P$-indexed diagram in~$\cS$ such that each $\bA_i\bt\vec{S}$ exists.
Then every morphism in~$\Bool_P$ is constricted and the diagram $\Vecm{\bA_i\bt\vec{S},\ga_i^j\bt\vec{S}}{i\leq j\text{ in }I}$ has a uniformization.
\end{lemma}

\begin{proof}
Since~$P$ has no bounded strictly ascending sequence, it is a conditional DCPO, thus every morphism in~$\Bool_P$ is constricted (cf. Lemma~\ref{L:Constr2easy}).
For each pair $(p,q)\in P\times P$ such that $p<q$, we pick a morphism $\gs_p^q\in\vec{S}(p,q)$.
We emphasize that the~$\gs_p^q$ are not expected to form a commutative diagram in~$\cS$ (which indeed they usually do not).

Set $\cU\eqdef\bigcup\Vecm{\set{k}\times(\Ultb{\bA_k})}{k\in\gL}$.
For each $(k,\fc)\in\cU$, denote by\linebreak $\pi_{k,\fc}\colon\gL\dnw k\to P$ the map defined by
$\pi_{k,\fc}(j)\eqdef|(\ga_j^k)^{-1}[\fc]|_{\bA_j}$ whenever $j\in\gL\dnw k$ (since~$\ga_j^k$ is constricted, the map~$\pi_{k,\fc}$ is indeed well defined).
Observe that $\pi_{k,\fc}(k)=|\fc|_{\bA_k}$.
For all $i\leq j\leq k$ in~$\gL$,
 \[
 \pi_{k,\fc}(i)=|(\ga_i^k)^{-1}[\fc]|_{\bA_i}=
 |(\ga_i^j)^{-1}(\ga_j^k)^{-1}[\fc]|_{\bA_i}\leq|(\ga_j^k)^{-1}[\fc]|_{\bA_j}
 =\pi_{k,\fc}(j)\,,
 \]
so~$\pi_{k,\fc}$ is isotone.
It follows that the range~$P_{k,\fc}$ of~$\pi_{k,\fc}$ is a chain contained in~$P$.

\begin{sclaim}
Let $j\leq k$.
Then $P_{k,\fc}\dnw\pi_{k,\fc}(j)=P_{j,(\ga_j^k)^{-1}[\fc]}$\,.
\end{sclaim}

\begin{scproof}
A direct calculation:
 \begin{align*}
 P_{j,(\ga_j^k)^{-1}[\fc]}&=
 \setm{|(\ga_i^j)^{-1}(\ga_j^k)^{-1}[\fc]|_{\bA_i}}{i\in\gL\dnw j}\\
 &=\setm{|(\ga_i^k)^{-1}[\fc]|_{\bA_i}}{i\in\gL\dnw j}\\
 &=\pi_{k,\fc}[\gL\dnw j]\\
 &=\pi_{k,\fc}[(\gL\dnw k)\dnw j]\\
 &=\pi_{k,\fc}[\gL\dnw k]\dnw\pi_{k,\fc}(j)
 \qquad(\text{because }\gL\dnw k
 \text{ is a chain and by Lemma~\ref{L:ChainInitSeg}})\\
 &=P_{k,\fc}\dnw\pi_{k,\fc}(j)\,.\tag*{\qedsc}
 \end{align*}
\renewcommand{\qedsc}{}
\end{scproof}

Whenever $(k,\fc)\in\cU$ and $j\leq k$, it follows from the assumptions on~$P$ that the subchain $P_{k,\fc}\upw\pi_{k,\fc}(j)=P_{k,\fc}\cap[\pi_{k,\fc}(j),\pi_{k,\fc}(k)]$ is finite; write it as $\set{p_0,\dots,p_n}$ with $p_0<\cdots<p_n$\,, and then define
 \[
 \tau_{j,k,\fc}\eqdef
 \gs_{p_{n-1}}^{p_n}\circ\cdots\circ\gs_{p_0}^{p_1}\,.
 \]
In particular, $\tau_{j,k,\fc}$ belongs to $\vec{S}(p_0,p_n)=\vec{S}(|(\ga_j^k)^{-1}[\fc]|_{\bA_j},|\fc|_{\bA_k})$\,.
It follows that the morphisms~$\tau_{j,k,\fc}$\,, for~$\fc$ ranging over~$\Ultb{\bA_k}$\,, are the components of a morphism
 \[
 \gf_j^k\eqdef\prod\vecm{\tau_{j,k,\fc}}{\fc\in\Ultb{\bA_k}}
 \in\ga_j^k\bt\vec{S}\,.
 \]
If $j=k$, then each~$\tau_{j,k,\fc}$ is the identity, thus~$\gf_j^k$ is the identity on~$\bA_j\bt\vec{S}$.

We shall now prove that $\gf_i^k=\gf_j^k\circ\gf_i^j$ whenever $i\leq j\leq k$ in~$\gL$.
This amounts to proving that the relation $\tau_{i,k,\fc}=\tau_{j,k,\fc}\circ\tau_{i,j,(\ga_j^k)^{-1}[\fc]}$ holds for every $\fc\in\Ultb\bA_k$\,.
Writing $P_{k,\fc}\upw\pi_{k,\fc}(i)=\set{p_0,\dots,p_n}$ with $p_0<\cdots<p_n$\,, we obtain, by definition,
 \begin{equation}\label{Eq:tauikfc}
 \tau_{i,k,\fc}=\gs_{p_{n-1}}^{p_n}\circ\cdots\circ\gs_{p_0}^{p_1}\,.
 \end{equation}
Let $m\in\set{0,\dots,n}$ such that $\pi_{k,\fc}(j)=p_m$\,.
Then $P_{k,\fc}\upw\pi_{k,\fc}(j)=\set{p_m,\dots,p_n}$\,, thus
 \begin{equation}\label{Eq:taujkfc}
 \tau_{j,k,\fc}=
 \gs_{p_{n-1}}^{p_n}\circ\cdots\circ\gs_{p_m}^{p_{m+1}}\,.
 \end{equation}
Finally, it follows from the Claim above that
$P_{j,(\ga_j^k)^{-1}[\fc]}=\set{p_0,\dots,p_m}$\,, thus
 \begin{equation}\label{Eq:tauijfc}
 \tau_{i,j,(\ga_j^k)^{-1}[\fc]}=
 \gs_{p_{m-1}}^{p_m}\circ\cdots\circ\gs_{p_0}^{p_1}\,.
 \end{equation}
The relations~\eqref{Eq:tauikfc}, \eqref{Eq:taujkfc}, and~\eqref{Eq:tauijfc} together imply the required relation\linebreak
$\tau_{i,k,\fc}=\tau_{j,k,\fc}\circ\tau_{i,j,(\ga_j^k)^{-1}[\fc]}$.
This concludes the proof that $\Vecm{\bA_i\bt\vec{S},\gf_i^j}{i\leq j\text{ in }I}$ is a uniformization of $\Vecm{\bA_i\bt\vec{S},\ga_i^j\bt\vec{S}}{i\leq j\text{ in }I}$.
\end{proof}

Under conditions on~$\gl$ and~$P$ and for a $\gl$-pre\-sentable $P$-scaled Boolean algebra~$\bB$, the $\Phi$-condensate $\bB\otlF\vec{S}$ is defined as $\Phi(\bB\bt\vec{S})$, thus it belongs to the range of~$\Phi$.
The following result enables us to extend the latter observation to the case where~$\bB$ is the colimit of a direct system, indexed by~$\gl$, of $\gl$-pre\-sentable $P$-scaled Boolean algebras.

\begin{lemma}[The Boosting Lemma]\label{L:Boosting}
Let~$\gl$ be an infinite regular cardinal and let~$P$ be a poset in which every bounded chain is finite.
Let~$\cS$ be a category with all $2^{<\gl}$-products \pup{cf. Definition~\textup{\ref{D:PRODgl}}} and all colimits indexed by~$\gl$, let $\cT$ be a category with all $\gl$-directed colimits, and let $\Phi\colon\cS\to\cT$ be a functor preserving all colimits indexed by~$\gl$.
Let~$\vec{S}$ be a $\Phi$-com\-mu\-ta\-tive $P$-indexed diagram in~$\cS$, and consider a directed colimit cocone
 \[
 \vecm{\bB,\beta_{\xi}}{\xi<\gl}=
 \varinjlim\vecm{\bB_{\xi},\beta_{\xi}^{\eta}}{\xi\leq\eta<\gl}
 \qquad\text{in }\Bool_P\,,
 \]
with all~$\bB_{\xi}$ $\gl$-pre\-sentable.
Then $\bB\otlF\vec{S}$ belongs to the range of~$\Phi$.
\end{lemma}

\begin{proof}
Our assumptions ensure that all box condensates~$\bB_{\xi}\bt\vec{S}$ exist and all the morphisms~$\gb_{\xi}^{\eta}$ are constricted (cf. Lemmas~\ref{L:Constr2easy} and~\ref{L:PRODgl}).
By Lemma~\ref{L:UnifLem}, the diagram
$\Vecm{\bB_{\xi}\bt\vec{S},\beta_{\xi}^{\eta}\bt\vec{S}}{\xi\leq\eta<\gl}$ has a uniformization
$\Vecm{\bB_{\xi}\bt\vec{S},\gf_{\xi}^{\eta}}{\xi\leq\eta<\gl}$.
Sin\-ce~$\cS$ has all colimits indexed by~$\gl$, there exists a directed colimit cocone
 \[
 \vecm{S,\gf_{\xi}}{\xi<\gl}=\varinjlim
 \vecm{\bB_{\xi}\bt\vec{S},\gf_{\xi}^{\eta}}{\xi\leq\eta<\gl}
 \qquad\text{in }\cS\,.
 \]
(\emph{Of course, $S$ will play the role of the undefined $\bB\bt\vec{S}$}.)
Since~$\Phi$ preserves all colimits indexed by~$\gl$, it follows that
 \begin{equation}\label{Eq:PhiSascolim}
 \vecm{\Phi(S),\Phi(\gf_{\xi})}{\xi<\gl}=\varinjlim
 \vecm{\bB_{\xi}\otlF\vec{S},\Phi(\gf_{\xi}^{\eta})}{\xi\leq\eta<\gl}
 \qquad\text{in }\cT\,.
 \end{equation}
Since~$\cT$ has all $\gl$-directed colimits, $\bB\otlF\vec{S}$ is well defined.
Since the functor ${}_{-}\otlF\vec{S}$ is $\gl$-continuous, it also preserves all colimits indexed by~$\gl$, whence
 \begin{equation}\label{Eq:bAoFSascolim}
 \vecm{\bB\otlF\vec{S},\beta_{\xi}\otlF\vec{S}}{\xi<\gl}=\varinjlim
 \vecm{\bB_{\xi}\otlF\vec{S},\beta_{\xi}^{\eta}\otlF\vec{S}}
 {\xi\leq\eta<\gl}
 \qquad\text{in }\cT\,.
 \end{equation}
Now whenever $\xi\leq\eta<\gl$, it follows from the relation $\gf_{\xi}^{\eta}\in\beta_{\xi}^{\eta}\bt\vec{S}$ that $\Phi(\gf_{\xi}^{\eta})=\beta_{\xi}^{\eta}\otlF\vec{S}$.
By~\eqref{Eq:PhiSascolim} and~\eqref{Eq:bAoFSascolim} together with the uniqueness of the directed colimit, we obtain that $\Phi(S)\cong\bB\otlF\vec{S}$, as desired.
\end{proof}

\section{Tensoring with normal morphisms}
\label{S:TensNorNeu}

\begin{all}{Standing Hypothesis}
$\gl$ is an infinite regular cardinal, $P$ is a conditional $\gl$-DCPO, $\cS$ is a category with all $2^{<\gl}$-products \pup{cf. Definition~\textup{\ref{D:PRODgl}}}, $\cT$ is a category with all $\gl$-directed colimits, $\Phi\colon\cS\to\cT$ is a functor, and~$\vec{S}$ is a $\Phi$-com\-mu\-ta\-tive $P$-indexed diagram in~$\cS$.
\end{all}

In this section we shall analyze the structure of the morphisms of the form $\gf\otlF\vec{S}$ in case~$\gf$ is a normal morphism of $P$-scaled Boolean algebras.

\goodbreak

\begin{definition}\label{D:ExtPhiProj}\hfill
\begin{enumerater}
\item
A \emph{projection} in a category is the canonical morphism from a product to one of its factors.

\item
A \emph{$\Phi$-pro\-jec\-tion} in~$\cT$ is an isomorphic copy of~$\Phi(p)$ for a projection~$p$ in~$\cS$.

\item
A \emph{$\gl$-extended $\Phi$-pro\-jec\-tion} in~$\cT$ is a $\gl$-directed colimit, within the category of all arrows of~$\cT$, of $\Phi$-pro\-jec\-tions in~$\cT$.
\end{enumerater}
\end{definition}

The following result extends Gillibert and Wehrung \cite[Proposition~3.1.2]{Larder} to our current context.

\begin{proposition}\label{P:otlFproj}
For every normal morphism $\gf\colon\bA\to\bB$ in~$\Bool_P$\,, $\gf\otlF\vec{S}$ is a $\gl$-extended $\Phi$-pro\-jec\-tion.
Moreover, if~$\bA$ is $\gl$-pre\-sentable, then $\gf\otlF\vec{S}$ is a $\Phi$-pro\-jec\-tion.
\end{proposition}

\begin{proof}
By Gillibert and Wehrung \cite[\S~2.5]{Larder}, we may assume that~$\bB=\bA/I$ and $\gf=\pi_I$ (cf. Section~\ref{S:BoolP}), for an ideal~$I$ of~$A$.

By Lemma~\ref{L:Approxmusmall}, it suffices to consider the case where~$\bA$ is $\gl$-pre\-sentable.
 \begin{align*}
 U&\eqdef\Ultb{\bA}\,,\\
 U_1&\eqdef\setm{\fa\in\Ultb{\bA}}{\fa\cap I=\es}\,,\\
 U_0&\eqdef U\setminus U_1\,. 
 \end{align*}
Then we can write
 \begin{align*}
 \bA\bt\vec{S}&=\prod\Vecm{S_{|\fa|_{\bA}}}{\fa\in U}\,,\\
 \bB\bt\vec{S}&=\prod\Vecm{S_{|\fa/I|_{\bA/I}}}{\fa\in U_1}\\
 &=\prod\Vecm{S_{|\fa|_{\bA}}}{\fa\in U_1}\,.
 \end{align*}
 The canonical projection $p\colon X\mathbin{\Pi} Y\to X$, where $X\eqdef\prod\Vecm{S_{|\fa|_{\bA}}}{\fa\in U_1}$ and $Y\eqdef\prod\Vecm{S_{|\fa|_{\bA}}}{\fa\in U_0}$, is $\prod\vecm{f_{\fa}}{\fa\in U_1}$ where each~$f_{\fa}$ is the identity on $S_{|\fa|_{\bA}}$; thus it belongs to $\gf\bt\vec{S}$.
Since~$\vec{S}$ is $\Phi$-com\-mu\-ta\-tive, it follows that $\gf\otlF\vec{S}=\Phi(p)$, a $\Phi$-pro\-jec\-tion in~$\cT$.
\end{proof}

\section{Submorphisms of norm-coverings}
\label{S:NormCov}
Let us first recall a few concepts from Gillibert and Wehrung~\cite{Larder}.
Following \cite[Definition~2.1.2]{Larder}, we say that a poset~$P$ is
\begin{itemize}
\item[---] a \emph{\pjs} if the set~$U$ of all upper bounds of any finite subset~$X$ of~$P$ is a finitely generated upper subset of~$P$; then we denote by $\Sor X$ the (finite) set of all minimal elements of~$U$;

\item[---] \emph{supported} if it is a \pjs\ and every finite subset of~$P$ is contained in a finite subset~$Y$ of~$P$ which is \emph{$\sor$-closed}, that is, $\Sor{Z}\subseteq Y$ whenever~$Z$ is a finite subset of~$Y$ (this definition is equivalent to the eponymous one introduced in Gillibert~\cite{Gill2009});

\item[---] an \emph{\ajs} if it is a \pjs\ in which every principal ideal~$P\dnw a$ is a \js.

\end{itemize}
In Gillibert and Wehrung \cite[\S~2.1]{Larder}, the non-reversible implications
 \[
 \text{\js}\Rightarrow\text{\ajs}\Rightarrow\text{supported}
 \Rightarrow\text{\pjs}
 \]
are observed.
As in~\cite{Larder}, we set $a_1\sor\cdots\sor a_n\eqdef\Sor\set{a_1,\dots,a_n}$.
Following \cite[\S~2.6]{Larder}, a \emph{norm-covering} of a poset~$P$ is a pair $(X,\partial)$ where~$X$ is a \pjs\ and $\partial\colon X\to P$ is an isotone map (sometimes denoted by~$\partial_X$ if~$X$ needs to be specified).
An ideal~$\bx$ of~$X$ is \emph{sharp} if the image~$\partial[\bx]$ has a largest element, then denoted by~$\partial\bx$.
We denote by~$\Ids{X}$ the set of all sharp ideals of~$X$, partially ordered under set inclusion.

We denote by~$\rF(X)$ the Boolean algebra defined by generators~$\tilde{u}$ (or $\tilde{u}^X$ in case~$X$ needs to be specified), where $u\in X$, and relations
 \begin{align}
 \tilde{v}&\leq\tilde{u}\,,&&\text{whenever }u\leq v\text{ in }X\,;
 \label{Eq:TildeDecr}\\
 \tilde{u}\wedge\tilde{v}&=
 \bigvee\vecm{\tilde{w}}{w\in u\sor v}\,,
 &&\text{whenever }u,v\in X\,;
 \label{Eq:TildeAntimor}\\
 1&=\bigvee\vecm{\tilde{w}}{w\in\Min{X}}\,.
 \label{Eq:Tilde01}
 \end{align}
(\emph{Since~$X$ is a \pjs, the sets $u\sor v=\Sor\set{u,v}$ and $\Min X=\Sor\es$ are both finite subsets of~$X$, thus the relations~\eqref{Eq:TildeAntimor} and~\eqref{Eq:Tilde01} both make sense}.)
Furthermore, for every $p\in P$, we denote by~$\rF(X)^{(p)}$ the ideal of~$\rF(X)$ generated by $\setm{\tilde{u}}{u\in X\,,\ p\leq\partial u}$.
The structure $\xF(X)\eqdef\pI{\rF(X),\vecm{\rF(X)^{(p)}}{p\in P}}$ is a $P$-scaled Boolean algebra (cf. \cite[Lemma~2.6.5]{Larder}).

For every sharp ideal~$\bx$ of~$X$, there is a unique morphism $\pi_{\bx}^X\colon\xF(X)\to\two[\partial\bx]$ that sends every $\tilde{u}$, where $u\in X$, to~$1$ if $u\in\bx$ and~$0$ otherwise.
This morphism is normal (cf. Gillibert and Wehrung \cite[Lemma~2.6.7]{Larder}).

As already observed in Gillibert and Wehrung~\cite{Larder}, every $\sor$-closed subset~$Y$ of~$X$ defines, by restriction of the map~$\partial$, a norm-covering of~$P$ and the inclusion map from~$Y$ into~$X$ induces a morphism $f_Y^X\colon\xF(Y)\to\xF(X)$ in~$\Bool_P$ (cf. \cite[Lemma~2.6.6]{Larder}).
We shall now extend that observation.

\begin{definition}\label{D:ProperNC}
For norm-coverings~$X$ and~$Y$ of a poset~$P$, a map $f\colon X\to Y$ is a \emph{submorphism} if~$f$ is isotone, $\partial_Xx\leq\partial_Yf(x)$ whenever $x\in X$, and for all~$n<\go$ and all $x_1,\dots,x_n\in X$ the containment $f(x_1)\sor_Y\cdots\sor_Yf(x_n)\subseteq f[x_1\sor_X\cdots\sor_Xx_n]$ holds.
\end{definition}

The latter condition can be reformulated as stating that for all $n<\go$, all $x_1,\dots,x_n\in X$, and all $y\in Y$, if each $f(x_i)\leq y$, then there exists $x\in X$ such that each $x_i\leq x$ and $f(x)\leq y$.
It obviously suffices to check this for $n=0$ and $n=2$.

\begin{lemma}\label{L:F(proper)}
Every submorphism $f\colon X\to Y$ of norm-coverings of~$P$ induces a unique morphism $\xF(f)\colon\xF(X)\to\xF(Y)$ of $P$-scaled Boolean algebras sending~$\tilde{x}^X$ to $\widetilde{f(x)}^Y$ whenever $x\in X$.
\end{lemma}

\begin{proof}
In order to verify the existence of~$\xF(f)$ at Boolean algebra level, it suffices to prove that the elements~$\widetilde{f(x)}^Y$ of~$\rF(Y)$, for $x\in X$, satisfy the defining relations of~$\rF(X)$.
Since~$f$ is isotone, $x\mapsto\widetilde{f(x)}^Y$ is antitone.
For all $x_0,x_1\in X$,
 \begin{align*}
 \widetilde{f(x_0)}^Y\cap\widetilde{f(x_1)}^Y&=
 \bigvee\Vecm{\tilde{y}^Y}{y\in f(x_0)\sor_Yf(x_1)}\\
 &\leq\bigvee\Vecm{\widetilde{f(x)}^Y}{x\in x_0\sor_Xx_1}
 &&(\text{by assumption on }f),
 \end{align*}
the converse inequality being obvious.
Similarly, $1=\bigvee\Bigl(\widetilde{f(x)}^Y\mid x\in\Min{X}\Bigr)$. This completes the proof of the existence of~$\gf\eqdef\xF(f)$ at Boolean algebra level.
Now let $p\in P$.
For each $x\in X$, $p\leq\partial_Xx$ implies $p\leq\partial_Yf(x)$, thus $\widetilde{f(x)}^Y\in\rF(Y)^{(p)}$.
By definition, it follows that $\gf\rI{\rF(X)^{(p)}}\subseteq\rF(Y)^{(p)}$, thus completing the proof that~$\gf$ is a morphism in~$\Bool_P$\,.
\end{proof}

\section{Standard lifters}
\label{S:Lifters}

We are now reaching the infinite combinatorial aspects of our theory.
Let us first recall the concept of \emph{lifter} introduced in Gillibert and Wehrung \cite[Definition~3.2.1]{Larder}.

\begin{definition}\label{D:Lifter}
Let~$\gl$ be an infinite cardinal and let~$P$ be a poset.
A \emph{$\gl$-lifter} of~$P$ is a triple $(X,\bX,\partial)$, with $(X,\partial)$ a norm-covering of~$P$ and $\bX\subseteq\Ids{X}$, which satisfies the following conditions:
\begin{enumerater}
\item
The set $\bX^=\eqdef\setm{\bx\in\bX}{\partial{\bx}\text{ is not maximal in }P}$ is lower $\cf(\gl)$-small.

\item
\label{DefLiftFree}
For every map $S\colon\bX^=\to[X]^{{<}\gl}$ there exists an isotone section $\gs\colon P\hookrightarrow\bX$ of~$\partial$ such that $S(\gs(p))\cap\gs(q)\subseteq\gs(p)$ whenever $p<q$ in~$P$.

\item
If $\gl=\go$, then~$X$ is supported.
\end{enumerater}
If $(X,\bX,\partial)$ is a $\gl$-lifter of~$P$ with~$\bX$ the set of all principal ideals of~$X$, we will say that $(X,\partial)$ is a \emph{principal $\gl$-lifter} of~$P$.
\end{definition}

Our next construction will provide us with all the lifters we will need.
The construction~$P\seq{K}$ was introduced in the proof of Gillibert and Wehrung \cite[Lemma~3.5.5]{Larder}.

\begin{definition}\label{D:PseqK}
For any poset~$P$ and any set~$K$, let
 \begin{equation}\label{Eq:DefPseqK}
 P\seq{K}\eqdef\setm{(a,x)}{a\in P\text{ and }
 (\exists\text{ finite }X\subseteq P\dnw a)(a\in\Sor{X}
 \text{ and }x\colon X\to K)}\,,
 \end{equation}
ordered componentwise (i.e., $(a,x)\leq(b,y)$ if{f} $a\leq b$ and~$y$ extends~$x$) and endowed with the $P$-valued ``norm function'' $\partial$ defined \emph{via} the rule $\partial(a,x)\eqdef a$.
\end{definition}

The following easy result is contained in the proof of Gillibert and Wehrung \cite[Lemma~3.5.5]{Larder}.

\begin{lemma}\label{L:PseqKajs}
Let~$P$ be an \ajs\ with zero and let~$K$ be a set.
Then~$P\seq{K}$ is a lower finite \ajs\ with zero and $(P\seq{K},\partial)$ is a norm-covering of~$P$.
\end{lemma}

\begin{definition}\label{D:StandLift}
If~$P\seq{K}$ \pup{together with its canonical norm function~$\partial$} is a principal $\gl$-lifter of~$P$, we will call it a \emph{standard $\gl$-lifter} of~$P$.
\end{definition}

The following observation shows that the construction $P\seq{K}$ yields a convenient class of submorphisms of norm-coverings (cf. Definition~\ref{D:ProperNC}), and thus morphisms of $P$-scaled Boolean algebras (cf. Lemma~\ref{L:F(proper)}).
 
\begin{lemma}\label{L:P(f,gS)}
Let~$P$ be an \ajs\ with zero and let~$f\colon X\rightarrowtail Y$ be a one-to-one map from a set~$X$ into a set~$Y$.
Then the rule $(a,x)\mapsto(a,f\circ x)$ defines a one-to-one submorphism $P\seq{f}\colon P\seq{X}\rightarrowtail P\seq{Y}$ of norm-coverings.
In particular, it induces a morphism $\xF(P\seq{f})\colon\xF(P\seq{X})\to\xF(P\seq{Y})$ of $P$-scaled Boolean algebras.
\end{lemma}

\begin{proof}
Let $n<\go$, $(a_1,x_1),\dots,(a_n,x_n)\in P\seq{X}$, and set $x=\bigcup_{i=1}^nx_i$\,.
It is straightforward to verify that
 \[
 \Sor\setm{(a_i,f\circ x_i)}{1\leq i\leq n}=\begin{cases}
 \Sor\setm{a_i}{1\leq i\leq n}\times\set{f\circ x}\,,
 &\text{if }x\text{ is a function}\,,\\
 \es\,,&\text{otherwise},
 \end{cases}
 \]
where the left hand side is evaluated within~$P\seq{Y}$.

The last part of the statement of Lemma~\ref{L:P(f,gS)} then follows from Lemma~\ref{L:F(proper)}.
\end{proof}

Following Gillibert and Wehrung \cite[Definition~3.5.1]{Larder}, for cardinals~$\gk$ and~$\gl$ and a poset~$P$, $(\gk,{<}\gl)\leadsto P$ means that for every $F\colon\Pow(\gk)\to[\gk]^{<\gl}$ there exists a one-to-one map $f\colon P\rightarrowtail\gk$ such that
 \[
 F(f[P\dnw a])\cap f[P\dnw b]\subseteq f[P\dnw a]
 \quad\text{whenever }a<b\text{ in }P\,.
 \]
Also recall from Erd\H{o}s \emph{et al.}~\cite{EHMR} that for cardinals~$\gk$, $\gl$, $\rho$, $(\gk,{<}\go,\gl)\rightarrow\rho$ means that for every map $F\colon[\gk]^{<\go}\to[\gk]^{{<}\gl}$ there exists $H\in[\gk]^{\rho}$ such that $F(X)\cap H\subseteq X$ whenever $X\in[H]^{<\go}$.

As observed in Gillibert and Wehrung \cite[Proposition~3.4]{GilWeh2011}, the statements $(\gk,{<}\gl)\leadsto([\rho]^{<\go},\subseteq)$ and $(\gk,{<}\go,\gl)\rightarrow\rho$ are equivalent.
Since every lower finite poset~$P$ of cardinality~$\rho$ embeds into~$[P]^{<\go}$ (via $x\mapsto P\dnw x$), thus into $[\rho]^{<\go}$, we obtain the following.

\begin{lemma}\label{L:raw2leadsto}
Let~$\gk$ and~$\gl$ be infinite cardinals and let~$P$ be a lower finite poset.
Set $\rho\eqdef\card{P}$.
If $(\gk,{<}\go,\gl)\rightarrow\rho$, then $(\gk,{<}\gl)\leadsto P$.
\end{lemma}

We record the following consequence of Gillibert and Wehrung \cite[Lemma~3.5.5]{Larder}.

\begin{lemma}\label{L:355Larder}
Let~$P$ be a lower finite \ajs\ with zero, let~$\gl$ and~$\gk$ be infinite cardinals such that every element of~$P$ has less than~$\cf(\gl)$ upper covers and $(\gk,{<}\gl)\leadsto P$.
Then~$P\seq{\gk}$ is a standard $\gl$-lifter of~$P$.
\end{lemma}

Using Gillibert and Wehrung~\cite[Proposition~4.7]{GilWeh2011}, Lemma~\ref{L:355Larder} enables us to find lifters for finite posets:

\begin{corollary}\label{C:355Larder}
Let~$P$ be a nontrivial finite \ajs\ with zero and denote by~$n$ the order-dimension of~$P$.
Then for every infinite cardinal~$\gl$ and for every $\gk\geq\gl^{+(n-1)}$, $P\seq{\gk}$ is a standard $\gl$-lifter of~$P$.
\end{corollary}

\section{Extending the Armature Lemma and CLL}
\label{S:ExtCLL}

The main aim of this section is to establish Lemmas~\ref{L:ExtArm} and~\ref{L:ExtCLL}, which are extensions to $\Phi$-com\-mu\-ta\-tive diagrams and~$\otmF$ of the original Armature Lemma and Condensate Lifting Lemma CLL (cf. Lemmas~3.2.2 and~3.4.2, respectively, in Gillibert and Wehrung~\cite{Larder}).
The updated statements, although still quite technical, are somehow trimmed down in comparison to the original statements from~\cite{Larder}, owing to an apparently smaller level of generality: for example, $\gl$ is now assumed to be regular and~$\bX^=$ is, in Lemma~\ref{L:ExtCLL}, assumed to be well-founded.
Other differences between the original statements and the new ones are the following:
\begin{itemize}
\item
The original statement of the Armature Lemma involved a morphism $\chi\colon S\to\Phi(\xF(X)\otimes\vec{A})$.
There is no loss of generality in assuming that~$\chi$ is the identity (just replace each~$\gf_{\bx}$ by~$\chi\circ\gf_{\bx}$), which we thus do in Lemma~\ref{L:ExtArm}.

\item
Our assumptions contain the additional statement that~$P$ is a conditional $\mu$-DCPO.

\item
Due to the different definition of~$\otmF$ (with respect to the original~$\otimes$ of Gillibert and Wehrung~\cite{Larder}), the functor~$\Phi$ no longer needs to preserve any kind of directed colimit.

\item
The cardinal~$\gl$ plays the same role as in Lemmas~3.2.2 and~3.4.2 of Gillibert and Wehrung~\cite{Larder}.
This is not the case for~$\mu$, which is the parameter indexing the operator~$\otmF$\,.

\end{itemize}

The proofs of Lemma~\ref{L:ExtArm} and~\ref{L:ExtCLL} are similar to the ones of Lemma~3.2.2 and~3.4.2 in Gillibert and Wehrung~\cite{Larder}, with a few subtle differences.
Due to the complexity of the underlying statements, we anchor the new formulations in our discussion by showing quite detailed outlines of those proofs.

\begin{lemma}[Extended Armature Lemma]\label{L:ExtArm}
Let~$\gl$ and~$\mu$ be infinite regular cardinals with $\mu\leq\gl$, and let~$P$ be a conditional $\mu$-DCPO with a $\gl$-lifter $(X,\bX,\partial)$.
Let~$\cA$ and~$\cS$ be categories and let $\Phi\colon\cA\to\cS$ be a functor.
We assume that~$\cA$ has all $2^{<\mu}$-products \pup{cf. Definition~\textup{\ref{D:PRODgl}}} and that~$\cS$ has all $\mu$-directed colimits.

Let~$\vec{A}$ be a $P$-indexed, $\Phi$-com\-mu\-ta\-tive diagram in~$\cA$, set $S\eqdef\xF(X)\otmF\vec{A}$, and let $\vecm{(S_{\bx},\gf_{\bx}),\gf_{\bx}^{\by}}{\bx\subseteq\by\text{ in }\bX}$ be an $\bX$-indexed com\-mu\-ta\-tive diagram in $\cS/S$ such that~$S_{\bx}$ is weakly $\gl$-pre\-sentable whenever $\bx\in\bX^=$.
Then there exists an isotone section $\gs\colon P\hookrightarrow\bX$ of~$\partial$ such that the family $\Vecm{(\pi_{\gs(p)}^X\otmF\vec{A})\circ\gf_{\gs(p)}}{p\in P}$ is a natural transformation from the com\-mu\-ta\-tive diagram~$\vec{S}\gs$ to~$\Phi\vec{A}$.
\end{lemma}

\begin{proof}
Our assumptions ensure the existence of the functor ${}_{-}\otmF\nobreak\vec{A}$ from~$\Bool_P$ to~$\cS$ (cf. Lemma~\ref{L:otmFexists}).
Moreover, since the diagram~$\vec{A}$ is $\Phi$-com\-mu\-ta\-tive, the diagram~$\Phi\vec{A}$ is com\-mu\-ta\-tive.
For all $p\leq q$ in~$P$, the constant value~$\bga_p^q$ of~$\Phi(x)$, for~$x$ ranging over~$\vec{A}(p,q)$, is a morphism from~$\Phi(A_p)$ to~$\Phi(A_q)$.

As at the beginning of the proof of Gillibert and Wehrung \cite[Lemma~3.2.2]{Larder}, $X$ is the $\gl$-directed, thus also $\mu$-directed (because $\mu\leq\gl$) union of the set $[X]^{{<}\gl}_{\sor}$ of all its $\gl$-small $\sor$-closed subsets.
For every $\bx\in\bX^=$, since
 \begin{equation}\label{Eq:gfbxinlimotlF}
 \gf_{\bx}\colon S_{\bx}\to\xF(X)\otmF\vec{A}
 =\varinjlim\vecm{\xF(Z)\otmF\vec{A}}{Z\in[X]^{{<}\gl}_{\sor}}
 \end{equation}
where the transition morphisms and limiting morphisms in the right hand side of~\eqref{Eq:gfbxinlimotlF} all have the form $f_{Z_0}^{Z_1}\otmF\vec{A}$,
and~$S_{\bx}$ is weakly $\gl$-pre\-sentable, there exists a $\gl$-small $\sor$-closed subset~$V(\bx)$ of~$X$ such that~$\gf_{\bx}$ factors through~$\xF(V(\bx))\otmF\vec{A}$.
The mapping~$V$ thus goes from~$\bX^=$ to~$[X]^{{<}\gl}$.
As in the proof of \cite[Lemma~3.2.2]{Larder}, we may assume that the map~$V$ is isotone.
By the definition of~$V(\bx)$, there is a morphism $\psi_{\bx}\colon S_{\bx}\to\xF(V(\bx))\otmF\vec{A}$ such that
 \begin{equation}\label{Eq:Defnpsi}
 \gf_{\bx}=(f_{V(\bx)}^{X}\otmF\vec{A})\circ\psi_{\bx}\,.
 \end{equation}
Since $(X,\bX,\partial)$ is a $\gl$-lifter of~$P$, there is an isotone section~$\gs\colon P\hookrightarrow\bX$ of~$\partial$ such that
 \begin{equation}\label{Eq:gsfreewrtV}
 V(\gs(p))\cap\gs(q)\subseteq\gs(p)\quad
 \text{for all }p<q\text{ in }P\,.
 \end{equation}
The proof of the following claim is identical to the one of the Claim in the proof of \cite[Lemma~3.2.2]{Larder} (it is a direct translation of~\eqref{Eq:gsfreewrtV}) and we omit it.

\begin{sclaim}
The equation $\pi_{\gs(q)}^X\circ f_{V\gs(p)}^X=\eps_p^q\circ\pi_{\gs(p)}^X\circ f_{V\gs(p)}^X$ holds for all $p<q$ in~$P$.
\end{sclaim}

By applying the functor ${}_{-}\otmF\vec{A}$ to the two sides of the Claim above and then applying Lemma~\ref{L:otmFexists}, together with Lemma~\ref{L:2potmvecS}, we obtain the equation
 \begin{equation}\label{Eq:tolF2Flag}
 (\pi_{\gs(q)}^X\otmF\vec{A})\circ
 (f_{V\gs(p)}^X\otmF\vec{A})=\bga_p^q\circ
 (\pi_{\gs(p)}^X\otmF\vec{A})\circ(f_{V\gs(p)}^X\otmF\vec{A})\,.
 \end{equation}
The end of the proof goes the same way as in the one of Gillibert and Wehrung \cite[Lemma~3.2.2]{Larder}: for all $p<q$ in~$P$,
 \begin{align*}
 \bga_p^q\circ\pI{\pi_{\gs(p)}^X\otmF\vec{A}}\circ\gf_{\gs(p)}&=
 \bga_p^q\circ\pI{\pi_{\gs(p)}^X\otmF\vec{A}}\circ
 \pI{f_{V\gs(p)}^X\otmF\vec{A}}\circ\psi_{\gs(p)}&&
 (\text{use~\eqref{Eq:Defnpsi}})\\
 &=\pI{\pi_{\gs(q)}^X\otmF\vec{A}}\circ
 \pI{f_{V\gs(p)}^X\otmF\vec{A}}\circ\psi_{\gs(p)}&&
 (\text{use~\eqref{Eq:tolF2Flag}})\\
 &=\pI{\pi_{\gs(q)}^X\otmF\vec{A}}\circ\gf_{\gs(p)}&&
 (\text{use~\eqref{Eq:Defnpsi}})\\
 &=\pI{\pi_{\gs(q)}^X\otmF\vec{A}}\circ\gf_{\gs(q)}
 \circ\gf_{\gs(p)}^{\gs(q)}\,,
 \end{align*}
which completes the proof of the desired naturalness statement.
\end{proof}

The following Lemma~\ref{L:ExtCLL}, extending the original CLL (viz. Gillibert and Wehrung \cite[Lemma~3.4.2]{Larder}) can be viewed as a more ``global'' version of Lem\-ma~\ref{L:ExtArm}.
Its statement involves a subcategory~$\cS^{\Rightarrow}$ of~$\cS$, whose morphisms will be called the \emph{double arrows} and denoted in the form $x\colon S_1\Rightarrow S_2$.
Similarly, natural transformations with all arrows in~$\cS^{\Rightarrow}$ will be denoted in the form $\vec{x}\colon\vec{S}_1\Todot\vec{S}_2$.

For the statement of Lemma~\ref{L:ExtCLL}, recall that $\gl$-extended $\Phi$-pro\-jec\-tions were introduced in Definition~\ref{D:ExtPhiProj}.
Lemma~\ref{L:ExtCLL} also involves the \emph{projectability witnesses} introduced in Wehrung~\cite[Definition~3.2]{RetrLift}, see also Gillibert and Wehrung \cite[Definition~1.5.1]{Larder}.
Heuristically, for a functor~$\Psi$, a projectability witness for an arrow $\psi\colon\Psi(C)\to S$ plays the role of a ``quotient''~$\ol{C}$ of~$C$ such that~$\psi$ induces an isomorphism $\ol{\psi}\colon\Psi(\ol{C})\to S$.
As a full definition of that concept is relatively technical, and as everything we need here about it has already been proved elsewhere, we refer the reader to the abovecited references for more detail.

\begin{lemma}[Extended CLL]\label{L:ExtCLL}
Let~$\gl$ and~$\mu$ be infinite regular cardinals with $\mu\leq\gl$, and let~$P$ be a poset with a $\gl$-lifter $(X,\bX,\partial)$.
Let~$\cA$, $\cB$, and~$\cS$ be categories, with functors $\Phi\colon\cA\to\cS$ and $\Psi\colon\cB\to\cS$.
Let~$\cB^{\dagger}$ be a full subcategory of~$\cB$ and let~$\cS^{\Rightarrow}$ \pup{the \emph{double arrows} in~$\cS$} be a subcategory of~$\cS$.
We are given the following data:
\begin{itemize}
\item
a $P$-indexed, $\Phi$-com\-mu\-ta\-tive diagram $\vec{A}=\vecm{A_p\,,\vec{A}(p,q)}{p\leq q\text{ in }P}$ in~$\cA$;

\item
an object~$B\in\cB$ together with a double arrow $\chi\colon\Psi(B)\Rightarrow\xF(X)\otmF\vec{A}$.
\end{itemize}

We make the following assumptions:
\begin{itemize}
\item[(WF)]
$\bX^=$ is well-founded.

\item[(COND$(\mu)$)]
$P$ is a conditional $\mu$-DCPO \pup{cf. Definition~\textup{\ref{D:CondJoinCplt}}}.

\item[(PROD$(\mu)$)]
$\cA$ has all $2^{<\mu}$-products \pup{cf. Definition~\textup{\ref{D:PRODgl}}}.

\item[(COLIM$(\mu)$)]
$\cS$ has all $\mu$-directed colimits.

\item[(PROJ$(\mu)$)]
Every $\mu$-extended $\Phi$-pro\-jec\-tion belongs to~$\cS^{\Rightarrow}$.

\item[(PRES$(\gl)$)]
For every $C\in\cB^{\dagger}$, $\Psi(C)$ is weakly $\gl$-pre\-sentable in~$\cS$.

\item[(LS$(\gl)$)]
For every $p\in P$, every $\psi\colon\Psi(B)\Rightarrow\Phi(A_p)$, every $\ga<\gl$, and every family $\vecm{\gc_{\xi}\colon C_{\xi}\rightarrowtail B}{\xi<\ga}$ in the subobject category $\cB^{\dagger}\dnw B$, there exists a subobject~$\gc$ in $\cB^{\dagger}\dnw B$ such that $\psi\circ\Psi(\gc)\in\cS^{\Rightarrow}$ and each $\gc_{\xi}\utr\gc$.
\end{itemize}
Then there are a com\-mu\-ta\-tive diagram $\vec{B}\in\cB^P$ and a natural transformation $\vec{\chi}\colon\Psi\vec{B}\Todot\Phi\vec{A}$ in~$\cS^{\Rightarrow}$.
Furthermore, if every double arrow $\psi\colon\Psi(C)\Rightarrow S$, where $C\in\cB$ and $S\in\cS$, has a projectability witness with respect to the functor~$\Psi$, then~$\vec{\chi}$ can be taken a natural equivalence.
\end{lemma}

\begin{proof}
As at the beginning of the proof of Lemma~\ref{L:ExtArm}, our assumptions ensure the commutativity of the diagram~$\Phi\vec{A}$ and the existence of the functor ${}_{-}\otmF\vec{A}$ from~$\Bool_P$ to~$\cS$ (cf. Lemma~\ref{L:otmFexists}).

For every $\bx\in\bX$, the morphism $\pi_{\bx}^X\colon\xF(X)\to\two[\partial\bx]$ is normal (cf. Gillibert and Wehrung \cite[Lem\-ma~2.6.7]{Larder}).
By Proposition~\ref{P:otlFproj}, $\pi_{\bx}^X\otmF\vec{A}$ is a $\mu$-extended $\Phi$-pro\-jec\-tion from $\xF(X)\otmF\vec{A}$ to $\two[\partial\bx]\otmF\vec{A}=\Phi(A_{\partial\bx})$ (cf. Lemma~\ref{L:2potmvecS}).
By (PROJ$(\mu)$), it follows that $\pi_{\bx}^X\otmF\vec{A}$ is a double arrow.
Therefore, $\pi_{\bx}^X\otmF\vec{A}$ is a double arrow from $\xF(X)\otmF\vec{A}$ to~$\Phi(A_{\partial\bx})$.
It follows that the composite $\rho_{\bx}\eqdef(\pi_{\bx}^X\otmF\vec{A})\circ\chi$ is a double arrow from~$\Psi(B)$ to~$\Phi(A_{\partial\bx})$.

Due to the simplification brought by assuming~(WF) from the start, all assumptions underlying the monic form of the Buttress Lemma (cf. Lemma~3.3.2 and Remark~3.3.3 in Gillibert and Wehrung~\cite{Larder}) are, taking $U:=\bX^=$, satisfied.
This yields an $\bX^=$-indexed com\-mu\-ta\-tive diagram
$\vecm{\gc_{\bx},\gc_{\bx}^{\by}}{\bx\subseteq\by\text{ in }\bX^=}$ in $\cB^{\dagger}\dnw B$, say $\gc_{\bx}\colon C_{\bx}\rightarrowtail B$ and $\gc_{\bx}^{\by}\colon C_{\bx}\rightarrowtail C_{\by}$ (all $C_{\bx}\in\cB^{\dagger}$), such that each $\rho_{\bx}\circ\Phi(\gc_{\bx})\in\cS^{\Rightarrow}$.
This diagram can be extended to an $\bX$-indexed com\-mu\-ta\-tive diagram in $\cB\dnw B$, by setting $C_{\bx}=B$ and $\gc_{\bx}=\gc_{\bx}^{\by}=\id_B$ whenever $\bx\subseteq\by$ in~$\bX^=$, whereas $\gc_{\bx}^{\by}=\gc_{\bx}$ whenever $\bx\subseteq\by$, $\bx\in\bX^=$, and $\by\in\bX\setminus\bX^=$.
The relation $\rho_{\bx}\circ\Phi(\gc_{\bx})\in\cS^{\Rightarrow}$ now holds for all $\bx\in\bX$: if~$\bx\in\bX\setminus\bX^=$, then $\rho_{\bx}\circ\Phi(\gc_{\bx})=\rho_{\bx}\in\cS^{\Rightarrow}$.
Hence,
 \begin{equation}\label{Eq:rhobxDA}
 \rho_{\bx}\circ\Phi(\gc_{\bx})\colon\Psi(B)
 \Rightarrow\Phi(A_{\partial\bx})\,,\quad\text{for any }\bx\in\bX\,.
 \end{equation}
It follows from the assumption~(PRES$(\gl)$) that~$\Psi(C_{\bx})$ is weakly $\gl$-pre\-sentable whenever $\bx\in\bX^=$.
Setting $\gf_{\bx}\eqdef\chi\circ\Psi(\gc_{\bx})$ and $\gf_{\bx}^{\by}\eqdef\Psi(\gc_{\bx}^{\by})$, all the assumptions of Lemma~\ref{L:ExtArm} are satisfied.
This yields an isotone section~$\gs$ of~$\partial$ such that the family $\vec{\chi}=\vecm{\chi_p}{p\in P}$, where each $\chi_p\eqdef\rho_{\gs(p)}\circ\Psi(\gc_{\gs(p)})$, is a natural transformation from the com\-mu\-ta\-tive diagram~$\Psi\vec{B}$, where $\vec{B}\eqdef\vecm{C_{\gs(p)},\gc_{\gs(p)}^{\gs(q)}}{p\leq q\text{ in }P}$, to~$\Phi\vec{A}$.
By~\eqref{Eq:rhobxDA}, each~$\chi_p$ is a double arrow.

The last statement of Lemma~\ref{L:ExtCLL}, that existence of enough projectability witnesses implies that~$\vec{\chi}$ can be taken a natural equivalence, is proved the same way as at the end of Gillibert and Wehrung \cite[Lemma~3.4.2]{Larder}.
\end{proof}

In Sections \ref{S:Ceva}--\ref{S:4SCML} we shall explore various occurrences of anti-el\-e\-men\-tar\-ity following from Lemmas~\ref{L:ExtArm} and~\ref{L:ExtCLL}.
The former (intervening in Section~\ref{S:Ceva}) offers the advantage of providing less restrictive cardinality assumptions, at the expense of requiring a deeper understanding of the class of structures under consideration.
By contrast, Lemma~\ref{L:ExtCLL} (intervening in Sections~\ref{S:PermCon}--\ref{S:4SCML}) yields more streamlined proofs, enabling us to apply known non-rep\-re\-sentabil\-ity results as black boxes, at the expense of more restrictive cardinality assumptions.

\section{Conrad frames and Cevian lattices}\label{S:Ceva}

Recall from Wehrung~\cite{Ceva} that a binary operation~$\sd$ on a distributive lattice~$D$ with zero is \emph{Cevian} if all inequalities $x\leq y\vee(x\sd y)$, $(x\sd y)\wedge(y\sd x)=0$, and $x\sd z\leq
 (x\sd y)\vee(y\sd z)$ hold whenever $x,y,z\in D$.
 We also say that the lattice~$D$ is Cevian if it carries a Cevian operation.
 
Recall from Iberkleid \emph{et al.}~\cite{IMM2011} that a \emph{Conrad frame} is a lattice isomorphic to the lattice~$\Cs{G}$ of all convex $\ell$-sub\-groups of an \lgrp\ (not necessarily Abelian)~$G$.
Since~$\Cs{G}$ is an algebraic frame, it is determined by the \jzs~$\Csc{G}$ of all finitely generated convex $\ell$-sub\-groups of~$G$, which turns out to be a distributive lattice with zero.

Let us record a few properties of Cevian lattices and (lattices of compact members of) Conrad frames, established in Wehrung~\cite{Ceva}:

\begin{proposition}\label{P:FewCevians}\hfill
\begin{enumerater}
\item
Every Cevian lattice is \emph{completely normal}, that is, for all $a,b\in D$ there are~$x,y\in D$ such that $a\vee b=a\vee y=x\vee b$ whereas $x\wedge y=0$.

\item
There exists a non-Cevian completely normal bounded distributive lattice, of cardinality~$\aleph_2$\,.

\item
For every \lgrp~$G$, the lattice~$\Csc{G}$ is Cevian.

\item
For every representable%
\footnote{
Recall that an \lgrp\ is \emph{representable} if it is a subdirect product of totally ordered groups.
}
\lgrp~$G$, the \jzs~$\Idc{G}$ of all finitely generated $\ell$-ideals of~$G$ is a lattice, and also a homomorphic image of~$\Csc{G}$; thus it is a Cevian lattice.

\end{enumerater}
\end{proposition}

Let us recall the construction of the $\Idc$-com\-mu\-ta\-tive diagram~$\vec{A}$, represented in Figure~\ref{Fig:DiagA}, introduced in Wehrung~\cite{Ceva}.
The indexing poset of our counterexample diagrams will be $P\eqdef\Pow[3]$ (cf. Section~\ref{S:Basic}).
Denote by~$A_{123}$ the Abelian \lgrp\ defined by the generators~$a$, $a'$, $b$, $c$ subjected to the relations $0\leq a\leq a'\leq 2a$, $0\leq b$, and~$0\leq c$.
For each $p\in P$, $A_p$ denotes the $\ell$-sub\-group of~$A_{123}$ generated by~$\nu(p)$ where $\nu(12)\eqdef\set{a,b}$, $\nu(13)\eqdef\set{a',c}$, $\nu(23)\eqdef\set{b,c}$, $\nu(1)\eqdef\set{a}$, $\nu(2)\eqdef\set{b}$, $\nu(3)\eqdef\set{c}$, $\nu(\es)\eqdef\es$.
For $p\subseteq q$ in~$P$, $\vec{A}(p,q)$ consists of the inclusion map, unless $p=1$ and $q=13$, in which case~$\vec{A}(p,q)$ consists of the map sending~$a$ to~$a'$, or $p=1$ and $q=123$, in which case~$\vec{A}(p,q)$ consists of the two maps sending~$a$ to either~$a$ or~$a'$.
On the diagram, we highlight the canonical set of generators of each~$A_p$\,, for example~$A_{12}(a,b)$, $A_{123}(a,a',b,c)$, and so on.

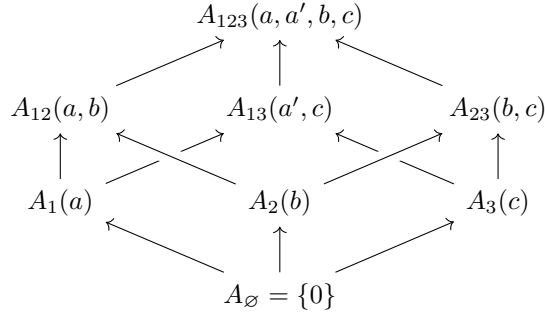
\begin{figure}[htb]
\begin{tikzcd}
\centering
& A_{123}(a,a',b,c) &\\
A_{12}(a,b)\ar[ur] &
A_{13}(a',c)\ar[u] &
A_{23}(b,c)\ar[ul]
&&\\
A_1(a)
\arrow[u]
\arrow[ur]
& A_2(b)
\arrow[ul,crossing over]
& A_3(c)
\arrow[ul]
\arrow[u]
\arrow[from=l,u,crossing over] \\
& A_{\es}=\set{0}\arrow[ul]\arrow[u]\arrow[ur] &
\end{tikzcd}
\caption{The non-com\-mu\-ta\-tive diagram~$\vec{A}$}
\label{Fig:DiagA}
\end{figure}

\begin{notation}\label{Not:cA(gq)}
For any infinite regular cardinal~$\gq$ and any diagram~$\vec{G}$ in the category~$\lGrp$ of all \lgrp{s} with \lhom{s}, we denote by~$\cA(\gq,\vec{G})$ the smallest subcategory of~$\lGrp$, containing all objects and arrows of~$\vec{G}$, and closed under products and under colimits indexed by~$\gl$, within~$\lGrp$, whenever $\gl\geq\gq$ is regular.
Also, we denote by~$\Cev$ the class of all Cevian distributive lattices with zero.
\end{notation}

Of course, every member of~$\cA(\gq,\vec{A})$ is an Abelian \lgrp.
Moreover, if $\gq>\nobreak\go$\,, then every member of~$\cA(\gq,\vec{A})$ is Ar\-chi\-me\-dean (because every object in~$\vec{A}$ is Ar\-chi\-me\-dean and the class of all Ar\-chi\-me\-dean \lgrp{s} is closed both under products and under all colimits indexed by uncountable regular cardinals).

We are now reaching this section's main result.

\begin{theorem}\label{T:NotCev}
For all infinite regular cardinals~$\gq$ and~$\gl$ with $\gq\leq\gl$, there exists a functor~$\gD$, from~$\Powi(\gl^{+2})$ to the category of all distributive lattices with zero with $\scL_{\infty\gl}$-el\-e\-men\-tary embeddings, satisfying the following statements:
\begin{enumerater}
\item\label{IncrUnion}
$\gD$ is $\gl$-continuous;

\item\label{smallDxirep}
For every $\gl^+$-small subset~$X$ of~$\gl^{+2}$, $\gD(X)$ belongs to~$\Idc\pI{\cA(\gq,\vec{A})}$;

\item\label{Dgqnotrep}
$\gD(\gl^{+2})$ is not Cevian;

\item\label{Dgqsmall}
$\card\gD(X)\leq\exp_2(\us{\gl})+\card{X}$ whenever $X\subseteq\gl^{+2}$.

\end{enumerater}
In particular, the pair
$\pII{\Idc\pI{\cA(\gq,\vec{A})},\Cev}$ is anti-el\-e\-men\-tary.
\end{theorem}

\begin{proof}
Consider again the poset $P\eqdef\Pow[3]$ and the $P$-indexed diagram~$\vec{A}$ introduced above.
We shall apply the Extended Armature Lemma (i.e., Lemma~\ref{L:ExtArm}), with $\gl=\mu$, to the category $\cA\eqdef\cA(\gq,\vec{A})$, the category $\cS:=\DLatz$ of all distributive lattices with zero and $0$-lattice homomorphisms, the restriction~$\Phi$ of the functor~$\Idc$ to~$\cA(\gq,\vec{A})$.
Since $\cA(\gq,\vec{A})$ is closed under products and since the diagram~$\vec{A}$ is $\Idc$-com\-mu\-ta\-tive, $\vec{A}$ is also $\Phi$-com\-mu\-ta\-tive.

Set $\gk\eqdef\gl^{+2}$.
It follows from Corollary~\ref{C:355Larder} that $K\eqdef P\seq{\gk}$ is a standard $\gl$-lifter of~$P$.
By Lemma~\ref{L:P(f,gS)}, the assignment $U\mapsto\xF(P\seq{U})$ extends naturally to a functor from $\Powi(\gk)$ to $\Bool_P$\,.
This functor sends every morphism $f\colon U\rightarrowtail V$ in~$\Powi(\gk)$ to~$\xF(P\seq{f})$ (cf. Lemma~\ref{L:P(f,gS)}).
By composing that functor with ${}_{-}\otlF\vec{A}$, we obtain a functor $\gC\colon\Powi(\gk)\to\nobreak\cS$, $U\mapsto\xF(P\seq{U})\otlF\vec{A}$.
Any directed colimit $U=\varinjlim_{i\in I}U_i$ in~$\Powi(\gk)$ (essentially a directed union) gives rise to a directed colimit
 \[
 \xF(P\seq{U})=\varinjlim\vecm{\xF(P\seq{U_i}}{i\in I}
 \quad\text{within }\Bool_P\,.
 \]
Hence, the functor $U\mapsto\xF(P\seq{U})$ is $\go$-continuous.
Since the functor~${}_{-}\otlF\vec{A}$ is $\gl$-continuous, it follows that the composite $\gC\eqdef\xF(P\seq{{}_{-}})\otlF\vec{A}$ is $\gl$-continuous.

\setcounter{claim}{0}

\begin{claim}\label{Cl:gCglCev}
For every $\gl^+$-small subset~$X$ of~$\gk$, the lattice~$\gC(X)$ belongs to the range of~$\Phi$ \pup{thus, by Proposition~\textup{\ref{P:FewCevians}}, it is Cevian}.
\end{claim}

\begin{cproof}
If $\card{X}<\gl$ then $\gC(X)=\Phi\pI{\xF(P\seq{X})\bt\vec{A}}$ belongs to the range of~$\Phi$.
Since~$\gC$ is a functor, the case where $\card{X}=\gl$ can be reduced to the case where $X=\gl$.
Since $\cA(\gq,\vec{A})$ is closed under directed colimits indexed by~$\gl$, those directed colimits are preserved by~$\Phi$, so we can apply the Boosting Lemma (viz. Lemma~\ref{L:Boosting}) to the $\gl$-pre\-sentable $P$-scaled Boolean algebras $\bB_{\xi}\eqdef\xF(P\seq{\xi})$ (for $\xi<\gl$), $\bB\eqdef\xF(P\seq{\gl})$, each~$\beta_{\xi}^{\eta}$ is the canonical morphism $f_{P\seq{\xi}}^{P\seq{\eta}}\colon\xF(P\seq{\xi})\to\xF(P\seq{\eta})$, each~$\beta_{\xi}$ is the canonical morphism $f_{P\seq{\xi}}^{P\seq{\gl}}\colon\xF(P\seq{\xi})\to\xF(P\seq{\gl})$.
Since $\bB=\varinjlim_{\xi<\gl}\bB_{\xi}$\,, it follows that $\gC(\gl)=\bB\otlF\vec{A}$ belongs to the range of~$\Phi$.
\end{cproof}

\begin{claim}\label{Cl:bBnotCev}
The lattice~$\gC(\gk)$ is not Cevian.
\end{claim}

\begin{cproof}
The proof of this claim is established by an argument similar to the one, in the proof of \cite[Theorem~7.2]{Ceva}, showing that the lattice denoted there by~$\bB$ is not Cevian.
We give an outline for convenience.

Suppose that~$\gC(\gk)$ has a Cevian operation~$\sd$ and denote by~$K_{(j)}$ the set of all elements of~$K$ of height~$j$, for $j\in\set{0,1,2,3}$.
For each $x\in K$, it follows from the normality of the morphism~$\pi_{x}^K\colon\xF(P\seq{\gk})\to\two[\partial x]$ (cf. Section~\ref{S:NormCov}), together with Proposition~\ref{P:otlFproj}, that the morphism $\rho_{x}\eqdef\pi_{x}^K\otlF\vec{A}\colon\gC(\gk)\to\Phi(A_{\partial x})$ is a $\gl$-extended $\Phi$-pro\-jec\-tion.
Since every $\Phi$-pro\-jec\-tion is (obviously) surjective, so is every $\gl$-extended $\Phi$-pro\-jec\-tion, and so~$\rho_{x}$ is surjective.
In particular, if $x\in K_{(1)}$\,, then $\Phi(A_{\partial x})=\two$, so we may pick $b_{x}\in\gC(\gk)$ such that $\rho_{x}(b_{x})=1$;
set $S_{x}\eqdef\set{0,b_{x}}$.
Further, if $x\in K_{(2)}\cup K_{(3)}$, denote by~$S_{x}$ the sublattice of~$\gC(\gk)$ generated by
 \[
 \setm{b_{u}}{u\in K_{(1)}\dnw x}\cup
 \setm{b_{u}\sd b_{v}}{u,v\in K_{(1)}\dnw x}\,.
 \]
Since~$K$ is lower finite and~$\gC(\gk)$ is distributive, it follows that~$S_{x}$ is finite.
Denote by $\gf_{x}\colon S_{x}\hookrightarrow\gC(\gk)$ the inclusion map, and, for $x\subseteq y$ in~$K$, denote by $\gf_{x}^{y}\colon S_{x}\hookrightarrow S_{y}$ the inclusion map.
By applying Lemma~\ref{L:ExtArm} with $\mu:=\gl$, we get an isotone section $\gs\colon P\hookrightarrow K$ of~$\partial$ such that the family $\vec{\chi}\eqdef\vecm{\chi_p}{p\in P}$, with each $\chi_p=\rho_{\gs(p)}\res_{S_{\gs(p)}}$\,, is a natural transformation from~$\vec{S}\gs$ to~$\Phi\vec{A}$, with $\chi_p(b_{\gs(p)})=1$ whenever~$p$ is an atom of~$P$.
However, the last stages of the proof of \cite[Theorem~7.2]{Ceva} show that the existence of such a natural transformation contradicts \cite[Lemma~4.3]{Ceva}.
\end{cproof}

Now denote by~$\overset{.}{+}$ the usual ordinal addition and set $\gD(X)\eqdef\gC(\gl\cup(\gl\overset{.}{+}X)$ (where $\gl\overset{.}{+}X\eqdef\setm{\gl\overset{.}{+}\xi}{\xi\in X}$) whenever $X\subseteq\gk$.
Extend~$\gD$ to a functor from~$\Powi(\gk)$ to the category of all distributive lattices with zero, in the natural way.
Since the cardinality of $\gl\cup(\gl\overset{.}{+}X)$ is always  greater than or equal to~$\gl$, it follows from Proposition~\ref{P:EltEq} that~$\gD$ sends morphisms in~$\Powi(\gk)$ to $\scL_{\infty\gl}$-el\-e\-men\-tary embeddings.
Moreover, \eqref{smallDxirep} follows from Claim~\ref{Cl:gCglCev} whereas~\eqref{Dgqnotrep} follows from Claim~\ref{Cl:bBnotCev}.
Since~$\gC$ is $\gl$-continuous, so is~$\gD$; that is, \eqref{IncrUnion} holds.

Finally, each~$\gC(Z)$, where $Z\in[\gk]^{{<}\gl}$, has cardinality bounded above by~$\exp_2(\us{\gl})$.
By elementary cardinal arithmetic, it follows that for every $X\subseteq\gk$,
 \[
 \card\Delta(X)\leq\sum
 \Vecm{\card\gC(Z)}{Z\in[\gl\cup(\gl\overset{.}{+}X)]^{{<}\gl}}\leq
 \exp_2(\us{\gl})+\card X\,;
 \]
that is, \eqref{Dgqsmall} holds.
\end{proof}

\begin{corollary}\label{C:NotCev.5}
Let~$\gq$ be an infinite regular cardinal and let~$\cG$ be a class of \lgrp{s} containing $\cA(\gq,\vec{A})$.
Then~$\Csc(\cG)$ is anti-el\-e\-men\-tary.
\end{corollary}

Since every member of~$\cA(\aleph_1,\vec{A})$ is Ar\-chi\-me\-dean, we obtain:

\begin{corollary}\label{C:NotCev.6}
Let~$\cG$ be a class of \lgrp{s} containing all Ar\-chi\-me\-dean \lgrp{s}.
Then~$\Csc(\cG)$ is anti-el\-e\-men\-tary.
\end{corollary}

\begin{corollary}\label{C:NotCev1}
Let~$\gq$ be an infinite regular cardinal and let~$\cG$ be a class of representable \lgrp{s} containing $\cA(\gq,\vec{A})$.
Then~$\Idc(\cG)$ is anti-el\-e\-men\-tary.
\end{corollary}

For a further extension of that result, see Theorem~\ref{T:NonIdc}.

Again since every member of~$\cA(\aleph_1,\vec{A})$ is Ar\-chi\-me\-dean, we obtain:

\begin{corollary}\label{C:NotCev1.5}
Let~$\cG$ be a class of representable \lgrp{s} containing all Ar\-chi\-me\-dean \lgrp{s}.
Then~$\Idc(\cG)$ is anti-el\-e\-men\-tary.
\end{corollary}

In order to extend Corollary~\ref{C:NotCev1} to \lgrp{s} with unit, we shall use the following easy observation, involving the notations~$P^{\infty}$ and $G\lex H$ introduced in Section~\ref{S:Basic}.

\begin{lemma}\label{L:AdjUnit}
For any \lgrp~$G$, $\Idc(\ZZ\lex G)\cong(\Idc{G})^{\infty}$.\end{lemma}

\begin{corollary}\label{C:NotCev2}
Let~$\gq$ be an infinite regular cardinal and let~$\cG$ be a class of representable \lgrp{s} containing $\setm{\ZZ\lex G}{G\in\cA(\gq,\vec{A})}$.
Then~$\Idc(\cG)$ is anti-el\-e\-men\-tary.
\end{corollary}

\begin{proof}
Let~$\gl$ be an infinite regular cardinal with $\gl\geq\gq$ and let~$\gD$ be the functor, defined on~$\Powi(\gl^{+2})$, given by Theorem~\ref{T:NotCev}.
Denote by~$E$ the functor, from the category of all distributive lattices with zero to the category of all bounded distributive lattices, that sends any object~$D$ to~$D^{\infty}$ and any morphism~$f$ to its extension preserving~$\infty$.
The functor $\gD^{\infty}\eqdef E\circ\gD$ is $\gl$-continuous.
Since $\gD(\gl)\cong\Idc{A}$ for some $A\in\cA(\gq,\vec{A})$, it follows from Lemma~\ref{L:AdjUnit} that $\gD^{\infty}(\gl)\cong\Idc(\ZZ\lex A)$, with $\ZZ\lex A\in\cC$ by assumption.

Since $\gD(\gl^{+2})$ is not Cevian, neither is $\gD^{\infty}(\gl^{+2})=\gD(\gl^{+2})^{\infty}$.
On the other hand, since every member of~$\cC$ is representable, every member of~$\Idc(\cC)$ is Cevian (cf. Proposition~\ref{P:FewCevians}).
\end{proof}

We do not know whether Corollary~\ref{C:NotCev1.5} extends to Ar\-chi\-me\-dean \lgrp{s} with order-unit.
The problem is that the arrows in the diagram~$\vec{A}$ are not unit-preserving.

Since Mundici's well known category equivalence~\cite{Mund86}, between \emph{Abelian \lgrp{s} with unit} and \emph{MV-algebras}, preserves the concept of ideal, we thus obtain the following solution to the \emph{MV-spectrum problem} stated in Mundici \cite[Problem~2]{Mund2011}:

\begin{corollary}\label{C:NotMV}
The class of all Stone duals of spectra of MV-algebras \pup{or, equivalently, of all Stone duals of spectra of Abelian \lgrp{s} with order-unit} is anti-el\-e\-men\-tary.
\end{corollary}

In particular, the class of all Stone duals of spectra of MV-algebras is not closed under $\scL_{\infty\gl}$-el\-e\-men\-tary equivalence for any cardinal~$\gl$ (in particular, it is not definable by any class of~$\scL_{\infty\gl}$ sentences).
Mellor and Tressl~\cite{MelTre2012} proved an analogue result for Stone duals of \emph{real spectra} of com\-mu\-ta\-tive unital rings.
We do not know whether the stronger statement, that the class of all those lattices is anti-el\-e\-men\-tary, holds.

\section{Structures with permutable congruences}
\label{S:PermCon}

This section will be devoted to establishing anti-el\-e\-men\-tar\-ity of classes of finitely generated congruences of various congruence-permutable structures, such as modules, rings, \lgrp{s}.
Our argument will slightly deviate from the one of Section~\ref{S:Ceva}, in the sense that we will use the Extended CLL (viz. Lemma~\ref{L:ExtCLL}) rather than the Extended Armature Lemma (viz. Lemma~\ref{L:ExtArm}), thus assuming more global, easier stated versions of the unliftability of the diagram~$\vec{S}$ of Figure~\ref{Fig:CandS}, forcing us to include in the statement of Theorem~\ref{T:NonPermlgrp} the assumption that $\gl\geq\aleph_1$\,.

\subsection{Categorical settings for Section~\ref{S:PermCon}}\label{Su:CatSetPermCon}

Throughout Section~\ref{S:PermCon} we shall consider the category $\cS:=\SLatz$ of all \jzs{s} with \jzh{s}.
Following Wehrung \cite[Definition 7-3.4]{STA1-9}, we say that a \jzh\ $f\colon S\to\nobreak T$ of \jzs{s} is \emph{weakly distributive} if for all $s\in S$ and all $t_0,t_1\in T$, if $f(s)\leq t_0\vee t_1$\,, then there are $s_0,s_1\in S$ such that $s\leq s_0\vee s_1$ and each $f(s_i)\leq t_i$\,.
We shall denote by~$\cS^{\Rightarrow}$ (this section's double arrows) the subcategory of~$\cS$ consisting of the weakly distributive \jzh{s}.

The indexing poset for our crucial counterexample diagrams will again be $P\eqdef\Pow[3]$ (cf. Section~\ref{S:Basic}).
We shall denote by~$\vec{S}$ the $P$-indexed com\-mu\-ta\-tive diagram in~$\cS$ represented in Figure~\ref{Fig:CandS}, with the maps~$\be$, $\bs$, $\bp$ defined by $\be(x)\eqdef(x,x)$, $\bs(x,y)\eqdef(y,x)$, $\bp(x,y)\eqdef x\vee y$ whenever $x,y\in\two$.
Its origin can be traced back to T\r{u}ma and Wehrung~\cite{TuWe2001}, with a more complicated precursor of that diagram.

 \begin{figure}[htb]
 \centering
 \begin{tikzcd}
 & \two &\\
 \two^2\arrow[ur,"\bp"] & \two^2\arrow[u,"\bp"description] &
 \two^2\arrow[ul,"\bp"']\\
  \two^2\arrow[u,equal]\arrow[ur,equal]
 & \two^2\arrow[ul,"\bs"description,crossing over,near start] &
 \two^2\arrow[u,equal]\arrow[ul,equal]
  \arrow[from=l,u,equal,crossing over]\\
 & \two\arrow[ul,"\be",hook']\arrow[u,"\be",hook]
  \arrow[ur,"\be"',hook] &
 \end{tikzcd}
\caption{The com\-mu\-ta\-tive diagram~$\vec{S}$}\label{Fig:CandS}
\end{figure}

Let us recall the definition of the category~$\Metr$ introduced in Gillibert and Wehrung \cite[\S~5.1]{Larder}.
The objects of~$\Metr$ are the \emph{semilattice-metric spaces}, that is, the triples $\bA=(A,\gd_{\bA},\widetilde{A})$ such that~$A$ is a set, $\widetilde{A}$ is a \jzs, and $\gd_{\bA}\colon A\times A\to\widetilde{A}$ is a \emph{semilattice-valued distance}, that is,
 \[
 \gd_{\bA}(x,x)=0\,;\quad\gd_{\bA}(x,y)=\gd_{\bA}(y,x)\,;\quad
 \gd_{\bA}(x,z)\leq\gd_{\bA}(x,y)\vee\gd_{\bA}(y,z)
 \]
whenever $x,y,z\in A$;
we say that~$\bA$ \emph{is of type~$1$} if for all $x,y\in A$ and all $\ba,\bb\in\widetilde{A}$, if $\gd_{\bA}(x,y)\leq\ba\vee\bb$, then there exists $z\in A$ such that $\gd_{\bA}(x,z)\leq\ba$ and $\gd_{\bA}(z,y)\leq\nobreak\bb$.
For objects~$\bA$ and~$\bB$ of~$\Metr$, a \emph{morphism} from~$\bA$ to~$\bB$ is a pair $(f,\tilde{f})$, where $f\colon A\to B$ is a map and $\tilde{f}\colon\widetilde{A}\to\widetilde{B}$ is a \jzh, such that $\gd_{\bB}(f(x),f(y))=\tilde{f}\gd_{\bA}(x,y)$ whenever $x,y\in A$.

Throughout Section~\ref{S:PermCon} we shall consider the full subcategory~$\cB$ of~$\Metr$ whose objects are all semilattice-metric spaces~$\bA$ of type~$1$ such that the range of~$\gd_{\bA}$ generates~$\widetilde{A}$ as a \jzs\ --- we will say that it \emph{join-generates}~$\widetilde{A}$.
Moreover, we shall denote by~$\Psi\colon\cB\to\cS$, $\bA\mapsto\widetilde{A}$ the forgetful functor.

Recall the following observation from R\r{u}\v{z}i\v{c}ka \emph{et al.}~\cite{RTW}:

\begin{proposition}\label{P:ExofrngPsi}
The \jzs\ of all finitely generated congruences of any congruence-permutable \pup{universal} algebra belongs to the range of~$\Psi$.
In particular, the following \jzs{s} all belong to the range of~$\Psi$:
\begin{enumerater}
\item\label{Idclgrp}
the \jzs\ $\Idc{G}$ of all principal $\ell$-ideals of any \lgrp~$G$;

\item\label{IdcvNRR}
the \jzs\ $\Idc{R}$ of all finitely generated two-sided ideals of any ring~$R$;

\item\label{SubcM}
the \jzs\ $\Subc{M}$ of all finitely generated submodules of any right module~$M$.

\end{enumerater}
\end{proposition}

We omit the straightforward proof of the following lemma.

\begin{lemma}\label{L:CompWD}
For any \jzs~$S$, any object~$\bA$ of~$\cB$, and any $f\colon\widetilde{A}\Rightarrow S$, the composite $f\bA\eqdef(A,f\circ\gd_{\bA},S)$ is an object of~$\cB$.
\end{lemma}

\begin{lemma}\label{L:NonvecS}
There are no com\-mu\-ta\-tive diagram $\vec{\bB}\in\cB^P$ and no natural transformation $\vec{\chi}\colon\Psi\vec{\bB}\Todot\vec{S}$.
\end{lemma}

\begin{proof}
Letting $\vec{\chi}=\vecm{\chi_p}{p\in P}$, we may, thanks to Lemma~\ref{L:CompWD}, replace each~$\bB_p$ by $\chi_p\bB_p$ and thus assume that each $\widetilde{B}_p=S_p$\,.
But then, Wehrung \cite[Theorem 9-5.1]{STA1-9} leads to a contradiction.
\end{proof}

 \begin{figure}[htb]
 \centering
 \begin{tikzcd}
 & C & && & \Mat{2}{\Bbbk} &\\
 \ZZ^2\arrow[ur,hook',"h"] &
 \ZZ^2\arrow[u,hook,"h"description] & 
 \ZZ^2\arrow[ul,hook,"h"'] &&
 \Bbbk^2\arrow[ur,"h",hook'] &
 \Bbbk^2\arrow[u,"h"description,hook] &
 \Bbbk^2\arrow[ul,"h"',hook]\\
 \ZZ^2\arrow[u,equal]\arrow[ur,equal]
  & \ZZ^2\arrow[ul,"s"description,crossing over,near start] & 
 \ZZ^2\arrow[u,equal]\arrow[ul,equal]
 \arrow[from=l,u,equal,crossing over] &&
  \Bbbk^2\arrow[u,equal]\arrow[ur,equal]
  & \Bbbk^2\arrow[ul,"s"description,crossing over,near start] & 
 \Bbbk^2\arrow[u,equal]\arrow[ul,equal]
  \arrow[from=l,u,equal,crossing over]\\
 & \ZZ\arrow[ul,"e",hook']\arrow[u,"e",hook]
 \arrow[ur,"e"',hook] & &&
  & \Bbbk\arrow[ul,"e",hook']\arrow[u,"e",hook]
  \arrow[ur,"e"',hook] &
 \end{tikzcd}
\caption{The non-com\-mu\-ta\-tive diagrams~$\vec{C}_{a,g}$ and~$\vec{R}_{\Bbbk}$}\label{Fig:CandR}
\end{figure}

The following lemma takes care of all instances of Condition~(LS$(\gl)$), from the statement of Lemma~\ref{L:ExtCLL}, occurring in Section~\ref{S:PermCon}.
Its proof is a standard L\"owenheim-Skolem type argument and we omit it.

\begin{lemma}\label{L:LS4lgrps}
Let~$\gl$ be an uncountable regular cardinal, let~$\bB$ be an object of~$\cB$, and let $f\colon\widetilde{B}\Rightarrow S$ be a weakly distributive \jzh\ with join-generating range.
If~$S$ is $\gl$-small, then the set of all $(X,Y)\in[B]^{<\gl}\times[\widetilde{B}]^{<\gl}$, such that $(X,\gd_{\bB}\res_{X\times X},Y)$ is an object of~$\cB$ and $f\res_{Y}$ is weakly distributive with join-generating range, is $\gl$-closed cofinal in $[B]^{<\gl}\times[\widetilde{B}]^{<\gl}$.
\end{lemma}

\subsection{Classes of non-representable \lgrp{s}}
\label{Su:NonReprlgrp}
It is well known that the assignment, that sends any \lgrp~$G$ to the \jzs\ of its finitely generated $\ell$-ideals, naturally extends to a functor $\Idc\colon\lGrp\to\cS$.
For any \lhom\ $f\colon G\to H$, $\Idc{f}$ sends every $\ell$-ideal~$X$ of~$G$ to the $\ell$-ideal of~$H$ generated by~$f[X]$.

Let~$C$ be a non-representable \lgrp.
By one of the equivalent forms of representability (cf. Bigard \emph{et al.} \cite[Proposition~4.2.9]{BKW}), there are $a\in C^+\setminus\set{0}$ and $g\in C$ such that $a\wedge(g+a-g)=0$ --- we will say that $(a,g)$ is an \emph{NR-pair} of~$C$.
Set $b\eqdef g+a-g$, then $e(x)\eqdef(x,x)$, $s(x,y)\eqdef(y,x)$, and $h(x,y)\eqdef xa+yb$, for all $x,y\in\ZZ$.
{F}rom $a\wedge b=0$ it follows that~$h$ is an \lhom.
We denote by~$\vec{C}_{a,g}$ the non-com\-mu\-ta\-tive diagram, in~$\lGrp$, represented in the left hand side of Figure~\ref{Fig:CandR}.

\begin{lemma}\label{L:vecCIdcComm}
The diagram~$\vec{C}_{a,g}$ is $\Idc$-com\-mu\-ta\-tive.
\end{lemma}

\begin{proof}
For any set~$I$, we need to prove that the image under the functor~$\Idc$\,, of the diagram~$\vec{C}_{a,g}^I$\,, is com\-mu\-ta\-tive.
Let $p=\vecm{p_i}{i\in I}$ and $q=\vecm{q_i}{i\in I}$ in~$P^I$ such that $p\leq q$, let $\vecm{f_i}{i\in I}$ and $\vecm{g_i}{i\in I}$ be elements of $\prod\Vecm{\vec{C}_{a,g}(p_i,q_i)}{i\in I}$.
Set $f\eqdef\prod\vecm{f_i}{i\in I}$ and $g\eqdef\prod\vecm{g_i}{i\in I}$.
We need to prove that~$f(z)$ and~$g(z)$ generate the same $\ell$-ideal of~$\vec{C}_{a,g}^I(q)$, whenever $z=\vecm{z_i}{i\in I}$ belongs to~$\vec{C}_{a,g}^I(p)^+$.
Every~$i\in I$ outside $I_0\eqdef\setm{i\in I}{f_i=g_i}$ satisfies $(p_i,q_i)=(1,123)$, so, setting
 \begin{align*}
 I_1&\eqdef\setm{i\in I}
 {(p_i,q_i)=(1,123)\text{ and }f_i=h\text{ and }g_i=h\circ s}\,,\\
 I_2&\eqdef\setm{i\in I}
 {(p_i,q_i)=(1,123)\text{ and }f_i=h\circ s\text{ and }g_i=h}\,,
 \end{align*}
it follows that~$I$ is the disjoint union of~$I_0$\,, $I_1$\,, and~$I_2$\,.
For every $J\subseteq I$, denote by~$z|J$ the element of~$\vec{C}_{a,g}^I(p)$ agreeing with~$z$ on~$J$ and taking the constant value~$0$ on $I\setminus J$.
Since~$f(z)$ generates the same $\ell$-ideal as $\set{f(z|{I_0}),f(z|{I_1}),f(z|{I_2})}$, and similarly for~$g$, it suffices to settle the case where $I=I_k$ for some $k\in\set{0,1,2}$.
If $k=0$ then $f=g$ and we are done.
The two other cases being symmetric, it remains to settle the case where $k=1$.
For each $i\in I$\,, we can write $z_i=(x_i,y_i)\in\ZZ^+\times\ZZ^+$, thus $f_i(z_i)=x_ia+y_ib$ and $g_i(z_i)=x_ib+y_ia$.
Since $nb=g+na-g$ for every integer~$n$, we get
 \begin{align*}
 \seql{\vecm{x_ia+y_ib}{i\in I}}&=\seql{\vecm{x_ia}{i\in I}}\vee
 \seql{\vecm{y_ib}{i\in I}}\\
 &=\seql{\vecm{x_ib}{i\in I}}\vee \seql{\vecm{y_ia}{i\in I}}\\
 &=\seql{\vecm{x_ib+y_ia}{i\in I}}\,,
 \end{align*}
that is, $\seql{f(z)}=\seql{g(z)}$, as required.
\end{proof}

We set $\vec{\pi}\eqdef\vecm{\pi_p}{p\in P}$, where $\pi_{123}$ is the unique map from~$\Idc{C}$ to~$\two$ sending~$0$ to~$0$ and any nonzero element to~$1$, whereas, for $p\neq123$, $\pi_p$ is the canonical isomorphism from~$\Idc{C_{a,g}(p)}$ (either~$\Idc\ZZ^2$ or $\Idc\ZZ$) onto~$S(p)$ (either~$\two^2$ or~$\two$).
We omit the straightforward proof of the following lemma.

\begin{lemma}\label{L:PhiC2S}
The family~$\vec{\pi}$ is a natural transformation from~$\Idc\vec{C}_{a,g}$ to~$\vec{S}$ in~$\cS^{\Rightarrow}$.
\end{lemma}

\begin{lemma}\label{L:noPsiB2PhiC}
There are no com\-mu\-ta\-tive diagram $\vec{\bB}\in\cB^P$ and no natural transformation $\vec{\chi}\colon\Psi\vec{\bB}\Todot\Idc\vec{C}_{a,g}$\,.
\end{lemma}

\begin{proof}
If~$\vec{\chi}$ were such a natural transformation, then, by Lemma~\ref{L:PhiC2S}, we would get $\vec{\pi}\circ\vec{\chi}\colon\Psi\vec{\bB}\Todot\vec{S}$, in contradiction with Lemma~\ref{L:NonvecS}.
\end{proof}

We are now reaching this section's main result.

\begin{theorem}\label{T:NonPermlgrp}
For all infinite regular cardinals~$\gq$ and~$\gl$ with $\gq+\aleph_1\leq\gl$, there exists a functor~$\gD$, from~$\Powi(\gl^{+2})$ to the category~$\cS$ of all \jzs{s} with~$\scL_{\infty\gl}$-el\-e\-men\-tary embeddings, satisfying the following statements:
\begin{enumerater}
\item\label{IncrUnionlgrp}
$\gD$ is $\gl$-continuous;

\item\label{smallDxireplgrp}
For every $\gl^+$-small subset~$X$ of~$\gl^{+2}$, $\gD(X)$ belongs to $\Idc\pI{\cA(\gq,\vec{C}_{a,g})}$;

\item\label{Dgqnotreplgrp}
There are no set~$A$ and no distance $\gd\colon A\times A\to\gD(\gl^{+2})$ of type~$1$ with join-generating range.

\end{enumerater}
In particular, the pair $\pII{\Idc\pI{\cA(\gq,\vec{C}_{a,g})},\rng\Psi}$ is anti-el\-e\-men\-tary.
\end{theorem}

\begin{proof}
Set $\gk\eqdef\gl^{+2}$.
It follows from Corollary~\ref{C:355Larder} that $P\seq{\gk}$ is a standard $\gl$-lifter of~$P$.
Denote by~$\Phi$ the restriction of the functor~$\Idc$ to~$\cA(\gq,\vec{C}_{a,g})$.
As at the beginning of the proof of Theorem~\ref{T:NotCev}, we obtain a functor $\gC\colon\Powi(\gk)\to\nobreak\cS$, $U\mapsto\xF(P\seq{U})\otlF\vec{C}_{a,g}$\,, and this functor is $\gl$-continuous.

\setcounter{claim}{0}

The proof of the following claim is, \emph{mutatis mutandis}, identical to the one of Claim~\ref{Cl:gCglCev} of the proof of Theorem~\ref{T:NotCev} (involving the Boosting Lemma) and we omit it.

\begin{claim}\label{Cl:gCglRep}
For every $\gl^+$-small subset~$X$ of~$\gk$, the lattice~$\gC(X)$ belongs to the range of~$\Phi$.
\end{claim}

\begin{claim}\label{Cl:bBnottype1}
There are no set~$A$ and no distance $\gd\colon A\times A\to\gC(\gk)$ of type~$1$ with join-generating range.
\end{claim}

\begin{cproof}
Let $\chi\colon\widetilde{A}=\Psi(\bA)\to\gC(\gk)$ be an isomorphism (or, more generally, a weakly distributive \jzh\ with join-generating range) and denote by~$\cB^{\dagger}$ the full subcategory of~$\cB$ whose objects are the~$\bB$ such that~$B$ and~$\widetilde{B}$ are both $\gl$-small.
We shall verify that Lemma~\ref{L:ExtCLL} applies.
All assumptions are trivially satisfied except for (PROJ$(\gl)$) and (LS$(\gl)$).
The latter follows from Lemma~\ref{L:LS4lgrps}.
For the former, for any \lgrp{s}~$G$, $H$ and any surjective \lhom\ $f\colon G\twoheadrightarrow\nobreak H$, the \jzh\ $\gf\eqdef\Idc{f}\colon\Idc{G}\to\Idc{H}$ is surjective.
Moreover, specializing Gillibert and Wehrung \cite[Lemma~4.5.1]{Larder} to the variety of all \lgrp{s} and by the natural isomorphism between~$\Idc{G}$ and the \jzs~$\Conc{G}$ of all finitely generated congruences of any \lgrp~$G$, we see that for all $\ba,\bb\in\Idc{G}$, $\gf(\ba)\leq\gf(\bb)$ if{f} there exists $\bx\in\gf^{-1}\set{0}$ such that $\ba\leq\bb\vee\bx$ and $\gf(\bx)=\nobreak0$; with the terminology of Gillibert and Wehrung~\cite{Larder}, $\gf$ is \emph{ideal-induced}.
Since every directed colimit of ideal-induced \jzh{s} is ideal-induced, and since ideal-induced implies weakly distributive, every extended $\Phi$-pro\-jec\-tion belongs to~$\cS^{\Rightarrow}$.
Condition (PROJ$(\gl)$) follows.

Therefore, Lemma~\ref{L:ExtCLL} applies and there are a com\-mu\-ta\-tive diagram $\vec{\bB}\in\nobreak\cB^P$ and a natural transformation $\vec{\chi}\colon\Psi\vec{\bB}\Todot\Phi\vec{C}_{a,g}$; in contradiction with Lemma~\ref{L:noPsiB2PhiC}.
\end{cproof}

The remainder of the proof runs as in the one of Theorem~\ref{T:NotCev}, again with $\gD(X)\eqdef\gC(\gl\cup(\gl\overset{.}{+}X))$.
\end{proof}

\begin{theorem}\label{T:NonIdc}
Let~$\cG$ be a subcategory of~$\lGrp$, closed under products and under colimits indexed by all large enough regular cardinals, and containing all objects and arrows of the diagram~$\vec{A}$ of Section~\textup{\ref{S:Ceva}}.
Then $\Idc(\cG)$ is anti-el\-e\-men\-tary.
\end{theorem}

\begin{proof}
Let~$\gq$ be any infinite regular cardinal such that~$\cG$ is closed under all $\gl$-indexed colimits whenever~$\gl$ is an infinite regular cardinal $\geq\gq$.
We separate cases.
Our assumptions imply that~$\cG$ contains the diagram $\cA(\gq,\vec{A})$ (cf. Notation~\ref{Not:cA(gq)}).
Hence, if~$\cG$ is representable, then we can apply Corollary~\ref{C:NotCev1}.
If~$\cG$ is not representable, that is, it has a non-representable member~$C$, which in turn has an NR-pair $(a,g)\in(C^+\setminus\set{0})\times C$.
Then $\cA(\gq,\vec{C}_{a,g})\subseteq\cG$ and we can apply Theorem~\ref{T:NonPermlgrp}.
\end{proof}

In particular, every nontrivial \emph{quasivariety}~$\cV$ of \lgrp{s} contains~$\ZZ$ as a member, thus (since all objects in~$\vec{A}$ are subdirect powers of~$\ZZ$) it contains~$\cA(\go,\vec{A})$.
Since~$\cV$ is closed under products and under directed colimits, Theorem~\ref{T:NonIdc} applies to~$\cV$, so \emph{$\Idc(\cV)$ is anti-el\-e\-men\-tary}.

Straying away from quasivarieties, we also obtain the following:

\begin{corollary}\label{C:NonIdc}
Let~$\cG$ be a full subcategory of~$\lGrp$, closed under products and under colimits indexed by all large enough regular cardinals, and containing all Ar\-chi\-me\-dean \lgrp{s}.
Then $\Idc(\cG)$ is anti-el\-e\-men\-tary.
\end{corollary}

Since the arrows in the diagram~$\vec{A}$ are not unit-preserving, we do not know whether Corollary~\ref{C:NonIdc} extends to the case of \lgrp{s} with order-unit.

\subsection{Ideal lattices of rings}
\label{Su:IdcRegRings}
Denote by~$\Ring$ the category of all unital rings with unital ring homomorphisms.
The assignment, that sends any ring~$R$ to the \jzs\ of its finitely generated two-sided ideals, naturally extends to a functor $\Idc\colon\Ring\to\cS$.
For a ring homomorphism $f\colon R\to S$, $\Idc{f}$ sends every two-sided ideal~$X$ of~$R$ to the two-sided ideal of~$S$ generated by~$f[X]$.

Now let~$\Bbbk$ be a field and denote by~$\Mat{2}{\Bbbk}$ the ring of all $2\times 2$ square matrices over~$\Bbbk$.
For all $x,y\in\Bbbk$, we set $e(x)\eqdef(x,x)$, $s(x,y)\eqdef(y,x)$, and $h(x,y)\eqdef\begin{pmatrix}x & 0\\ 0 & y\end{pmatrix}$.
We denote by~$\vec{R}_{\Bbbk}$ the non-com\-mu\-ta\-tive diagram, in~$\Ring$, represented in the right hand side of Figure~\ref{Fig:CandR}.

\begin{lemma}\label{L:vecRIdcComm}
The diagram~$\vec{R}_{\Bbbk}$ is $\Idc$-com\-mu\-ta\-tive.
\end{lemma}

\begin{proof}
For any set~$I$, we need to prove that the image under the functor~$\Idc$\,, of the diagram~$\vec{R}_{\Bbbk}^I$\,, is com\-mu\-ta\-tive.
Let $p=\vecm{p_i}{i\in I}$ and $q=\vecm{q_i}{i\in I}$ in~$P^I$ such that $p\leq q$, let $\vecm{f_i}{i\in I}$ and $\vecm{g_i}{i\in I}$ be elements of $\prod\Vecm{\vec{R}_{\Bbbk}(p_i,q_i)}{i\in I}$.
Set $f\eqdef\prod\vecm{f_i}{i\in I}$ and $g\eqdef\prod\vecm{g_i}{i\in I}$.
In order to prove that the maps~$\Idc{f}$ and~$\Idc{g}$ are equal, it suffices to prove that they agree on the two-sided ideal of~$\vec{R}_{\Bbbk}^I(p)$ generated by any $z=\vecm{z_i}{i\in I}\in\vec{R}_{\Bbbk}^I(p)$;
that is, it suffices to verify that the elements~$f(z)$ and~$g(z)$ generate the same two-sided ideal of $\vec{R}_{\Bbbk}^I(q)$.
The complement in~$I$ of $I_0\eqdef\setm{i\in I}{f_i=g_i}$ is
contained in $\setm{i\in I}{(p_i,q_i)=(1,123)}$, so, setting
 \begin{align*}
 I_1&\eqdef\setm{i\in I}
 {(p_i,q_i)=(1,123)\text{ and }f_i=h\text{ and }g_i=h\circ s}\,,\\
 I_2&\eqdef\setm{i\in I}
 {(p_i,q_i)=(1,123)\text{ and }f_i=h\circ s\text{ and }g_i=h}\,,
 \end{align*}
it follows that~$I$ is the disjoint union of~$I_0$\,, $I_1$\,, and~$I_2$\,.
Denote by~$\chi_J$ the characteristic function of any subset~$J$ of~$I$.
Since~$f(z)$ generates the same two-sided ideal as $\set{f(z\chi_{I_0}),f(z\chi_{I_1}),f(z\chi_{I_2})}$, and similarly for~$g$, it suffices to settle the case where $I=I_k$ for some $k\in\set{0,1,2}$.
If $k=0$ then $f=g$ and we are done.
The two other cases being symmetric, it remains to settle the case where $k=1$.
For each $i\in I_1$\,, we can write $z_i=(x_i,y_i)\in\Bbbk\times\Bbbk$, thus $f_i(z_i)=\begin{pmatrix}x_i & 0\\ 0 & y_i\end{pmatrix}$ and $g_i(z_i)=\begin{pmatrix}y_i & 0\\ 0 & x_i\end{pmatrix}$.
Defining~$\ol{u}$ as the constant $I$-indexed family with value $\begin{pmatrix}0 & 1\\ 1 & 0\end{pmatrix}$, we get $\ol{u}=\ol{u}^{-1}$ and $g(z)=\ol{u}\cdot f(z)\cdot\ol{u}^{-1}$.
The desired conclusion follows.
\end{proof}

We set $\vec{\pi}\eqdef\vecm{\pi_p}{p\in P}$ where each $\pi_p\colon\Idc{R_{\Bbbk}(p)}\to S(p)$ is the canonical isomorphism.
The following lemma is trivial.

\begin{lemma}\label{L:PhiR2S}
The family~$\vec{\pi}$ is a natural isomorphism from~$\Idc\vec{R}_{\Bbbk}$ onto~$\vec{S}$.
\end{lemma}

The following lemma is then an immediate consequence of Lemma~\ref{L:NonvecS}.

\begin{lemma}\label{L:noPsiB2PhiR}
There are no com\-mu\-ta\-tive diagram $\vec{\bB}\in\cB^P$ and no natural transformation $\vec{\chi}\colon\Psi\vec{\bB}\Todot\Idc\vec{R}_{\Bbbk}$\,.
\end{lemma}

\begin{notation}\label{Not:cR(gq)}
For any infinite regular cardinal~$\gq$, we denote by~$\cR(\gq,\Bbbk)$ the smallest subcategory of~$\Ring$, containing all objects and arrows of~$\vec{R}_{\Bbbk}$\,, and closed under products and under colimits, within~$\Ring$, indexed by any regular $\gl\geq\gq$.
\end{notation}

We are now reaching this section's main result.

\begin{theorem}\label{T:NonPermvNR}
For all infinite regular cardinals~$\gq$ and~$\gl$ with $\gq+\aleph_1\leq\gl$ and every field~$\Bbbk$, there exists a functor~$\gD$, from~$\Powi(\gl^{+2})$ to the category~$\cS$ of all \jzs{s} with~$\scL_{\infty\gl}$-el\-e\-men\-tary embeddings, satisfying the following statements:
\begin{enumerater}
\item\label{IncrUnionvNR}
$\gD$ is $\gl$-continuous;

\item\label{smallDxirepvNR}
For every $\gl^+$-small subset~$X$ of~$\gl^{+2}$, $\gD(X)$ belongs to $\Idc\pI{\cR(\gq,\Bbbk)}$;

\item\label{DgqnotrepvNR}
There are no set~$A$ and no distance $\gd\colon A\times A\to\gD(\gl^{+2})$ of type~$1$ with join-generating range.
\end{enumerater}
In particular, the pair $\pII{\Idc\pI{\cR(\gq,\Bbbk)},\rng\Psi}$ is anti-el\-e\-men\-tary.
\end{theorem}

\begin{proof}
The proof follows the lines of the one of Theorem~\ref{T:NonPermlgrp}, with~$\Phi$ now defined as the restriction of the functor~$\Idc$ to~$\cR(\gq,\Bbbk)$, Lemma~\ref{L:noPsiB2PhiR} used instead of Lemma~\ref{L:noPsiB2PhiC} and with $\gC\colon\Powi(\gk)\to\nobreak\cS$, $U\mapsto\xF(P\seq{U})\otlF\vec{R}_{\Bbbk}$\,.
\end{proof}

Recall that an algebra over a field~$\Bbbk$ is \emph{locally matricial} if it is a directed colimit of finite products of finite-dimensional matrix rings over~$\Bbbk$.

\begin{corollary}\label{C:NonIdcvNR}
The following classes of \jzs{s} are all anti-el\-e\-men\-tary:
\begin{enumerater}
\item\label{NonvNRR}
$\Idc(\cR)$, whenever~$\cR$ is a class of unital rings containing all objects and arrows of~$\vec{R}_{\Bbbk}$\,, and closed under products and under colimits, within~$\Ring$, indexed by any large enough regular cardinal~$\gl$;

\item\label{NonMod}
The class of all semilattices of finitely generated submodules of right modules.

\end{enumerater}
\end{corollary}

\begin{proof}
\emph{Ad}~\eqref{NonvNRR}.
Since all the vertices of~$\vec{R}_{\Bbbk}$ are locally matricial $\Bbbk$-algebras, we get $\cR(\go,\Bbbk)\subseteq\cR\subseteq\Ring$.
By Proposition~\ref{P:ExofrngPsi}\eqref{IdcvNRR}, $\Idc(\Ring)\subseteq\rng\Psi$.
Apply Theorem~\ref{T:NonPermvNR}.

\emph{Ad}~\eqref{NonMod}.
By a well known module-theoretical trick (see, for example, the proof of R\r{u}\v{z}i\v{c}ka \emph{et al.} \cite[Theorem~4.2]{RTW}), the semilattice of all finitely two-sided ideals of any unital von Neumann regular ring is isomorphic to $\Subc{M}$ for some right module~$M$.
Moreover, by Proposition~\ref{P:ExofrngPsi}\eqref{SubcM}, $\Subc{M}$ belongs to the range of~$\Psi$ for every right module~$M$.
Apply Theorem~\ref{T:NonPermvNR}.
\end{proof}

We do not know whether $\Idc(\cR_{\Bbbk})$, with~$\cR_{\Bbbk}$ defined as the category of all locally matricial $\Bbbk$-algebras, is anti-elementary.
Corollary~\ref{C:NonIdcvNR} cannot be used to that end \emph{a priori}, because~$\cR_{\Bbbk}$ does not have all infinite products (for example, one can prove that the sequence of all matrix algebras of even order over~$\Bbbk$ has no product in~$\cR_{\Bbbk}$).

\section{Nonstable K${}_0$-theory of exchange rings}\label{S:V(R)}


Section~\ref{S:V(R)} will be devoted to proving anti-el\-e\-men\-tar\-ity of classes of monoids arising from the nonstable K${}_0$-theory of rings and operator algebras.
The statement of this section's main result (viz. Theorem~\ref{T:NonV(R)reg}) will deviate from those of Sections~\ref{S:Ceva} and~\ref{S:PermCon}, in the sense that the category~$\cR$ in question (playing the role of the category~$\cA$ in Lemma~\ref{L:ExtCLL}) will no longer be assumed to be closed under directed colimits (indexed by regular cardinals $\geq\gq$) within the category of all unital rings.
This way, we will be able to cover both the case of von Neumann regular rings (where directed colimits from the category~$\Ring$ of all unital rings are preserved; see Corollary~\ref{C:NonVvNR}) and C*-algebras of real rank zero (where those colimits are not preserved --- we need to take the completion; see Corollary~\ref{C:NonVRR0}).

Throughout Section~\ref{S:V(R)} we will consider the category $\cS:=\mathbf{CMon}$ of all \cm{s} with monoid homomorphisms.
Moreover, following the terminology in Wehrung~\cite{VLiftDefect}, a monoid homomorphism $f\colon M\to N$ is a \emph{pre-\Vhom} if whenever $a,b\in N$ and $c\in M$, $f(c)=a+b$ implies the existence of $x,y\in M$ such that $c=x+y$, $f(x)=a$, and $f(y)=b$.
We shall denote by~$\cS^{\Rightarrow}$ (this section's double arrows) the subcategory of~$\cS$ consisting of all pre-\Vhom{s}.

Still following standard terminology, we denote by~$\Mat{\infty}{R}$ the directed union, over all positive integers~$n$, of all matrix rings~$\Mat{n}{R}$, where every~$a\in\Mat{n}{R}$ is identified with
$\begin{pmatrix}a & 0\\ 0 & 0\end{pmatrix}\in\Mat{n+1}{R}$.
Whenever~$n$ is either a positive integer or~$\infty$, we denote by~$\IM{n}{R}$ the set of all idempotent elements in the ring~$\Mat{n}{R}$.
Matrices $a,b\in\IM{\infty}{R}$ are \emph{Murray-von~Neumann equivalent}, in notation $a\sim b$, if there are $x,y\in\Mat{\infty}{R}$ such that $a=xy$ and $b=yx$.
We denote by~$[a]$ the Murray-von~Neumann equivalence class of an idempotent matrix~$a$.
Those equivalence classes can be added, \emph{via} the rule $[a]+[b]\eqdef[a+b]$ whenever $ab=ba=0$ (then we write $a\oplus b$ instead of $a+b$), and $\rV(R)\eqdef\setm{[a]}{a\in\IM{\infty}{R}}$ is a \cm.
This monoid is \emph{conical}, that is, $\bx+\by=0$ implies $\bx=\by=0$, whenever $\bx,\by\in\rV(R)$.
It encodes the so-called \emph{nonstable K${}_0$-theory of~$R$}.

A ring~$R$ is \emph{V-semiprimitive} (cf. Wehrung \cite[Definition~9.1]{VLiftDefect}) if for any $a,b\in\IM{\infty}{R}$, $ab\neq0$ implies that there are decompositions $a=u\oplus u'$ and $b=v\oplus v'$ in $\IM{\infty}{R}$ such that $u\neq0$ and $u\sim v$.
A classical argument shows that it is equivalent to require that every matrix ring~$\Mat{n}{R}$, with~$n$ a positive integer, satisfies the statement above.
In particular, \emph{V-semi\-prim\-it\-iv\-ity can be expressed by a collection of first-order sentences in ring theory}.
Moreover, \emph{any product of V-semiprimitive rings is V-semiprimitive}.

The assignment $R\mapsto\rV(R)$ naturally extends to a functor, from the category of all rings and ring homomorphisms to~$\cS$.
Throughout Section~\ref{S:V(R)} we shall denote by~$\cB$ the category of all V-semiprimitive rings with ring homomorphisms, and by~$\Psi\colon\cB\to\cS$ the restriction of~$\rV$ to~$\cB$.

We record here the following observation from Wehrung \cite[Proposition~9.2]{VLiftDefect}.

\begin{proposition}\label{P:ExofrngPsiVR}
Every semiprimitive exchange ring is V-semiprimitive.
In particular, every ring which is either von Neumann regular or a C*-algebra of real rank zero is V-semiprimitive.
\end{proposition}

The following lemma takes care of all instances of Condition~(LS$(\gl)$), from the statement of Lemma~\ref{L:ExtCLL}, occurring in Section~\ref{S:V(R)}.
Since, as observed above, V-semi\-prim\-it\-iv\-ity is first-order, the proof of Lemma~\ref{L:LS4V(R)} follows from a standard L\"owenheim-Skolem type argument, so we omit it.

\begin{lemma}\label{L:LS4V(R)}
For any uncountable regular cardinal~$\gl$, any $\gl$-small \cm~$M$, any V-semiprimitive ring~$R$, and any pre-\Vhom\linebreak $\psi\colon\rV(R)\Rightarrow M$, the collection~$\cC$ of all $\gl$-small elementary subrings~$R'$ of~$R$, such that $\psi\circ\rV(\id_{R'}^R)$ is a pre-\Vhom, is a $\gl$-closed cofinal subset of the powerset of~$R$.
Moreover, every member of~$\cC$ is V-semiprimitive.
\end{lemma}

Set again~$P\eqdef\Pow[3]$ (cf. Section~\ref{S:Basic}).
Consider the maps~$\be\colon\ZZ^+\to\ZZ^+\times\ZZ^+$, $\bs\colon\ZZ^+\times\ZZ^+\to\ZZ^+\times\ZZ^+$, and $\bp\colon\ZZ^+\times\ZZ^+\to\ZZ^+$ defined \emph{via} the rules $\be(x)\eqdef(x,x)$, $\bs(x,y)\eqdef(y,x)$, and $\bp(x,y)\eqdef x+y$ whenever $x,y\in\ZZ^+$.
Further, we denote by~$\vec{D}$ the $P$-indexed com\-mu\-ta\-tive diagram in~$\cS$ represented in Figure~\ref{Fig:diagD}.

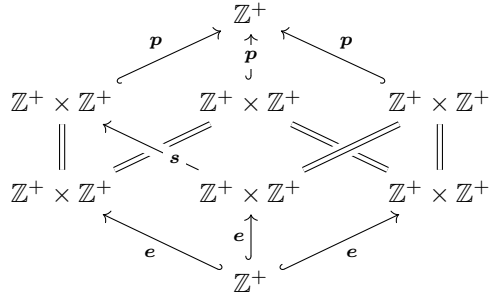
\begin{figure}[htb]
 \centering
 \begin{tikzcd}
 & \ZZ^+ & \\
 \ZZ^+\times\ZZ^+\arrow[ur,hook',"\bp"] &
 \ZZ^+\times\ZZ^+
 \arrow[u,hook,"\bp"description]
 & 
 \ZZ^+\times\ZZ^+\arrow[ul,hook,"\bp"']\\
 \ZZ^+\times\ZZ^+
 \arrow[u,equal]\arrow[ur,equal]
 & \ZZ^+\times\ZZ^+\arrow[ul,"\bs"description,crossing over,
 near start]
 & 
 \ZZ^+\times\ZZ^+\arrow[u,equal]\arrow[ul,equal]
 \arrow[from=l,u,equal,crossing over]\\
 & \ZZ^+\arrow[ul,"\be",hook']
 \arrow[u,"\be",hook]
 \arrow[ur,"\be"',hook] &
 \end{tikzcd}
\caption{The com\-mu\-ta\-tive diagram~$\vec{D}$}\label{Fig:diagD}
\end{figure}

The proof of the following lemma is, \emph{mutatis mutandis}, identical to the one of Lemma~\ref{L:vecRIdcComm}, the main change being that instead of saying that~$f(z)$ generates the same two-sided ideal as $\set{f(z\chi_{I_0}),f(z\chi_{I_1}),f(z\chi_{I_2})}$, we now observe that $[f(z)]=[f(z\chi_{I_0})]+[f(z\chi_{I_1})]+[f(z\chi_{I_2})]$, within $\rV\pI{\vec{R}_{\Bbbk}^I(q)}$, for every idempotent matrix~$z$ over~$\vec{R}_{\Bbbk}^I(p)$.

\begin{lemma}\label{L:vecRVComm}
For any field~$\Bbbk$, the diagram~$\vec{R}_{\Bbbk}$ \pup{introduced in Subsection~\textup{\ref{Su:CatSetPermCon}}}, of unital rings and unital ring homomorphisms, is $\rV$-com\-mu\-ta\-tive.
\end{lemma}

We omit the trivial proof of our next observation.

\begin{lemma}\label{L:rVR2S}
For any field~$\Bbbk$, the diagrams~$\rV(\vec{R}_{\Bbbk})$ and~$\vec{D}$ are naturally isomorphic.
\end{lemma}

The following result is established in Wehrung \cite[Theorem~10.1]{VLiftDefect}.

\begin{lemma}\label{L:noPsiB2PhiRBbbk}
There are no com\-mu\-ta\-tive diagram $\vec{\bB}\in\cB^P$ and no natural transformation $\vec{\chi}\colon\Psi\vec{\bB}\Todot\vec{D}$.
\end{lemma}

We are now reaching this section's main result.
Due to its expanded scope (e.g., regular rings and C*-algebras), its formulation appears somewhat technical.

\begin{theorem}\label{T:NonV(R)reg}
Let~$\Bbbk$ be a field, let~$\gq$ be an infinite regular cardinal, and let~$\cR$ be a subcategory of the category of all unital rings and unital ring homomorphisms satisfying the following assumptions:
\begin{enumeratei}
\item\label{VSemivecR}
all objects and arrows of the diagram~$\vec{R}_{\Bbbk}$ \pup{introduced in Subsection~\textup{\ref{Su:IdcRegRings}}} belong to~$\cR$;

\item\label{VSemicR}
every object of~$\cR$ is V-semiprimitive;

\item\label{VSemiProd}
$\cR$ is closed under products \pup{within the category of all unital rings and unital ring homomorphisms};

\item\label{VSemiColim}
$\cR$ has all colimits indexed by regular cardinals $\geq\gq$, and the functor~$\rV$ preserves those colimits.

\end{enumeratei}
Denote by~$\cC$ the class of all \cm{s} that are the image of~$\rV(R)$ under a pre-\Vhom, for some V-semiprimitive ring~$R$.
Then for every regular cardinal $\gl\geq\gq+\aleph_1$\,, there exists a functor $\gD$, from~$\Powi(\gl^{+2})$ to the category of all \cm{s} with~$\scL_{\infty\gl}$-el\-e\-men\-tary embeddings, satisfying the following statements:
\begin{enumerater}
\item\label{IncrUnionvNRVR}
$\gD$ is $\gl$-continuous;

\item\label{smallDxirepvNRVR}
for every $\gl^+$-small subset~$X$ of~$\gl^{+2}$, $\gD(X)$ belongs to~$\rV(\cR)$;

\item\label{DgqnotrepvNRVR}
$\gD(\gl^{+2})$ does not belong to~$\cC$.
\end{enumerater}
In particular, the pair $(\rV(\cR),\cC)$ is anti-el\-e\-men\-tary.

\end{theorem}

\begin{proof}
Denote by~$\Phi$ the restriction of the functor~$\rV$ from~$\Ring$ to~$\cR$.
It follows from Assumption~\eqref{VSemiProd} and Lemma~\ref{L:vecRVComm} that the diagram~$\vec{R}_{\Bbbk}$ is $\Phi$-commutative.
Moreover, it follows from Assumption~\eqref{VSemicR} that~$\rV(\cR)$ is contained in the range of~$\Psi$, thus in~$\cC$.

Using Assumption~\eqref{VSemivecR} and setting $\gk\eqdef\gl^{+2}$, we obtain, as at the beginning of the proof of Theorem~\ref{T:NotCev}, a functor $\gC\colon\Powi(\gk)\to\nobreak\cS$, $U\mapsto\xF(P\seq{U})\otlF\vec{R}_{\Bbbk}$\,, and this functor is $\gl$-continuous.

Then we need to prove the analogue of Claim~\ref{Cl:gCglRep} of Theorem~\ref{T:NonPermlgrp}, which states that for every $\gl^+$-small subset~$X$ of~$\gk$, the lattice~$\gC(X)$ belongs to the range of~$\Phi$.
The proof goes the same way, using the Boosting Lemma (viz. Lemma~\ref{L:Boosting}).
A crucial observation is that Assumption~\eqref{VSemiColim} ensures that all assumptions required for applying the Boosting Lemma are satisfied.

The remainder our proof follows the lines of the one of Theorem~\ref{T:NonPermlgrp}, with Lemmas~\ref{L:LS4V(R)}--\ref{L:noPsiB2PhiRBbbk} used instead of Lemmas~\ref{L:LS4lgrps}--\ref{L:noPsiB2PhiC} and with the functor $\gC\colon\Powi(\gk)\to\nobreak\cS$,\linebreak $U\mapsto\xF(P\seq{U})\otlF\vec{R}_{\Bbbk}$\,.
The property (PROJ$(\gl)$) follows from the argument of the part of the proof of Wehrung \cite[Lemma~13.1]{VLiftDefect} establishing the property denoted there by (PROJ$(\Phi,\mathbf{CMon}^{\Rightarrow})$).
\end{proof}

Let us now present two applications of Theorem~\ref{T:NonV(R)reg}, obtained by specializing the category~$\cR$, both for $\gq:=\go$.

Defining~$\cR$ as the category~$\vNR$ of all unital von Neumann regular rings with unital ring homomorphisms, we obtain the following.

\begin{corollary}\label{C:NonVvNR}
The class $\rV(\vNR)$ is anti-el\-e\-men\-tary.
\end{corollary}

Now define~$\cR$ as the category~$\mathbf{RR_0}$ of all unital C*-algebras of real rank zero with unital C*-homomorphisms.
We need to verify Assumptions \eqref{VSemivecR}--\eqref{VSemiColim} of the statement of Theorem~\ref{T:NonV(R)reg}, for $\Bbbk:=\CC$.
Assumption~\eqref{VSemivecR} follows from the fact that~$\vec{R}_{\CC}$ is a diagram in~$\mathbf{RR_0}$.
Assumption~\eqref{VSemicR} follows from Proposition~\ref{P:ExofrngPsiVR}.
Assumption~\eqref{VSemiProd} is obvious.
Finally, the colimit of any direct system in~$\mathbf{RR_0}$ is the norm-completion of its colimit in the category of unital rings and unital ring homomorphisms; invoking Blackadar \cite[\S~5.1]{Black1998}, Assumption~\eqref{VSemiColim} follows. 

\begin{corollary}\label{C:NonVRR0}
The class $\rV(\mathbf{RR_0})$ is anti-el\-e\-men\-tary.
\end{corollary}

\section[Coordinatizability]{Coordinatizability by von Neumann regular rings without unit}
\label{S:4SCML}

\subsection{Some background on coordinatization and frames}\label{Su:RegCoord}

We refer the reader to Wehrung~\cite{NonCoord} for more references and detail.

A lattice~$L$ is
\begin{itemize}
\item
\emph{modular} if $x\wedge(y\vee(x\wedge z))=(x\wedge y)\vee(x\wedge z)$ whenever $x,y,z\in L$;

\item
\emph{sectionally complemented} if it has a zero (i.e., smallest) element~$0$ and for all $x\leq y$ in~$L$ there exists $z\in L$ such that $x\wedge z=0$ and $x\vee z=y$ (in abbreviation, $y=x\oplus z$).

\end{itemize}

Two elements~$a$ and~$b$ of~$L$ are \emph{perspective} if there exists $x\in L$ such that $a\oplus x=b\oplus x$.
An ideal~$I$ of a sectionally complemented modular~$L$ is \emph{neutral} if every element of~$L$ perspective to some element of~$I$ belongs to~$I$.
For a positive integer~$n$, an \emph{$(n+1)$-frame of~$L$} is a pair $\pI{\vecm{a_i}{0\leq i\leq n},\vecm{c_i}{1\leq i\leq n}}$ such that $a_0\oplus c_i=a_i\oplus c_i=a_0\oplus a_i$ whenever $1\leq i\leq n$.
The frame is \emph{large} if~$L$ is generated by~$a_0$ (equivalently, by any~$a_i$) as a neutral ideal.
 
For any von Neumann regular ring~$R$, the set $\LL(R)\eqdef\setm{xR}{x\in R}$, partially ordered by set inclusion, is a sectionally complemented modular lattice.
A lattice is \emph{coordinatizable} if it is isomorphic to~$\LL(R)$ for some von Neumann regular ring~$R$.

The argument of Wehrung \cite[Theorem 9.4]{CoordCX} shows that the class of all coordinatizable complemented modular lattices (with unit!) is not closed under $\scL_{\infty\gl}$-el\-e\-men\-tary equivalence, for any infinite regular cardinal~$\gl$.
We deal here with the more difficult case of lattices without unit, but with a large $4$-frame.

J\'onsson proved in~\cite{Jons1960} that every sectionally complemented modular lattice with a countable cofinal sequence and a large $4$-frame is coordinatizable.
The author proved, \emph{via} a counterexample of cardinality~$\aleph_1$\,, that the countable cofinal sequence assumption cannot be dispensed with (cf. Wehrung~\cite{NonCoord}).

In this section we shall go further, by proving that the class of coordinatizable lattices with a large $4$-frame is anti-el\-e\-men\-tary.
In contrast to our settings in Sections~\ref{S:Ceva}--\ref{S:V(R)}, the indexing poset of our main diagram counterexample will be the transfinite chain~$\go_1$\,, with the drawback that the proof of our main result (viz. Theorem~\ref{T:LLAntiElt}) will require a large cardinal assumption.

We define binary relations~$\utrr$\,, $\aspr$\,, and~$\utr$ on the set~$\Idp{R}$ of all idempotent elements of any ring~$R$ by
 \begin{align*}
 a\utrr b&\qquad\text{if}\qquad a=ba\,,\\
 a\aspr b&\qquad\text{if}\qquad a=ba\text{ and }b=ab\,,\\
 a\utr b&\qquad\text{if}\qquad a=ab=ba\,, 
 \end{align*}
for all $a,b\in\Idp{R}$.
In particular, $a\utrr b$ if{f} $aR\subseteq bR$ and $a\aspr b$ if{f} $aR=bR$.

The proof of J\'onsson \cite[Lemma~10.2]{Jons1960} remains valid for non-unital rings so we omit it.

\begin{lemma}\label{L:L(eRe)}
For any idempotent element~$e$ in a von Neumann regular ring~$R$, there are mutually inverse isomorphisms $\gf_e\colon\LL(R)\dnw eR\to\LL(eRe)$, $\bx\mapsto\bx\cap eRe=\bx e$ and $\psi_e\colon\LL(eRe)\to\LL(R)\dnw eR$, $\by\mapsto\by R$.
\end{lemma}

\begin{lemma}\label{L:abc=ac}
Let~$n$ be a positive integer and let~$a_1$\,, \dots, $a_n$ be idempotent elements in a ring~$R$ with $a_1\utrr\cdots\utrr a_n$\,.
Then $u\eqdef a_1\cdots a_n$ and $v\eqdef a_1a_n$ are both idempotent and $u\aspr v$.
\end{lemma}

\begin{proof}
Since $a_i=a_na_i$ and $a_1=a_ia_1$ for every~$i$, we get $u=a_nu$ and $a_1=ua_1$\,, thus
$u^2=ua_1\cdots a_n=u$ and $v^2=v$.
Further, $uv=ua_1a_n=a_1a_n=v$ and $vu=a_1a_nu=a_1u=u$.
\end{proof}

\begin{lemma}\label{L:abRdiag}
Let~$R$ and~$S$ be von Neumann regular rings, let $f\colon R\to S$ be a ring homomorphism, let $a\in\Idp{R}$ and $b\in\Idp{S}$ such that $f(a)\utrr b$.
Denote by~$\gf$ the domain-range restriction of~$\LL(f)$ from~$\LL(R)\dnw aR$ to~$\LL(S)\dnw bS$.
Then the assignment $x\mapsto f(x)b$ defines a ring homomorphism $f'\colon aRa\to bSb$ and the following diagram commutes:
 \[
 \begin{tikzcd}
 \LL(aRa) \arrow[r,"\LL(f')"']\arrow[d,"\psi_a","\cong"'] 
 & \LL(bSb) \arrow[d,"\psi_b"',"\cong"]\\
 \LL(R)\dnw aR\arrow[r,"\gf"] & \LL(S)\dnw bS
 \end{tikzcd}
\]
\end{lemma}

\begin{proof}
For any $x,y\in aRa$, $bf(y)=bf(a)f(y)=f(a)f(y)=f(y)$ thus $f'(x)f'(y)=f(x)bf(y)b=f(xy)b=f'(xy)$, so~$f'$ is a ring homomorphism.
Now let $\bx\in\LL(aRa)$.
We can write $\bx=xRa$ for some idempotent $x\utr a$, so $(\psi_b\circ\LL(f'))(\bx)=\psi_b(f(x)bSb)=f(x)bSbS=f(x)bS$.
Since $f(x)=f(x)f(x)=f(x)bf(x)\in f(x)bS$, we get $(\psi_b\circ\LL(f'))(\bx)=f(x)S=\gf(\bx R)=(\gf\circ\psi_a)(\bx)$, as required.
\end{proof}

\begin{definition}\label{D:Trunc}
Let~$P$ be a poset and let $\vec{L}=\vecm{L_p\,,\,f_p^q}{p\leq q\text{ in }P}$ be a $P$-indexed com\-mu\-ta\-tive diagram of lattices and lattice homomorphisms.
An \emph{$\vec{L}$-thread} is a family $\vec{a}=\vecm{a_p}{p\in P}\in\prod\vecm{L_p}{p\in P}$ such that $f_p^q(a_p)\leq a_q$ whenever $p\leq q$ in~$P$.
Then we define the \emph{truncation of~$\vec{L}$ at~$\vec{a}$} as the com\-mu\-ta\-tive diagram $\vec{L}\res\vec{a}\eqdef\vecm{L_p\dnw a_p\,,\,g_p^q}{p\leq q\text{ in }P}$, where for all $p\leq q$ in~$P$, $g_p^q$ denotes the domain-range restriction of~$f_p^q$ from~$L_p\dnw a_p$ to~$L_q\dnw a_q$\,.
\end{definition}

The following lemma states that for a com\-mu\-ta\-tive diagram~$\vec{R}$ of von Neumann regular rings, any truncation of~$\LL\vec{R}$ can be represented as~$\LL\vec{R}'$ for some $\LL$-com\-mu\-ta\-tive diagram~$\vec{R}'$ of von Neumann regular rings.

\begin{lemma}\label{L:LiftaRcibR}
Let~$P$ be a poset, let $\vec{R}=\vecm{R_p\,,\ f_p^q}{p\leq q\text{ in }P}$ be a $P$-indexed com\-mu\-ta\-tive diagram of von Neumann regular rings and ring homomorphisms, and let $\vecm{a_p}{p\in P}\in\prod\vecm{\Idp{R_p}}{p\in P}$ such that $f_p^q(a_p)\utrr a_q$ whenever $p\leq q$ in~$P$.
Set $\vec{\ba}\eqdef\vecm{a_pR_p}{p\in P}$.
For all $p\leq q$ in~$P$, set $R'_p\eqdef a_pR_pa_p$ and let $\vec{R}'(p,q)$ be the set of all maps of the form $f\colon R'_p\to R'_q$\,, $x\mapsto f_p^q(x)\prod_{i=0}^{n-1}f_{r_i}^{q}(a_{r_i})$ where $p=r_0\leq\cdots\leq r_n=q$ in~$P$.
Then~$\vec{R}'$ is an $\LL$-com\-mu\-ta\-tive diagram of von Neumann regular rings and ring homomorphisms, and $\vec{\psi}\eqdef\vecm{\psi_p}{p\in P}$ is a natural isomorphism from~$\LL\vec{R}'$ to~$\LL\vec{R}\dnw\vec{\ba}$.
\end{lemma}

\begin{proof}
It is obvious that~$\vec{R}'$ is a diagram of von Neumann regular rings and ring homomorphisms.
For any $p\leq q$ in~$P$, any $f\in\vec{R}'(p,q)$ is defined by a rule of the form $f(x)=f_p^q(x)\prod_{i=0}^{n-1}f_{r_i}^{q}(a_{r_i})$ where $p=r_0\leq\cdots\leq r_n=q$.
Hence, defining $g\colon R'_p\to R'_q$ \emph{via} the rule $g(x)\eqdef xf_p^q(a_p)a_q$ (whenever $x\in R_p$) and setting $u\eqdef \prod_{i=0}^{n-1}f_{r_i}^{q}(a_{r_i})$, $v\eqdef f_p^q(a_p)a_q$\,, we obtain that $f(x)=f_p^q(x)u$ and $g(x)=f_p^q(x)v$ whenever $x\in R_p$\,.
Moreover, it follows from Lemma~\ref{L:abc=ac} that~$u$ and~$v$ are idempotent with $u\aspr v$.

Now let~$I$ be a set and let $\vecm{p_i}{i\in I}$ and $\vecm{q_i}{i\in I}$ be families of elements of~$P$ such that $p_i\leq q_i$ whenever $i\in I$.
Further, let $f_i,g_i\in\vec{R}'(p_i,q_i)$ for $i\in I$, let $f\eqdef\prod\vecm{f_i}{i\in I}$ and $g\eqdef\prod\vecm{g_i}{i\in I}$.
For each $i\in I$, it follows from the paragraph above that there are idempotent elements~$u_i\aspr v_i$ of~$R_{q_i}$ such that $f_i(x)=f_{p_i}^{q_i}(x)u_i$ and $g_i(x)=f_{p_i}^{q_i}(x)v_i$ for every $x\in R_{p_i}$\,.
Hence, the elements $u\eqdef\vecm{u_i}{i\in I}$ and $v\eqdef\vecm{v_i}{i\in I}$ of $\prod\vecm{R_{q_i}}{i\in I}$ are idempotent with $u\aspr v$, and further, setting $h(x)\eqdef\vecm{f_{p_i}^{q_i}(x_i)}{i\in I}$, we obtain that $f(x)=h(x)u$ and $g(x)=h(x)v$ for every $x\in\prod\vecm{R_{p_i}}{i\in I}$.
In particular, $\LL(f)=\LL(g)$.
This completes the proof of the $\LL$-com\-mu\-ta\-tiv\-ity of~$\vec{R}'$.
The final statement of Lemma~\ref{L:LiftaRcibR}, about~$\vec{\psi}$ being a natural isomorphism, follows immediately from Lemma~\ref{L:abRdiag}.
\end{proof}

\subsection{Categorical settings for Section~\ref{S:4SCML}}
\label{Su:LLCat}
Let us define a category~$\cA$ as follows.
The objects of~$\cA$ are the structures $(R,\mu_R)$ where~$R$ is a von~Neumann regular ring (not necessarily unital) and $\mu_R=\vecm{e_{ij}^R}{0\leq i,j\leq 3}$ where the~$e_{ij}^R$ form a system of matrix units of~$R$ (i.e., $e_{ij}^Re_{kl}^R=\gd_{jk}e_{il}^R$ whenever $i,j,k,l\in\set{0,1,2,3}$, where~$\gd$ denotes Kronecker's symbol), such that, setting $e^R\eqdef\sum_{i=0}^3e_{ii}^R$\,,
 \begin{equation}\label{Eq:4/5ring}
 \forall x\in R\,,e^Rx=xe^R=0\Rightarrow
 (\exists y,z\in R)(x=ye^Rz).
  \end{equation}
We shall often omit the superscript~$R$ in~$e_{ij}^R$ and~$e^R$ when the ring~$R$ is understood.
The canonical $4$-frame of~$\LL(R)$ associated with the system~$\mu_R$ of matrix units is
 \[
 \tau_R\eqdef\pI{\vecm{e_{ii}^RR}{0\leq i\leq 3},
 \vecm{(e_{ii}^R-e_{0i}^R)R}{0\leq i\leq 3}}\,.
 \]
For objects~$(R,\mu_R)$ and~$(S,\mu_S)$ of~$\cA$, define the morphisms $f\colon(R,\mu_R)\to(S,\mu_S)$ in~$\cA$ as the ring homomorphisms $f\colon R\to S$ such that $f(e_{ij}^R)=e_{ij}^S$ whenever $0\leq i,j\leq 3$.

\begin{lemma}\label{L:cT24-frame}
Let~$R$ be an object of~$\cA$.
Then~$\tau_R$ is a large $4$-frame in~$\LL(R)$.
\end{lemma}

\begin{proof}
Since the~$e_{ij}$ form a system of matrix units in~$R$, $\tau_R$ is a $4$-frame in~$\LL(R)$.
In order to prove that it is a large $4$-frame, it is, by Wehrung \cite[Theorem~8-3.24]{STA1-8}, sufficient to verify that the two-sided ideal~$J$ of~$R$ generated by~$e$ is equal to~$R$.
Now any $x\in R$ can be written as $y+ex+xe-exe$, where $ey=ye=0$.
Since~$R$ satisfies the axiom~\eqref{Eq:4/5ring}, it follows that $y\in J$.
Moreover, $e\in J$ implies that $ex+xe-exe\in J$.
Therefore, $x\in J$.
\end{proof}

We shall consider the following categories and functors (represented in Figure~\ref{Fig:LLfunctsett}):
\begin{itemize}
\item
the above-defined category~$\cA$, of all (not necessarily unital) regular rings with ``large'' systems of $4\times4$ matrix units;

\item
$\cS$ is the category of all sectionally complemented modular lattices with $0$-lattice homomorphisms;

\item
$\cS^{\Rightarrow}$ (this section's double arrows) is the subcategory of~$\cS$ consisting of all surjective homomorphisms;

\item
$\Phi(R,\mu_R)=\LL(R)$ for every object $(R,\mu_R)$ of~$\cA$, and $\Phi(f)=\LL(f)$ for every morphism~$f$ in~$\cA$;

\item
$\cB$ is the category of all (not necessarily unital) von Neumann regular rings and ring homomorphisms;

\item
$\Psi\eqdef\LL\colon\cB\to\cS$.

\end{itemize}

\begin{figure}[htb]
 \begin{tikzcd}
 \cA \arrow[rr,"\text{forgetful}"]
 \arrow[rrd,"\Phi\text{ extends }\LL"']
 && \cB \arrow[d,"\Psi=\LL"] &&\\
 && \cS &&
 \arrow[ll,hook',"\text{inclusion}"]
 \cS^{\Rightarrow}
 \end{tikzcd}
\caption{Categories and functors for Section~\ref{S:4SCML}}
\label{Fig:LLfunctsett}
\end{figure}
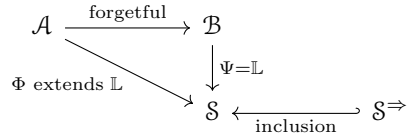

The following lemma collects mere routine facts, with straightforward proofs that we shall thus omit.
Actually, most of it got already checked in Gillibert and Wehrung \cite[Ch.~6]{Larder} and Wehrung~\cite{NonCoord};
Lemma~\ref{L:cSS*cocompl}\eqref{LLprojwit} is contained in Gillibert and Wehrung \cite[Lemma~6.2.1]{Larder}.
\goodbreak

\begin{lemma}\label{L:cSS*cocompl}\hfill
\begin{enumerater}
\item
The categories~$\cA$, $\cB$, and~$\cS$ have all directed colimits and all products.

\item
Both functors~$\Phi$ and~$\Psi$ are $\go$-continuous.

\item\label{LLprojwit}
Every double arrow $\psi\colon\Psi(C)\Rightarrow S$, where $C\in\cB$ and $S\in\cS$, has a projectability witness with respect to the functor~$\Psi$.

\end{enumerater}
\end{lemma}

\subsection{The $\LL$-commutative diagram~$\vec{R}$}
\label{Su:vecR}

We shall now consider the diagram $\vec{L}\eqdef\vecm{L_{\xi}\,,\,\be_{\xi}^{\eta}}{\xi\leq\eta<\go_1}$ of~$\cS$ introduced in Wehrung \cite[\S~7]{NonCoord} and denoted there by~$\vec{A}$, of sectionally complemented modular lattices and $0$-lattice embeddings.
Although the full construction of that diagram is quite complex, we will need here only a small part of its underlying information.
The main part of the construction, which we will not need here, is the one of an increasing $\go_1$-sequence $\vecm{A_{\xi}}{\xi<\go_1}$ of countable unital von Neumann regular rings, all sharing the same unit, and an $\utrr$-increasing sequence $\vecm{\tilde{\xi}}{\xi<\go_1}$ of idempotents $\tilde{\xi}\in A_{\xi}$\,.
We set $B_{\xi}\eqdef A_{\xi}^{5\times 5}$ (the ring of all square matrices of order~$5$ over~$A_{\xi}$), with canonical system of matrix units $\vecm{e_{ij}}{0\leq i,j\leq 4}$ and $e\eqdef\sum_{i=0}^3e_{ii}$\,.
Further, we set $b_{\xi}\eqdef e+\tilde{\xi}\cdot e_{44}$ and $L_{\xi}\eqdef\LL(B_{\xi})\dnw b_{\xi}B_{\xi}$, for any~$\xi<\go_1$\,.
We need to know that $\vecm{b_{\xi}}{\xi<\go_1}$ is an $\utrr$-increasing $\go_1$-sequence of idempotents and that all matrix units~$e_{ij}$\,, for $0\leq i,j\leq3$, belong to all $b_{\xi}B_{\xi}b_{\xi}$\,.
For $\xi<\eta<\go_1$\,, $b_{\xi}=b_{\eta}b_{\xi}\neq b_{\xi}b_{\eta}$\,.

Let~$\vec{R}$ be constructed from~$\vec{B}$ as~$\vec{R}'$ is constructed from~$\vec{R}$ in the proof of Lemma~\ref{L:LiftaRcibR}.
Hence $\vec{R}$ is an $\go_1$-indexed, $\LL$-com\-mu\-ta\-tive diagram in~$\cB$ such that $\vec{L}=\LL\vec{R}$.
In particular, whenever $\xi\leq\eta<\go_1$\,, $R_{\xi}=b_{\xi}B_{\xi}b_{\xi}$ and all morphisms in $\vec{R}(\xi,\eta)$ are finite compositions of right multiplications by elements of the form~$b_{\zeta}$ where $\xi\leq\zeta\leq\eta$.
Since all the matrix units~$e_{ij}$ belong to all $b_{\zeta}B_{\zeta}b_{\zeta}$\,, it follows that they are fixed points of every morphism in $\vec{R}(\xi,\eta)$.
It follows that, expanding each~$R_{\xi}$ by $\mu\eqdef\vecm{e_{ij}}{0\leq i,j\leq 3}$, we obtain a $\Phi$-com\-mu\-ta\-tive diagram $(\vec{R},\mu)$ in~$\cA$ such that $\Phi(\vec{R},\mu)=\LL\vec{R}=\vec{L}$.

\textbf{Now a complication arises}.
We would like to apply the Extended CLL (i.e., Lemma~\ref{L:ExtCLL}) to the $\Phi$-com\-mu\-ta\-tive diagram $(\vec{R},\mu)$, together with the Boosting Lemma (viz. Lemma~\ref{L:Boosting}).
The former step causes no difficulty, \emph{via} Gillibert \cite[Proposition~4.6]{Gill2009} (stating that every tree has ``many lifters'').
However, the latter step is definitely a problem, as the poset~$\go_1$ is not lower finite.

\textbf{In order to get around that difficulty}, we need to resort to the \emph{pseudo-retracts} introduced in Gillibert and Wehrung \cite[\S~3.6]{Larder}.
By definition, a \emph{pseudo-retraction} for a pair $(P,Q)$ of posets is a pair of isotone maps $\eps\colon P\to\Id{Q}$ and $\pi\colon Q\to P$ such that~$\pi[\eps(p)]$ is cofinal in~$P\dnw p$ whenever $p\in P$.
Then we say that~$P$ is a \emph{pseudo-retract} of~$Q$.
By Gillibert and Wehrung \cite[Lemma~3.6.6]{Larder}, every \ajs~$P$ is a pseudo-retract of some lower finite \ajs~$Q$ with the same cardinality.
We may assume, in addition, that if~$P$ has a zero, then so does~$Q$ (replace~$Q$ by~$Q'\eqdef Q\upw q_0$ for any $q_0\in\eps(0)$, set $\eps'(p)\eqdef\eps(p)\upw q_0$ and $f'\eqdef f\res_{Q'}$).
We apply the latter observation to $P:=\go_1$\,; denote by $\eps\colon\go_1\to\Id{Q}$, $\pi\colon Q\to\go_1$ the implied pseudo-retraction.
Since the diagram~$\vec{R}$ is $\LL$-com\-mu\-ta\-tive, so is the composite~$\vec{R}\pi$; in particular, $\LL\vec{R}\pi$ is a $Q$-indexed com\-mu\-ta\-tive diagram in~$\cS$.

\begin{lemma}\label{L:LRpiNotLift}
For any $Q$-indexed com\-mu\-ta\-tive diagram $\vec{R}'\in\cB^Q$, there is no natural transformation $\vec{\chi}\colon\LL\vec{R}'\Todot\LL\vec{R}\pi$ in~$\cS^{\Rightarrow}$.
\end{lemma}

\begin{proof}
Forming directed colimits over the fibers of~$\pi$ and arranging those directed colimits into a $P$-indexed diagram (cf. Gillibert and Wehrung \cite[Lemma~3.7.1]{Larder}), we obtain that any natural transformation $\vec{\chi}\colon\LL\vec{R}'\Todot\LL\vec{R}\pi$ would yield a natural transformation $\LL\vec{R}'\Todot\LL\vec{R}$.
However, we established in Wehrung \cite[Lem\-ma~7.4]{NonCoord} that the latter cannot occur.
\end{proof}

\begin{theorem}\label{T:LLAntiElt}
Assume that for every cardinal~$\gl$ there exists a cardinal~$\gk$ such that $(\gk,{<}\go,\gl)\rightarrow\aleph_1$\,.
Then the class of all coordinatizable sectionally complemented modular lattices with a large $4$-frame is anti-el\-e\-men\-tary.
\end{theorem}

\begin{proof}
Let~$\gk$ and~$\gl$ satisfy $(\gk,{<}\go,\gl)\rightarrow\aleph_1$\,.
It follows from Lemma~\ref{L:raw2leadsto}, applied to the poset~$Q$, that $(\gk,{<}\gl)\leadsto Q$.
By Lemma~\ref{L:355Larder}, it follows that for every infinite cardinal~$\gq\geq\aleph_1$\,, $Q$ has a standard $\gq^+$-lifter~$Q\seq{\gk}$ for some infinite cardinal~$\gk$.
The assumptions (WF)--(LS$(\gl)$) of the Extended CLL (i.e., Lemma~\ref{L:ExtCLL}), with $\gl=\mu=\gq^+$ and defining~$\cB^{\dagger}$ as the full subcategory of~$\cB$ consisting of all $\gl$-small von Neumann regular rings, are easily seen to be satisfied.

\setcounter{claim}{0}

\begin{claim}\label{Cl:LLBotgkNon}
The lattice $\xF(Q\seq{\gk})\otlF\vec{R}\pi$ is not coordinatizable.
\end{claim}

\begin{cproof}
Suppose otherwise, so there are a von Neumann regular ring~$B$ and a double arrow $\chi\colon\LL(B)\Rightarrow\xF(Q\seq{\gk})\otlF\vec{R}\pi$.
By Lemma~\ref{L:ExtCLL}, this implies the existence of a natural transformation $\vec{\chi}\colon\LL\vec{R}'\Todot\LL\vec{R}\pi$ in~$\cS^{\Rightarrow}$; in contradiction with Lemma~\ref{L:LRpiNotLift}.
\end{cproof}

\begin{claim}\label{Cl:LLglRepr}
The lattice $\xF(Q\seq{\gl})\otlF\vec{R}\pi$ belongs to the range of~$\Phi$; thus it is coordinatizable and it has a large $4$-frame.
\end{claim}

\begin{cproof}
Since $Q\seq{\gl}$ is the directed union of all $Q\seq{\xi}$ for $\xi<\gl$, we get
$\xF(Q\seq{\gl})=\varinjlim_{\xi<\gl}\xF(Q\seq{\xi})$, so the first part of the statement follows from the Boosting Lemma (viz. Lemma~\ref{L:Boosting}).
For every object $(R,\mu_R)$ of~$\cA$, $\tau_R$ is a large $4$-frame of~$\LL(R)$ (cf. Lemma~\ref{L:cT24-frame}) and $\Phi(R,\mu_R)=\LL(R)$ is coordinatizable.
\end{cproof}

Now the assignment $X\mapsto\xF(Q\seq{X})\otlF\vec{R}\pi$ naturally extends to a $\gl$-continuous functor from~$\Powi(\gk)$ to~$\cS$.
The desired conclusion then follows from Claims~\ref{Cl:LLBotgkNon} and~\ref{Cl:LLglRepr}.
\end{proof}

\begin{corollary}\label{C:LLAntiElt}
If there are arbitrarily large Erd\H{o}s cardinals, then the class of all coordinatizable sectionally complemented modular lattices with a large $4$-frame is anti-el\-e\-men\-tary.
\end{corollary}

By Devlin and Paris~\cite{DevPar1972} (see also Koepke \cite{Koep1984,Koep1989} for further relative consistency strength results), the existence of~$\gk$ such that $(\gk,{<}\go,\aleph_1)\rightarrow\aleph_1$ entails the existence of~$0^\sharp$ (thus it is a large cardinal axiom).
In the author's opinion, the apparent reliance of Theorem~\ref{T:LLAntiElt} on the large cardinal assumption $\forall\gl\exists\gk$ $(\gk,{<}\go,\gl)\rightarrow\aleph_1$ is accidental: at the time the main counterexample of Wehrung~\cite{NonCoord} (of a non-coordinatizable sectionally complemented modular lattice with a large $4$-frame) was constructed, the Uniformization Lemma (Lemma~\ref{L:UnifLem}) was not known, so the most natural (although not easy) way to obtain the counterexample seemed to be the application of the condensate construction to an $\go_1$-indexed diagram; and it seems very unlikely that no finitely indexed diagram could play a similar role.

\section{Loose ends}\label{S:Loose}
Anti-elementarity of a class~$\cC$ produces pairs~$(A,B)$, where $A\in\cC$, $B\notin\cC$, and~$A$ is an elementary submodel of~$B$ with respect to an arbitrarily large infinitary language.
By contrast, the (definitely condensate-like) argument of the proof of Wehrung \cite[Theorem~9.4]{CoordCX} produces~$A\notin\cC$ and~$B\in\cC$ for the given projective class~$\cC$.
This raises the question whether there are different sets of hypotheses under which a given class is not elementary with respect to arbitrarily large infinitary languages --- possibly by using variants of the condensate construction.

Another question, on which we did not elaborate in this paper, is which further lower bounds for logical complexity could be reached \emph{via}  anti-el\-e\-men\-tar\-ity of a class~$\cC$.
For example, for which logics, other than standard~$\scL_{\infty\gl}$\,, would anti-el\-e\-men\-tar\-ity entail non-rep\-re\-sentabil\-ity?
Investigate freshness and anti-el\-e\-men\-tar\-ity in relation to general categorical model theory.


\providecommand{\noopsort}[1]{}\def\cprime{$'$}
  \def\polhk#1{\setbox0=\hbox{#1}{\ooalign{\hidewidth
  \lower1.5ex\hbox{`}\hidewidth\crcr\unhbox0}}}
  \providecommand{\bysame}{\leavevmode\hbox to3em{\hrulefill}\thinspace}
\providecommand{\MR}{\relax\ifhmode\unskip\space\fi MR }
\providecommand{\MRhref}[2]{%
  \href{http://www.ams.org/mathscinet-getitem?mr=#1}{#2}
}
\providecommand{\href}[2]{#2}

\end{document}